\documentclass[11pt,twoside]{article}
\usepackage{amsmath,amssymb,amsthm}
\usepackage{enumerate}
\usepackage{cite}
\usepackage{latexsym,bm}
\usepackage{hyperref} %点击跳转参考文献，参考文献跳原文：\usepackage[backref]{hyperref}
\usepackage{bm}%点击跳转到目标方程
\usepackage{authblk}%作者，单位，邮箱宏包
\usepackage{indentfirst} %首行缩进宏包
\usepackage{extarrows}%识别长等号
\usepackage{CJK} % CJK 中文支持
\usepackage{mathrsfs}%\mathscr{}的宏包
\usepackage{graphicx} % 导入图像的宏包、单图
\usepackage{subfigure} % 导入图像的宏包、子图
\usepackage{float}%浮动格式

\allowdisplaybreaks

\allowdisplaybreaks

\theoremstyle{plain}
\newtheorem{assumption}{Assumption}
\newtheorem{theorem}{Theorem}[section]
\newtheorem{proposition}{Proposition}[section]
\newtheorem{lemma}{Lemma}[section]

\theoremstyle{definition}
\newtheorem{definition}[theorem]{Definition}

\theoremstyle{remark}
\newtheorem{remark}{Remark}[section]

\numberwithin{equation}{section}

\title{ Weak and strong averaging principle for 2D Boussinesq equations with non-Lipschitz Poisson jump noise}

\author[1]{Yangyang Shi\thanks{shiyy@chnu.edu.cn}}
\author[2]{Dong Su\thanks{dsu224466@163.com}}
\author[3]{Hui Liu\thanks{Corresponding author, ss\_liuhui@ujn.edu.cn}}
\affil[1]{School of Mathematics and Statistics, Huaibei Normal University, Huaibei 235000, P.R. China}
\affil[2]{School of Mathematics, Nanjing Audit University, Nanjing 211815, P.R. China}
\affil[3]{School of Mathematical Sciences, University of Jinan, Jinan 250022, P.R. China}
\begin{document}

\date{}

\maketitle
\thispagestyle{empty}

\begin{abstract}
In this paper, we study the averaging principle for 2D Boussinesq equations with non-Lipschitz Poisson jump noise. Precisely, we will first explore the well-posedness, regularity estimates and tightness of the vorticity variable. Then, we prove the ergodicity of the temperature variable. Next, we prove that the vorticity variable converge to the solution of the averaged equation in probability and $2p$th-mean, under different conditions, as time scale parameter $\varepsilon$ goes to zero. Finally, we present a specific case study and conduct numerical simulations to substantiate the main conclusions of this paper.

\end{abstract}

\noindent
\textbf{2010 Mathematical Subject Classification~~} 37N10, 70K65, 70K70.

%% -------------------------------------------------------------------

\section{Introduction}\label{s1}
 In this paper, we refer to the Boussinesq equations on two dimensional torus $\mathbb{T}^2=[0,2\pi]\times [0,2\pi]$ by 
 \begin{equation}\label{e1}
 \left\{
\begin{aligned}%对齐
&\frac{\partial u}{\partial t}+(u\cdot \triangledown )u=\triangle u-\triangledown p+\theta e_2,\\
&\frac{\partial\theta}{\partial t}+(u\cdot \triangledown )\theta=\triangle\theta,\\
&\triangledown\cdot u=0,\\
&u(0,x,y)=u_0(x,y),\ \theta(0,x,y)=\theta_0(x,y),
\end{aligned}
\right.
\end{equation}
with periodic boundary value condition in space. Where $u(t,x,y)=(u_1,u_2)$ represents the velocity vector, $p=p(t,x,y)$ is the scalar pressure, variable $\theta(t,x,y)$ represents the scalar temperature and vector $e_2=(0,1)$.

Since the thermal diffusivity determines the rate at which heat is transferred within a fluid, particularly when the Prandtl number (Pr) or Lewis number (Le) of the fluid is significantly less than 1, the thermal diffusivity is substantially greater than both the momentum diffusivity and solute diffusivity \cite{Zaussinger}.
 
Applying a curl operator on the first equation of (\ref{e1}),  we shall consider the following two-scale stochastic Boussinesq equations driven by Poisson jump noise
 \begin{equation}\label{e2}
 \left\{
\begin{aligned}%对齐
&\frac{\partial j^{\varepsilon}}{\partial t}+(u^{\varepsilon}\cdot \triangledown )j^{\varepsilon}=\triangle j^{\varepsilon}+f(j^{\varepsilon},\theta^{\varepsilon})+\partial_x\theta^{\varepsilon}+\int_{\left|z\right|_{H}<1}\sigma_1(j^{\varepsilon},\theta^{\varepsilon},z)\dot{\tilde{\eta} }_1(t,dz) ,\\
&\frac{\partial\theta^{\varepsilon}}{\partial t}+\frac{1}{\varepsilon}(u^{\varepsilon}\cdot \triangledown )\theta^{\varepsilon}=\frac{1}{\varepsilon}\triangle\theta^{\varepsilon}+\int_{\left|z\right|_{H}<1}\sigma_2(j^{\varepsilon},\theta^{\varepsilon},z)\dot{\tilde{\eta} }_2^{\varepsilon}(t,dz),\\
&j^{\varepsilon}(0)=j_0,\ \theta^{\varepsilon}(0)=\theta_0,
\end{aligned}
\right.
\end{equation}
where $j^{\varepsilon}:=\triangledown^{\perp }\cdot u^{\varepsilon}=\partial_x u^{\varepsilon}_2-\partial_y u^{\varepsilon}_1$ is the vorticity, small parameter $0<\varepsilon \ll 1$ describing the ratio of the time scale between $j^{\varepsilon}$ and $\theta^{\varepsilon}$. The compensated Poisson random measures $\tilde{\eta}_1(t,z)=\eta_1(t,z)-\nu_1(z)\cdot t$ and $\tilde{\eta}^{\varepsilon}_2(t,z)=\eta_2(\frac{1}{\varepsilon}t,z)-\frac{1}{\varepsilon}\nu_2(z)\cdot t$  on a measurable space $(H,\mathscr{B}(H))$, and intensity measures $\nu_i(i=1,2)$ defined over the filtered probability space $(\Omega,\mathscr{F},\left\{\mathscr{F}_t\right\}_{t\geq 0},\mathbb{P})$.

As a mathematical model describing the laws governing fluid flow and heat transport under specific approximation conditions, the Boussinesq equations finds extensive application in fluid mechanics\cite{Thomas,Ahmet}, environmental hydrology\cite{Farlin,Knight},  engineering disciplines\cite{Yang} and so on. 

The study of stochastic Boussinesq equations has attracted considerable attention over the past two decades, as these equations provide important models for understanding fluid dynamics under random perturbations. The mathematical analysis of these stochastic systems has progressed significantly in several directions. The foundational question of existence and uniqueness of solutions has been extensively investigated. Pu and Guo \cite{Pu} established global well-posedness for the stochastic 2D Boussinesq equations with partial viscosity. Du and Zhang \cite{Lihuai} obtained local and global existence of pathwise solutions with multiplicative noise, while Alonso-Or\'{a}n et al. \cite{Diego} studied well-posedness in the presence of transport noise. The 3D case was addressed by Sun et al. \cite{Jinyi}, who established global existence of mild solutions in Besov spaces. Wu and Huang \cite{Shang} further investigated well-posedness for fractional Boussinesq equations with nonlinear noise. Lin et al. \cite{Quyuan} established global existence with transport noise and small rough data. In addition, The limiting regime with vanishing viscosity has received special attention. Yamazaki \cite{Yamazaki} constructed global martingale solutions for the stochastic Boussinesq system with zero dissipation. Luo \cite{Dejun} studied the convergence of stochastic 2D inviscid Boussinesq equations with transport noise to deterministic viscous systems. 

The theory of large deviations provides insights into the probability of rare events in stochastic Boussinesq systems. Duan and Millet \cite{Millet} established large deviation principles for Boussinesq equations under random influences. Sun et al. \cite{Chengfeng} investigated rare events with fluctuating dynamical boundary conditions. Zheng and Huang \cite{Yan} later extended large deviation results to systems driven by L\'{e}vy noise. Furthermore, understanding the statistical properties of solutions over long time scales is crucial. Huang et al. \cite{Jianhua} studied asymptotic properties of 2D stochastic fractional Boussinesq equations with degenerate noise. Dai et al. \cite{Haoran} analyzed asymptotic behavior for partial dissipative systems with memory effects.

The averaging principle as a crucial tool for handling multi-scale systems and finds significant applications across multiple fields, as detailed in \cite{Bertram,Weinan,Schutter} and there references. The averaging principle for stochastic differential equations(SDEs) was first proposed by Khasminskii in \cite{Khasminskii1,Khasminskii2}, and since then, substantial research progress has been made in this field. Previously, research on the averaging principle has largely focused on abstract systems. For findings concerning the averaging principle in abstract SDEs and SPDEs, such as \cite{Cerrai2,Duan1,Golec,Di,Liu1,Pavliotis,Pei1,Guo} and references therein.

Moreover, the averaging principle has extensive applications in physical models. Majda et al. employed the principle of averaging to address stochastic climate models in \cite{Majda}, thereby extracting effective reduced equations from the original model. Arnold et al. investigated the Lorenz-Maas system via the  averagingprinciple and conducted effective numerical simulations, see \cite{Arnold}. Yeong et al. applied the averaging principle to investigate particle filters with thrust control in multiscale chaotic systems; see \cite{Yeong} for details. In addition, Gao has achieved a series of outstanding results in the averaging of stochastic fluid equations, see for example \cite{Gao1,Gao2,Gao3}.

However, to this day, research findings on multi-scale Boussinesq equations remain scarce. Zhang et al. investigated the averaging of fast waves in moist atmospheric dynamics, see \cite{Yeyu}. Du et al. considered the averaging principle for the random Boussinesq equations when the Froude number is small, refer to \cite{Lihuai1}.

In existing results, fast variables are generally required to satisfy an abstract reaction-diffusion equation, which is different from our model (\ref{e2}). The presence of convective terms imposes numerous difficulties in handling the frozen equation and in determining the properties of the coefficients in averaged equations.

The paper is structured as follows. In the subsequent section, we introduce necessary notations and delineate the assumptions pertaining to each coefficient in equation (\ref{e2}). In Section \ref{s3}, we address the existence and uniqueness of solutions to equation (\ref{e2}) under non-Lipschitz conditions. Moreover, we investigate regularity estimates for vorticity variable, thereby obtaining their tightness with respect to $\varepsilon$. Section \ref{s4} specifically examines the ergodicity of frozen equations for temperature variables.  In Section \ref{s5}, we demonstrate that, under varying condition, the vorticity variable can converges to the solution of the averaged equations both in probability and 2pth-mean, as $\varepsilon$ approaches zero. The final section features an illustrative example accompanied by numerical simulations.
%% -------------------------------------------------------------------
\section{Notations and assumptions}\label{s2}

\subsection{Notations}

For any $p\geq 1$, let $L^p(\mathbb{T}^2)$ denote the space of $p$-order integrable functions on $\mathbb{T}^2$ with norm
\begin{align*}
	\left|f\right|_{L^p(\mathbb{T}^2)}=\left(\int_{\mathbb{T}^2}\left|f(x,y)\right|^pdxdy\right)^{\frac{1}{p}}.
\end{align*}
Particularly, let 
\begin{align*}
	H:=\left\{f\in L^2(\mathbb{T}^2);\int_{\mathbb{T}^2}fdxdy=0\right\}
\end{align*}
with inner product $(\cdot,\cdot)$.

For any $\alpha>0$, let $H^{\alpha}$ be the Sobolev space of functions on $\mathbb{T}^2$ with $\alpha$-order generalized derivative in $H$, and endowed with the norm 
\begin{align*}
	\left|u\right|_{H^{\alpha}}=\left|u\right|_H+\left|(-\triangle)^{\frac{\alpha}{2}}u\right|_H.
\end{align*}
Moreover, $\left \langle \cdot,\cdot  \right \rangle $ denotes the dual pairs of $H^1$ and $H^{-1}$.

$\mathscr{B}_b(H)$ denotes the space of bounded Borel functions $\varphi (x):H \rightarrow  \mathbb{R}$, endowed with the norm
\begin{align*}
|\varphi|_0=\mathop {\sup }\limits_{x \in H}\left | \varphi(x) \right |.
\end{align*}

$C^k_0(\mathbb{T}^2)$ is used to represent the space of kth-order continuous differentiable functions, with zero bounded, up to the kth-order.

\subsection{Assumptions}

We give the assumptions for the coefficients in the equation (\ref{e2}).
\begin{assumption}\label{a1}
$f: H\times H\rightarrow H$, for any $u_1,u_2,v_1,v_2\in H$, there exists a concave continuous non-decreasing function $\kappa(\cdot)$ and some $p\geq 1$ such that
\begin{align*}
	\left|f(u_1,v_1)-f(u_2,v_2)\right|^{2p}_H\leq \kappa\left(\left|u_1-u_2\right|^{2p}_H+\left|v_1-v_2\right|^{2p}_H\right),
\end{align*}
where $\kappa (\cdot ):\mathbb{R}^{+}\rightarrow \mathbb{R}^{+},\kappa (0)=0$, and $\int _{0^{+}}\frac{du}{\kappa (u)}=\infty $. One can choose some constants $C_1,C_2>0$, such that $\kappa (u)\leqslant C_1u+C_2,\ \forall u\geqslant 0$. Hence, we have growth conditions:
\begin{align*}
	\left|f(u,v)\right|^{2p}_H\leq C_1\left(1+\left|u\right|^{2p}_H+\left|v\right|^{2p}_H\right).
\end{align*}
\end{assumption}

\begin{assumption}\label{a2}
$\sigma_1: H\times H\times H\rightarrow H$, for some $p\geq 1$, and for any $u_1,u_2,v_1,v_2\in H$, we have
\begin{align*}
\int_{\left|z\right|_H<1}\left|\sigma_1(u_1,v_1,z)-\sigma_1(u_2,v_2,z)\right|^{2p}_H\nu_1(dz)\leq \kappa\left(\left|u_1-u_2\right|^{2p}_H+\left|v_1-v_2\right|^{2p}_H\right),
\end{align*}
and the growth conditions:
\begin{align*}
\int_{\left|z\right|_H<1}\left|\sigma_1(u,v,z)\right|^{2p}_H\nu_1(dz)\leq C_1\left(1+\left|u\right|^{2p}_H+\left|v\right|^{2p}_H\right).
\end{align*}
\end{assumption}

\begin{assumption}\label{a3}
	$\sigma_2: H\times H\times H\rightarrow H$, for some $p\geq 1$, for any $u_1,u_2,v_1,v_2\in H$, there exists $L_{\sigma_2}>0$ such that 
\begin{align*}
\int_{\left|z\right|_H<1}\left|\sigma_2(u_1,v_1,z)-\sigma_2(u_2,v_2,z)\right|^{2p}_H\nu_2(dz)\leq L_{\sigma_2}\left|v_1-v_2\right|^{2p}_H+\kappa\left(\left|u_1-u_2\right|^{2p}_H\right).
\end{align*}
Therefore, we have growth conditions:
\begin{align*}
\int _{\left | z\right |_H<1}\left |\sigma _2(u,v)\right |^{2p}\nu_2(dz)\leq L_{\sigma_2}\left|v\right|^{2p}_H+C_1\left(1+\left|u\right|^{2p}_H\right).
\end{align*}
\end{assumption}

\begin{assumption}\label{a4}
$\lambda_p :=2p\lambda-M_p\left(\frac{(p-1)}{p}+\frac{L_{\sigma_2}}{p}+L_{\sigma_2}\right)>0$, where $\lambda$ is the smallest eigenvalue of $-\triangle$ in $H$ and $M_p:=2p\cdot(2p-1)\cdot 2^{2p-3}$.
\end{assumption}

\begin{remark}
\begin{enumerate}
	\item For any $x,h\in H$ and $p\geq 1$, we have
\begin{align}\label{e3}
		\left|\left|x+h\right|^{2p}_H-\left|x\right|^{2p}_H-2p\left|x\right|^{2(p-1)}_H\left(x,h\right)\right|\leq M_p\left(\left|x\right|^{2(p-1)}_H\left|h\right|^2_H+\left|h\right|^{2p}_H\right),
\end{align}
and the details can be see in \cite{Brzezniak}.
	\item When $p=1$, $\lambda_p=2\lambda-2L_{\sigma_2}$ is a common dissipative condition.
\end{enumerate}
\end{remark}

%------------------------------------------------------------------------

\section{Well-posedness, regularity estimates and tightness}\label{s3}
In this section, we discuss the well-posedness of equations (\ref{e2}), regularity estimates and the tightness of vorticity variable $j^{\varepsilon}$. 
Moreover, we define the solution of (\ref{e2}) as follow:
\begin{definition}
An $H\times H$-valued c\`{a}dl\`{a}g $\mathscr{F}_t$-adapted process $(j^{\varepsilon},\theta^{\varepsilon})$ is called a solution of (\ref{e2}), if the following conditions are satisfied
\begin{enumerate}
	\item $(j^{\varepsilon},\theta^{\varepsilon})\in L^{\infty}([0,T];H\times H)\cap L^2([0,T];H^1\times H^1)$, $\mathbb{P}$-a.s.;
	\item the following equality holds $\mathbb{P}$-a.s.:
	\begin{align*}
		&j^{\varepsilon}+\int_0^t\left(u^{\varepsilon}\cdot \triangledown\right)j^{\varepsilon}ds=j_0+\int_0^t\triangle j^{\varepsilon}ds+\int_0^tf(j^{\varepsilon},\theta^{\varepsilon})ds+\int_0^t\int_{\left|z\right|_H<1}\sigma_1(j^{\varepsilon},\theta^{\varepsilon},z)\dot{\tilde{\eta}}_1(ds,dz),\\
		&\theta^{\varepsilon}+\frac{1}{\varepsilon}\int_0^t\left(u^{\varepsilon}\cdot \triangledown\right)\theta^{\varepsilon}ds=\theta_0+\frac{1}{\varepsilon}\int_0^t\triangle\theta^{\varepsilon}ds+\int_0^t\int_{\left|z\right|_H<1}\sigma_2(j^{\varepsilon},\theta^{\varepsilon},z)\dot{\tilde{\eta}}^{\varepsilon}_2(ds,dz).
	\end{align*}
\end{enumerate}
\end{definition}
Now, we declare the following result:
\begin{theorem}\label{th1}
Suppose that Assumptions \ref{a1} - \ref{a4} hold, we have for any $(j_0,\theta_0)\in H\times H$, 
\begin{enumerate}
	\item (\ref{e2}) has a unique solution $(j^{\varepsilon},\theta^{\varepsilon})$;
	\item for any $T>0$, and for some $p\geq 1$, we have
\begin{align}\label{e4}
\mathbb{E}\sup_{t\in [0,T]}\left | j^{\varepsilon} \right |^{2p}_H+\mathbb{E}\int_{0}^{T}\left | j^{\varepsilon} \right |^{2(p-1)}_H\left|\triangledown j^{\varepsilon}\right|^2_Hds\leqslant C_{T}\left ( 1+\left | j_0 \right |^{2p}_H+\left | \theta_0\right |^{2p}_H\right ),
\end{align}
and
\begin{align}\label{e5}
\sup_{t\in [0,T]}\mathbb{E}\left | \theta^{\varepsilon} \right |^{2p}_H+\mathbb{E}\int_{0}^{T}\left | \theta^{\varepsilon} \right |^{2(p-1)}_H\left|\triangledown \theta^{\varepsilon}\right|^2_Hds\leqslant C_{T}\left ( 1+\left | j_0 \right |^{2p}_H+\left | \theta_0\right |^{2p}_H\right ).
\end{align}
\end{enumerate}
\end{theorem}

\begin{proof}
There have been numerous results concerning the well-posedness of  Boussinesq equations. Here, we generalize the well-posedness to the non-Lipschitz noise case. The details are given in the Appendix.

Let's verify (\ref{e4}) and (\ref{e5}). Note that 
\begin{align*}
\left((u\cdot \triangledown)v,w\right)=-\left((u\cdot \triangledown)w,v\right),\ and\  \left((u\cdot \triangledown)v,v\right)=0,\quad for\ any\ u,v,w\in H^1.
\end{align*}
According to It\^{o}'s formula with L\'{e}vy noise (cf. \cite{Justin}), we have

\begin{align*}
	\left|\theta^{\varepsilon}\right|^{2p}_H=&\left|\theta_0\right|^{2p}_H+\frac{2p}{\varepsilon}\int_0^t\left|\theta^{\varepsilon}\right|^{2(p-1)}_H\left \langle \triangle \theta^{\varepsilon},\theta^{\varepsilon} \right \rangle ds+2p\int_0^t\int_{\left|z\right|_H<1}\left|\theta^{\varepsilon}\right|^{2(p-1)}_H\left(\theta^{\varepsilon},\sigma_2(j^{\varepsilon},\theta^{\varepsilon},z)\right)\tilde{\eta}^{\varepsilon}_2(ds,dz)\\
	&+\int_0^t\int_{\left|z\right|_H<1}\left [\left|\theta^{\varepsilon}+\sigma_2(j^{\varepsilon},\theta^{\varepsilon},z)\right|^{2p}_H-\left|\theta^{\varepsilon}\right|^{2p}_H-2p\left|\theta^{\varepsilon}\right|^{2(p-1)}_H\left(\theta^{\varepsilon},\sigma_2(j^{\varepsilon},\theta^{\varepsilon},z)\right)  \right ] \eta_2(\frac{1}{\varepsilon}ds,dz),
\end{align*}
therefore
\begin{align*}
	\mathbb{E}\left|\theta^{\varepsilon}\right|^{2p}_H=&\left|\theta_0\right|^{2p}_H+\frac{2p}{\varepsilon}\mathbb{E}\int_0^t\left|\theta^{\varepsilon}\right|^{2(p-1)}_H\left \langle \triangle \theta^{\varepsilon},\theta^{\varepsilon} \right \rangle ds+\frac{1}{\varepsilon}\mathbb{E}\int_0^t\int_{\left|z\right|_H<1}\left [\left|\theta^{\varepsilon}+\sigma_2(j^{\varepsilon},\theta^{\varepsilon},z)\right|^{2p}_H\right.\\
	&\left.-\left|\theta^{\varepsilon}\right|^{2p}_H-2p\left|\theta^{\varepsilon}\right|^{2(p-1)}_H\left(\theta^{\varepsilon},\sigma_2(j^{\varepsilon},\theta^{\varepsilon},z)\right)  \right ] \nu_2(dz)ds.
\end{align*}
According to Assumptions \ref{a3}, \ref{a4} and (\ref{e3}), we have
\begin{align}\label{e5.5}
	\nonumber\frac{d}{dt}\mathbb{E}\left|\theta^{\varepsilon}\right|^{2p}_H&+\frac{2p\gamma}{\varepsilon}\mathbb{E}\left|\theta^{\varepsilon}\right|^{2(p-1)}_H\left|\triangledown\theta^{\varepsilon}\right|^2_H\leq -\frac{2p(1-\gamma)}{\varepsilon}\mathbb{E}\left|\theta^{\varepsilon}\right|^{2(p-1)}_H\left|\triangledown\theta^{\varepsilon}\right|^2_H\\
	\nonumber&+\frac{M_p}{\varepsilon}\mathbb{E}\int_{\left|z\right|_H<1}\left(\left|\theta^{\varepsilon}\right|^{2(p-1)}_H\left|\sigma_2(j^{\varepsilon},\theta^{\varepsilon},z)\right|^2_H+\left|\sigma_2(j^{\varepsilon},\theta^{\varepsilon},z)\right|^{2p}_H\right)\nu_2(dz)\\
	\nonumber\leq & -\frac{2p(1-\gamma)\lambda}{\varepsilon}\mathbb{E}\left|\theta^{\varepsilon}\right|^{2p}_H+\frac{M_p}{\varepsilon}\mathbb{E}\int_{\left|z\right|_H<1}\left(\frac{p-1}{p}\left|\theta^{\varepsilon}\right|^{2p}_H+\frac{1}{p}\left|\sigma_2(j^{\varepsilon},\theta^{\varepsilon},z)\right|^{2p}_H\right)\nu_2(dz)\\
	\nonumber&+\frac{M_pL_{\sigma_2}}{\varepsilon}\mathbb{E}\left|\theta^{\varepsilon}\right|^{2p}_H+\frac{M_pC_1}{\varepsilon}\left(1+\mathbb{E}\left|j^{\varepsilon}\right|^{2p}_H\right)\\
	\leq &-\frac{\lambda_{p,\gamma}}{\varepsilon}\mathbb{E}\left|\theta^{\varepsilon}\right|^{2p}_H+\frac{C}{\varepsilon}\left(1+\mathbb{E}\left|j^{\varepsilon}\right|^{2p}_H\right),
\end{align}
where $\lambda_{p,\gamma}:=2p(1-\gamma)\lambda-M_p\left(\frac{(p-1)}{p}+\frac{L_{\sigma_2}}{p}+L_{\sigma_2}\right)>0$, for $\gamma$ small enough. Using Gronwall's inequality, we have
\begin{align}\label{e6}
\nonumber\mathbb{E}\left|\theta^{\varepsilon}\right|^{2p}_H+\frac{2p\gamma}{\varepsilon}\mathbb{E}\int_0^t\left|\theta^{\varepsilon}\right|^{2(p-1)}_H\left|\triangledown\theta^{\varepsilon}\right|^2_Hds\leq &	\left|\theta_0\right|^{2p}_He^{-\frac{\lambda_{p,\gamma}}{\varepsilon}t}+\frac{C}{\varepsilon}\int_0^te^{-\frac{\lambda_{p,\gamma}}{\varepsilon}(t-s)}\left(1+\mathbb{E}\left|j^{\varepsilon}\right|^{2p}_H\right)ds\\
\leq & C\left(1+\left|\theta_0\right|^{2p}_H\right)+\int_0^te^{-\frac{\lambda_{p,\gamma}}{\varepsilon}(t-s)}\mathbb{E}\left|j^{\varepsilon}\right|^{2p}_Hds.
\end{align}

In addition, by It\^{o}'s formula, we have
\begin{align*}
	\left|j^{\varepsilon}\right|^{2p}_H=&\left|j_0\right|^{2p}_H+2p\int_0^t\left|j^{\varepsilon}\right|^{2(p-1)}_H\left \langle \triangle j^{\varepsilon},j^{\varepsilon} \right \rangle ds+2p\int_0^t\left|j^{\varepsilon}\right|^{2(p-1)}_H\left(f(j^{\varepsilon},\theta^{\varepsilon}),j^{\varepsilon}\right)ds\\
	&+2p\int_0^t\left|j^{\varepsilon}\right|^{2(p-1)}_H\left(\partial_x\theta^{\varepsilon},j^{\varepsilon}\right)ds+2p\int_0^t\int_{\left|z\right|_H<1}\left|j^{\varepsilon}\right|^{2(p-1)}_H\left(j^{\varepsilon},\sigma_1(j^{\varepsilon},\theta^{\varepsilon},z)\right)\tilde{\eta}_1(ds,dz) \\
	&+\int_0^t\int_{\left|z\right|_H<1}\left [ \left|j^{\varepsilon}+\sigma_1(j^{\varepsilon},\theta^{\varepsilon},z)\right|^{2p}_H-\left|j^{\varepsilon}\right|^{2p}_H-2p\left|j^{\varepsilon}\right|^{2(p-1)}_H\left(j^{\varepsilon},\sigma_1(j^{\varepsilon},\theta^{\varepsilon},z)\right) \right ] \eta_1(ds,dz).
\end{align*}
Therefore
\begin{align*}
	\left|j^{\varepsilon}\right|^{2p}_H&+2p\lambda\int_0^t\left|j^{\varepsilon}\right|^{2(p-1)}_H\left|\triangledown j^{\varepsilon}\right|^2_Hds\leq \left|j_0\right|^{2p}_H+2p\int_0^t\left|j^{\varepsilon}\right|^{2(p-1)}_H\left(f(j^{\varepsilon},\theta^{\varepsilon}),j^{\varepsilon}\right)ds\\
	&-2p\int_0^t\left|j^{\varepsilon}\right|^{2(p-1)}_H\left(\theta^{\varepsilon},\partial_xj^{\varepsilon}\right)ds+2p\int_0^t\int_{\left|z\right|_H<1}\left|j^{\varepsilon}\right|^{2(p-1)}_H\left(j^{\varepsilon},\sigma_1(j^{\varepsilon},\theta^{\varepsilon},z)\right)\tilde{\eta}_1(ds,dz) \\
	&+\int_0^t\int_{\left|z\right|_H<1}\left [ \left|j^{\varepsilon}+\sigma_1(j^{\varepsilon},\theta^{\varepsilon},z)\right|^{2p}_H-\left|j^{\varepsilon}\right|^{2p}_H-2p\left|j^{\varepsilon}\right|^{2(p-1)}_H\left(j^{\varepsilon},\sigma_1(j^{\varepsilon},\theta^{\varepsilon},z)\right) \right ] \eta_1(ds,dz).
\end{align*}
According to Assumptions \ref{a1}, \ref{a2} and Burkholder-Davis-Gundy (B-D-G)'s inequality, we have
\begin{align*}
	&\mathbb{E}\sup_{t\in[0,T]}\left|j^{\varepsilon}\right|^{2p}_H+2p\lambda\mathbb{E}\int_0^T\left|j^{\varepsilon}\right|^{2(p-1)}_H\left|\triangledown j^{\varepsilon}\right|^2_Hds\\
	\leq& \left|j_0\right|^{2p}_H+\gamma\mathbb{E}\int_0^T\left|j^{\varepsilon}\right|^{2(p-1)}_H\left|\triangledown j^{\varepsilon}\right|^2_Hds+C_{\gamma}\int_0^T\left(1+\mathbb{E}\left|j^{\varepsilon}\right|^{2p}_H+\mathbb{E}\left|\theta^{\varepsilon}\right|^{2p}_H\right)ds\\
	&+C_p\mathbb{E}\left(\int_0^T\int_{\left|z\right|_H<1}\left|j^{\varepsilon}\right|^{4p-4}_H\left|\left(j,\sigma_1(j^{\varepsilon},\theta^{\varepsilon},z)\right)\right|^2_H\eta_1(ds,dz)\right)^{\frac{1}{2}}.\\
\end{align*}
Note that 
\begin{align*}
	&\mathbb{E}\left(\int_0^T\int_{\left|z\right|_H<1}\left|j^{\varepsilon}\right|^{4p-4}_H\left|\left(j,\sigma_1(j^{\varepsilon},\theta^{\varepsilon},z)\right)\right|^2_H\eta_1(ds,dz)\right)^{\frac{1}{2}}\\
	\leq &\mathbb{E}\left(\int_0^T\int_{\left|z\right|_H<1}\left|j^{\varepsilon}\right|^{4p-2}_H\left|\sigma_1(j^{\varepsilon},\theta^{\varepsilon},z)\right|^2_H\eta_1(ds,dz)\right)^{\frac{1}{2}}\\
	\leq &\mathbb{E}\left(\sup_{t\in[0,T]}\left|j^{\varepsilon}\right|^{2p}_H\int_0^T\int_{\left|z\right|_H<1}\left|j^{\varepsilon}\right|^{2p-2}_H\left|\sigma_1(j^{\varepsilon},\theta^{\varepsilon},z)\right|^2_H\eta_1(ds,dz)\right)^{\frac{1}{2}}\\
	\leq &\gamma\mathbb{E}\sup_{t\in[0,T]}\left|j^{\varepsilon}\right|^{2p}_H+C_{\gamma}\int_0^T\left(1+\mathbb{E}\left|j^{\varepsilon}\right|^{2p}_H+\mathbb{E}\left|\theta^{\varepsilon}\right|^{2p}_H\right)ds.
\end{align*}
Choose $\gamma$ small enough, we obtain that
\begin{align*}
	\mathbb{E}\sup_{t\in[0,T]}\left|j^{\varepsilon}\right|^{2p}_H+C_p\mathbb{E}\int_0^T\left|j^{\varepsilon}\right|^{2(p-1)}_H\left|\triangledown j^{\varepsilon}\right|^2_Hds\leq \left|j_0\right|^{2p}_H+C_p\int_0^T\left(1+\mathbb{E}\left|j^{\varepsilon}\right|^{2p}_H+\mathbb{E}\left|\theta^{\varepsilon}\right|^{2p}_H\right)ds.
\end{align*}
Combining (\ref{e6}), we have
\begin{align*}
	&\mathbb{E}\sup_{t\in[0,T]}\left|j^{\varepsilon}\right|^{2p}_H+C_p\mathbb{E}\int_0^T\left|j^{\varepsilon}\right|^{2(p-1)}_H\left|\triangledown j^{\varepsilon}\right|^2_Hds\\
	\leq &C_T\left(1+\left|j_0\right|^{2p}_H+\left|\theta_0\right|^{2p}_H\right)+C_p\int_0^T\mathbb{E}\left|j^{\varepsilon}\right|^{2p}_Hds+\frac{C_p}{\varepsilon}\int_0^T\int_0^se^{-\frac{\lambda_{p,\gamma}}{\varepsilon}(s-\tau)}\mathbb{E}\left|j^{\varepsilon}\right|^{2p}_Hd\tau ds\\
	=& C_T\left(1+\left|j_0\right|^{2p}_H+\left|\theta_0\right|^{2p}_H\right)+C_p\int_0^T\mathbb{E}\left|j^{\varepsilon}\right|^{2p}_Hds+\frac{C_p}{\varepsilon}\int_0^T\int_{\tau}^Te^{-\frac{\lambda_{p,\gamma}}{\varepsilon}(s-\tau)}ds\mathbb{E}\left|j^{\varepsilon}\right|^{2p}_Hd\tau \\
	\leq &C_T\left(1+\left|j_0\right|^{2p}_H+\left|\theta_0\right|^{2p}_H\right)+C_p\int_0^T\mathbb{E}\sup_{r\in[0,s]}\left|j^{\varepsilon}\right|^{2p}_Hds.
\end{align*}
Using Gronwall's inequality, we obtain that
\begin{align*}
	\mathbb{E}\sup_{t\in[0,T]}\left|j^{\varepsilon}\right|^{2p}_H+C_p\mathbb{E}\int_0^T\left|j^{\varepsilon}\right|^{2(p-1)}_H\left|\triangledown j^{\varepsilon}\right|^2_Hds\leq C_T\left(1+\left|j_0\right|^{2p}_H+\left|\theta_0\right|^{2p}_H\right).
\end{align*}
And so
\begin{align*}
\mathbb{E}\left|\theta^{\varepsilon}\right|^{2p}_H+\frac{2p\gamma}{\varepsilon}\mathbb{E}\int_0^t\left|\theta^{\varepsilon}\right|^{2(p-1)}_H\left|\triangledown\theta^{\varepsilon}\right|^2_Hds\leq C_T\left(1+\left|j_0\right|^{2p}_H+\left|\theta_0\right|^{2p}_H\right).
\end{align*}
\end{proof}
\begin{remark}
In the previous theorem, we have only proved uniform bounds with respect to $\varepsilon$ for $\sup_{t\in [0,T]}\mathbb{E}\left | \theta^{\varepsilon} \right |^{2p}$ and not for $\mathbb{E}\sup_{t\in [0,T]}\left | \theta^{\varepsilon} \right |^{2p}$. In fact, we can only able to prove that
\begin{align}\label{xe11}
\mathbb{E}\sup_{t\in [0,T]}\left | \theta^{\varepsilon} \right |^{2p}\leqslant C_T\left (  1+\left | j_0 \right |^{2p}_H+\left | \theta_0\right |^{2p}_H+\varepsilon ^{-1}\right ).
\end{align}
\end{remark}

Next, we consider the regularity estimates of the vorticity variable $j^{\varepsilon}$. We declare the following theorem:
\begin{theorem}\label{th2}
For any $T>0$, $M>0$ and $\left ( j_0,\theta_0\right )\in H^1\times H$, we have
\begin{align}\label{e7}
\mathbb{E}\sup_{t\in [0,\tau_M \wedge T]}\left | j^{\varepsilon}\right |^2_{H^1}+\mathbb{E}\int_{0}^{\tau_M \wedge T}\left | j^{\varepsilon}\right |^2_{H^2}ds\leqslant C_{T,M}\left (  1+\left | j_0 \right |^{2}_{H^1}+\left | \theta_0\right |^{2}_H\right ),
\end{align}
where stopping time $\tau_M:=\inf\left\{t\geq 0; \left|j^{\varepsilon}\right|_H>M\right\}$.
\end{theorem}
\begin{proof}
We take the inner product of the first equation in (\ref{e2}) with $-\triangle j^{\varepsilon}$, yielding
\begin{align*}
	\frac{1}{2}\frac{d}{dt}\left|\triangledown j^{\varepsilon}\right|^2_H+\left|\triangle j^{\varepsilon}\right|^2_H=&\left((u^{\varepsilon}\cdot\triangledown)j^{\varepsilon},\triangle j^{\varepsilon}\right)+\left(f(j^{\varepsilon},\theta^{\varepsilon}),-\triangle j^{\varepsilon}\right)+\left(\partial_x\theta^{\varepsilon},-\triangle j^{\varepsilon}\right)\\
	&+\int_{\left|z\right|_H<1}\left(-\triangle j^{\varepsilon},\sigma_1(j^{\varepsilon},\theta^{\varepsilon},z)\dot{\tilde{\eta} }_1 (t,dz)\right).
\end{align*}
Note that 
\begin{align*}
	\left((u^{\varepsilon}\cdot\triangledown)j^{\varepsilon},\triangle j^{\varepsilon}\right)=-\left(\triangledown (u^{\varepsilon}\cdot \triangledown)j^{\varepsilon},\triangledown j^{\varepsilon}\right)=-\left((\triangledown u^{\varepsilon}\cdot \triangledown)j^{\varepsilon},\triangledown j^{\varepsilon}\right),
\end{align*}
by Ladyzhenskaya's inequality, we have
\begin{align*}
	\left|\left((u^{\varepsilon}\cdot\triangledown)j^{\varepsilon},\triangle j^{\varepsilon}\right)\right|\leq \left|\triangledown u^{\varepsilon}\right|_{H}\left|\triangledown j^{\varepsilon}\right|^2_{L^4}\leq \left|j^{\varepsilon}\right|_{H}\left|\triangledown j^{\varepsilon}\right|_{H}\left|\triangle j^{\varepsilon}\right|_{H}\leq C_{\gamma}\left|j^{\varepsilon}\right|^2_H\left|\triangledown j^{\varepsilon}\right|^2_H+\gamma\left|\triangle j^{\varepsilon}\right|^2_H.
\end{align*}
Hence
\begin{align*}
	\frac{1}{2}\frac{d}{dt}\left|\triangledown j^{\varepsilon}\right|^2_H+\left|\triangle j^{\varepsilon}\right|^2_H\leq& C_{\gamma}\left|j^{\varepsilon}\right|^2_H\left|\triangledown j^{\varepsilon}\right|^2_H+C_{\gamma}\left(1+\left|j^{\varepsilon}\right|^2_H+\left|\theta^{\varepsilon}\right|^2_{H^1}\right)+4\gamma\left|\triangle j^{\varepsilon}\right|^2_H\\
	&+C_{\gamma}\int_{\left|z\right|_H<1}\left|\sigma_1(j^{\varepsilon},\theta^{\varepsilon},z)\right|^2_H\dot{\tilde{\eta} }_1 (t,dz).
\end{align*}
Choose $\gamma$ small enough, and by (\ref{e4}), (\ref{e5}), we have
\begin{align*}
	&\mathbb{E}\sup_{t\in[0,\tau_M \wedge T]}\left|\triangledown j^{\varepsilon}\right|^2_H+\mathbb{E}\int_0^{\tau_M \wedge T}\left|\triangle j^{\varepsilon}\right|^2_Hds\\
	\leq& C\left|\triangledown j_0\right|^2_H+C_{M}\mathbb{E}\int_0^{\tau_M \wedge T}\left|\triangledown j^{\varepsilon}\right|^2_Hds+C\mathbb{E}\int_0^{\tau_M \wedge T}\left(1+\left|j^{\varepsilon}\right|^2_H+\left|\theta^{\varepsilon}\right|^2_{H^1}\right)ds\\
	\leq & C_T\left(1+\left|j_0\right|^2_{H^1}+\left|\theta_0\right|^2_{H}\right)+C_{M}\mathbb{E}\int_0^{\tau_M \wedge T}\left|\triangledown j^{\varepsilon}\right|^2_Hds.
\end{align*}
The Gronwall's inequality yields that
\begin{align*}
	\mathbb{E}\sup_{t\in[0,\tau_M \wedge T]}\left|\triangledown j^{\varepsilon}\right|^2_H+\mathbb{E}\int_0^{\tau_M \wedge T}\left|\triangle j^{\varepsilon}\right|^2_Hds\leq C_{T,M}\left(1+\left|j_0\right|^2_{H^1}+\left|\theta_0\right|^2_{H}\right).
\end{align*}
\end{proof}

In the next, we will show the tightness of the vorticity variable $j^{\varepsilon}$. For this we need the following Lemma:
\begin{lemma}\label{l1}
For any $0\leq t-\delta< t \leq \tau _M \wedge T $ and $\left ( j_0,\theta_0\right )\in H^1\times H$, we have
\begin{align}\label{e8}
\mathbb{E}\left | j^{\varepsilon}_{t}-j^{\varepsilon}_{t-\delta}\right |^2_{H}\leqslant C_{T,M}\delta^{\frac{1}{2}}.
\end{align}
In addition, for any $(j_0,\theta_0)\in H\times H$, and for some $p\geq 1$, we have
\begin{align}\label{e8.5.5}
	\mathbb{E}\int_0^{\tau _M \wedge T}\left | j^{\varepsilon}_{t}-j^{\varepsilon}_{t-\delta}\right |^{2p}_{H}dt\leqslant C_{T,M}\delta^{\frac{1}{2}}.
\end{align}
\end{lemma}
\begin{proof}
For any $0\leq t-\delta<t\leq \tau_{M}\wedge T$, we have
\begin{align*}
	j^{\varepsilon}_{t}-j^{\varepsilon}_{t-\delta}=&-\int_{t-\delta}^{t}\left(u^{\varepsilon}_s\cdot \triangledown\right)j^{\varepsilon}_sds+\int_{t-\delta}^{t}\left(\triangle j^{\varepsilon}_s+f(j^{\varepsilon}_s,\theta^{\varepsilon}_s)+\partial_x\theta^{\varepsilon}_s\right)ds\\
	&+\int_{t-\delta}^{t}\int_{\left|z\right|_H<1}\sigma_1(j^{\varepsilon}_s,\theta^{\varepsilon}_s,z)\tilde{\eta}_1(ds,dz).
\end{align*}
By It\^{o}'s formula, we get
\begin{align}\label{e8.5.5.5}
\nonumber	\left|j^{\varepsilon}_{t}-j^{\varepsilon}_{t-\delta}\right|^2_H=&-2\int_{t-\delta}^{t}\left(\left(u^{\varepsilon}_s\cdot\triangledown\right)j^{\varepsilon}_s,j^{\varepsilon}_s-j^{\varepsilon}_{t-\delta}\right)ds+2\int_{t-\delta}^{t}\left \langle \triangle j^{\varepsilon}_s, j^{\varepsilon}_s-j^{\varepsilon}_{t-\delta}\right \rangle ds\\
\nonumber	&+2\int_{t-\delta}^{t}\left(f(j^{\varepsilon}_s,\theta^{\varepsilon}_s),j^{\varepsilon}_s-j^{\varepsilon}_{t-\delta}\right)ds+2\int_{t-\delta}^{t}\left(\partial_x \theta^{\varepsilon}_s,j^{\varepsilon}_s-j^{\varepsilon}_{t-\delta}\right)ds\\
\nonumber	&+2\int_{t-\delta}^{t}\int_{\left|z\right|_H<1}\left(\sigma_1(j^{\varepsilon}_s,\theta^{\varepsilon}_s,z),j^{\varepsilon}_s-j^{\varepsilon}_{t-\delta}\right)\tilde{\eta}_1(ds,dz)\\
	&+\int_{t-\delta}^{t}\int_{\left|z\right|_H<1}\left|\sigma_1(j^{\varepsilon}_s,\theta^{\varepsilon}_s,z)\right|^2_H\eta_1(ds,dz).
\end{align}
Hence
\begin{align*}
	&\mathbb{E}\left|j^{\varepsilon}_{t}-j^{\varepsilon}_{t-\delta}\right|^2_H+2\mathbb{E}\int_{t-\delta}^{t}\left|\triangledown j^{\varepsilon}_s\right|^2_Hds\\
	=&\left\{2\mathbb{E}\int_{t-\delta}^{t}\left(\left(u^{\varepsilon}_s\cdot\triangledown\right)j^{\varepsilon}_s,j^{\varepsilon}_{t-\delta}\right)ds\right\}_{I_1}+\left\{2\mathbb{E}\int_{t-\delta}^{t}\left(\triangledown j^{\varepsilon}_s,\triangledown j^{\varepsilon}_{t-\delta}\right)ds\right\}_{I_2}\\
	&+\left\{2\mathbb{E}\int_{t-\delta}^{t}\left(f(j^{\varepsilon}_s,\theta^{\varepsilon}_s),j^{\varepsilon}_s-j^{\varepsilon}_{t-\delta}\right)ds\right\}_{I_3}-\left\{2\mathbb{E}\int_{t-\delta}^{t}\left(\theta^{\varepsilon}_s,\partial_x j^{\varepsilon}_s-\partial_x j^{\varepsilon}_{t-\delta}\right)ds\right\}_{I_4}\\
	&+\left\{\mathbb{E}\int_{t-\delta}^{t}\int_{\left|z\right|_H<1}\left|\sigma_1(j^{\varepsilon}_s,\theta^{\varepsilon}_s,z)\right|^2_H\eta_1(ds,dz)\right\}_{I_5}.
\end{align*}
For $I_1$, according to the following useful inequality(cf. \cite{Martin})
\begin{align}\label{e8.5}
	\left|\left((u\cdot \triangledown)v,w\right)\right|\leq C\left|u\right|_{H^{s_1}}\left|v\right|_{H^{1+s_2}}\left|w\right|_{H^{s_3}},
\end{align}
where $s=(s_1,s_2,s_3)\in \mathbb{R}^3_+$, $\sum s_i\geq 1$ and $s\neq (1,0,0),(0,1,0),(0,0,1)$. We have
\begin{align*}
	\left|I_1\right|\leq& C\mathbb{E}\int_{t-\delta}^{t}\left|u^{\varepsilon}_s\right|_{H^{1}}\left|j^{\varepsilon}_s\right|_{H^{1}}\left|j^{\varepsilon}_{t-\delta}\right|_{H^{1}}ds\leq C\mathbb{E}\int_{t-\delta}^{t}\left|j^{\varepsilon}_s\right|_{H}\left|j^{\varepsilon}_s\right|_{H^{1}}\left|j^{\varepsilon}_{t-\delta}\right|_{H^{1}}ds\\
	\leq & C_M\left(\mathbb{E}\left(\int_{t-\delta}^{t}\left|j^{\varepsilon}_s\right|_{H^{1}}ds\right)^2\right)^{\frac{1}{2}}\cdot \left(\mathbb{E}\left|j^{\varepsilon}_{t-\delta}\right|^2_{H^1}\right)^{\frac{1}{2}}\\
	\leq & C_{T,M}\delta^{\frac{1}{2}}.
\end{align*}
For $I_2$, we have
\begin{align*}
	\left|I_2\right|\leq \left(\mathbb{E}\left(\int_{t-\delta}^{t}\left|\triangledown j^{\varepsilon}_s\right|_{H}ds\right)^2\right)^{\frac{1}{2}}\cdot \left(\mathbb{E}\left|\triangledown j^{\varepsilon}_{t-\delta}\right|^2_{H}\right)^{\frac{1}{2}}\leq C_{T,M}\delta^{\frac{1}{2}}.
\end{align*}
For $I_3$, according to Assumption \ref{a1}, it is easy to obtain that
\begin{align*}
	\left|I_3\right|\leq \int_{t-\delta}^{t}\left(1+\mathbb{E}\left|j^{\varepsilon}_s\right|^2_H+\mathbb{E}\left|\theta^{\varepsilon}_s\right|^2_H+\mathbb{E}\left|j^{\varepsilon}_{t-\delta}\right|^2_H\right)ds\leq C_T\delta.
\end{align*}
For $I_4$, we have
\begin{align*}
	\left|I_4\right|\leq \int_{t-\delta}^{t}\left(\mathbb{E}\left|\theta^{\varepsilon}_s\right|^2_H+\mathbb{E}\left|\triangledown j^{\varepsilon}_s\right|^2_H+\mathbb{E}\left|\triangledown j^{\varepsilon}_{t-\delta}\right|^2_H\right)ds\leq C_{T,M}\delta.
\end{align*}
For $I_5$, by Assumption \ref{a2}, we have
\begin{align*}
	\left|I_5\right|=\mathbb{E}\int_{t-\delta}^{t}\int_{\left|z\right|_H<1}\left|\sigma_1((j^{\varepsilon}_s,\theta^{\varepsilon}_s,z))\right|^2_H d\nu_1(dz)ds\leq C_T\delta.
\end{align*}
Therefore, (\ref{e8}) hold.

In addition, by (\ref{e4}), we have
\begin{align*}
	\mathbb{E}\int_0^{\tau_{M}\wedge T}\left|j^{\varepsilon}_t-j^{\varepsilon}_{t-\delta}\right|^{2p}_Hdt=&\mathbb{E}\int_0^{\delta}\left|j^{\varepsilon}_t-j^{\varepsilon}_{t-\delta}\right|^{2p}_Hdt+\mathbb{E}\int_{\delta}^{\tau_{M}\wedge T}\left|j^{\varepsilon}_t-j^{\varepsilon}_{t-\delta}\right|^{2p}_Hdt\\
	\leq & C_{T}\delta+\mathbb{E}\int_{\delta}^{\tau_{M}\wedge T}\left|j^{\varepsilon}_t-j^{\varepsilon}_{t-\delta}\right|^{2p}_Hdt.
\end{align*}
Using It\^{o}'s formula, we get
\begin{align*}
	&\left|j^{\varepsilon}_t-j^{\varepsilon}_{t-\delta}\right|^{2p}_H+2p\int_{t-\delta}^t\left|j^{\varepsilon}_s-j^{\varepsilon}_{t-\delta}\right|^{2p-2}_H\left(\left(u^{\varepsilon}_s\cdot\triangledown\right)j^{\varepsilon}_s,j^{\varepsilon}_s-j^{\varepsilon}_{t-\delta}\right)ds\\
	=& 2p\int_{t-\delta}^t\left|j^{\varepsilon}_s-j^{\varepsilon}_{t-\delta}\right|^{2p-2}_H\left \langle \triangle j^{\varepsilon}_s, j^{\varepsilon}_s-j^{\varepsilon}_{t-\delta}\right \rangle ds+2p\int_{t-\delta}^t\left|j^{\varepsilon}_s-j^{\varepsilon}_{t-\delta}\right|^{2p-2}_H\left(f(j^{\varepsilon}_s,\theta^{\varepsilon}_s),j^{\varepsilon}_s-j^{\varepsilon}_{t-\delta}\right)ds\\
	&+2p\int_{t-\delta}^t\left|j^{\varepsilon}_s-j^{\varepsilon}_{t-\delta}\right|^{2p-2}_H\left(\partial_x\theta^{\varepsilon}_s,j^{\varepsilon}_s-j^{\varepsilon}_{t-\delta}\right)ds\\
	&+2p\int_{t-\delta}^t\int_{\left|z\right|_H<1}\left|j^{\varepsilon}_s-j^{\varepsilon}_{t-\delta}\right|^{2p-2}_H\left(\sigma_1(j^{\varepsilon}_s,\theta^{\varepsilon}_s,z),j^{\varepsilon}_s-j^{\varepsilon}_{t-\delta}\right)\tilde{\eta}_1(ds,dz)\\
	&+\int_{t-\delta}^t\int_{\left|z\right|_H<1}\left[\left|\left(j^{\varepsilon}_s-j^{\varepsilon}_{t-\delta}\right)+\sigma_1(j^{\varepsilon}_s,\theta^{\varepsilon}_s,z)\right|^{2p}_H-\left|j^{\varepsilon}_s-j^{\varepsilon}_{t-\delta}\right|^{2p}_H\right.\\
	&\left.-2p\left|j^{\varepsilon}_s-j^{\varepsilon}_{t-\delta}\right|^{2p-2}_H\left(j^{\varepsilon}_s-j^{\varepsilon}_{t-\delta},\sigma_1(j^{\varepsilon}_s,\theta^{\varepsilon}_s,z)\right)\right]\eta_1(ds,dz).
\end{align*}
According to (\ref{e3}) and (\ref{e8.5}), we have
\begin{align*}
	&\left|j^{\varepsilon}_t-j^{\varepsilon}_{t-\delta}\right|^{2p}_H+2p\int_{t-\delta}^t\left|j^{\varepsilon}_s-j^{\varepsilon}_{t-\delta}\right|^{2p-2}_H\left|\triangledown j^{\varepsilon}_s\right|^2_Hds\\
	\leq &2p\int_{t-\delta}^t\left|j^{\varepsilon}_s-j^{\varepsilon}_{t-\delta}\right|^{2p-2}_H\left|j^{\varepsilon}_s\right|^2_{H^1}\left|j^{\varepsilon}_{t-\delta}\right|_Hds+2p\int_{t-\delta}^t\left|j^{\varepsilon}_s-j^{\varepsilon}_{t-\delta}\right|^{2p-2}_H\left|\triangledown j^{\varepsilon}_s\right|_H\left|\triangledown j^{\varepsilon}_{t-\delta}\right|_Hds\\
	&+2p\int_{t-\delta}^t\left|j^{\varepsilon}_s-j^{\varepsilon}_{t-\delta}\right|^{2p-2}_H\left|\theta^{\varepsilon}_s\right|_H\left|\triangledown j^{\varepsilon}_s\right|_Hds+2p\int_{t-\delta}^t\left|j^{\varepsilon}_s-j^{\varepsilon}_{t-\delta}\right|^{2p-2}_H\left|\theta^{\varepsilon}_s\right|_H\left|\triangledown j^{\varepsilon}_{t-\delta}\right|_Hds\\
	&+M_p\int_{t-\delta}^t\int_{\left|z\right|_H<1}\left(\left|j^{\varepsilon}_s-j^{\varepsilon}_{t-\delta}\right|^{2p-2}_H\left|\sigma_1(j^{\varepsilon}_s,\theta^{\varepsilon}_s,z)\right|^2_H+\left|\sigma_1(j^{\varepsilon}_s,\theta^{\varepsilon}_s,z)\right|^{2p}_H\right)\eta_1(ds,dz)\\
	&+2p\int_{t-\delta}^t\int_{\left|z\right|_H<1}\left|j^{\varepsilon}_s-j^{\varepsilon}_{t-\delta}\right|^{2p-2}_H\left(\sigma_1(j^{\varepsilon}_s,\theta^{\varepsilon}_s,z),j^{\varepsilon}_s-j^{\varepsilon}_{t-\delta}\right)\tilde{\eta}_1(ds,dz).
\end{align*}
The H\"{o}lder's inequality yields that
\begin{align*}
	&\left|j^{\varepsilon}_t-j^{\varepsilon}_{t-\delta}\right|^{2p}_H\\
	\leq & 2p\left\{\int_{t-\delta}^t\left|j^{\varepsilon}_s-j^{\varepsilon}_{t-\delta}\right|^{2p-2}_H\left|j^{\varepsilon}_s\right|^2_{H^1}\left|j^{\varepsilon}_{t-\delta}\right|_Hds\right\}_{J_1}+2p\left\{\int_{t-\delta}^t\left|j^{\varepsilon}_s-j^{\varepsilon}_{t-\delta}\right|^{2p-2}_H\left|\triangledown j^{\varepsilon}_{t-\delta}\right|^2_Hds\right\}_{J_2}\\
	&+2p\left\{\int_{t-\delta}^t\left|j^{\varepsilon}_s-j^{\varepsilon}_{t-\delta}\right|^{2p-2}_H\left| \theta^{\varepsilon}_s\right|^2_Hds\right\}_{J_3}\\
	&+M_p\left\{\int_{t-\delta}^t\int_{\left|z\right|_H<1}\left(\left|j^{\varepsilon}_s-j^{\varepsilon}_{t-\delta}\right|^{2p-2}_H\left|\sigma_1(j^{\varepsilon}_s,\theta^{\varepsilon}_s,z)\right|^2_H+\left|\sigma_1(j^{\varepsilon}_s,\theta^{\varepsilon}_s,z)\right|^{2p}_H\right)\eta_1(ds,dz)\right\}_{J_4}\\
	&+2p\left\{\int_{t-\delta}^t\int_{\left|z\right|_H<1}\left|j^{\varepsilon}_s-j^{\varepsilon}_{t-\delta}\right|^{2p-2}_H\left(\sigma_1(j^{\varepsilon}_s,\theta^{\varepsilon}_s,z),j^{\varepsilon}_s-j^{\varepsilon}_{t-\delta}\right)\tilde{\eta}_1(ds,dz)\right\}_{J_5}.
\end{align*}
For $J_1$, by (\ref{e4}) and Fubini's theorem, we have
\begin{align*}
	\mathbb{E}\int_{\delta}^{\tau_{M}\wedge T}J_1dt\leq &\mathbb{E}\int_{0}^{\tau_{M}\wedge T}\int_s^{s+\delta}\left|j^{\varepsilon}_s-j^{\varepsilon}_{t-\delta}\right|^{2p-2}_H\left|j^{\varepsilon}_s\right|^2_{H^1}\left|j^{\varepsilon}_{t-\delta}\right|_Hdtds\\
	\leq &C_M\mathbb{E}\int_{0}^{\tau_{M}\wedge T}\int_s^{s+\delta}\left|j^{\varepsilon}_s\right|^{2p-2}_H\left|j^{\varepsilon}_s\right|^2_{H^1}dtds\\
	\leq  &C_M\delta\mathbb{E}\int_{0}^{\tau_{M}\wedge T}\left|j^{\varepsilon}_s\right|^{2p-2}_H\left|j^{\varepsilon}_s\right|^2_{H^1}ds\\
	\leq &C_{T,M}\delta.
\end{align*}
For $J_2$, we have
\begin{align*}
	\mathbb{E}\int_{\delta}^{\tau_{M}\wedge T}J_2dt\leq &C_M\mathbb{E}\int_{\delta}^{\tau_{M}\wedge T}\int_{t-\delta}^t\left|j^{\varepsilon}_{t-\delta}\right|^{2p-2}_H\left|\triangledown j^{\varepsilon}_{t-\delta}\right|^2_Hdsdt\\
	\leq &C_M\delta\mathbb{E}\int_{\delta}^{\tau_{M}\wedge T}\left|j^{\varepsilon}_{t-\delta}\right|^{2p-2}_H\left|\triangledown j^{\varepsilon}_{t-\delta}\right|^2_Hdt\\
	\leq &C_{T,M}\delta.
\end{align*}
For $J_3$, we have
\begin{align*}
	\mathbb{E}\int_{\delta}^{\tau_{M}\wedge T}J_3dt\leq & C_p\int_{\delta}^{\tau_{M}\wedge T}\int_{t-\delta}^t\left(1+\mathbb{E}\left|j^{\varepsilon}_s\right|^{2p}_H+\mathbb{E}\left|j^{\varepsilon}_{t-\delta}\right|^{2p}_H+\mathbb{E}\left|\theta^{\varepsilon}_s\right|^{2p}_H\right)dsdt\\
	\leq & C_T\delta.
\end{align*}
For $J_4$, we have
\begin{align*}
	\mathbb{E}\int_{\delta}^{\tau_{M}\wedge T}J_4dt\leq & C_p \mathbb{E}\int_{\delta}^{\tau_{M}\wedge T}\int_{t-\delta}^t\left|j^{\varepsilon}_s-j^{\varepsilon}_{t-\delta}\right|^{2p}_Hdsdt\\
	&+C_p\mathbb{E}\int_{\delta}^{\tau_{M}\wedge T}\int_{t-\delta}^t\int_{\left|z\right|_H<1}\left|\sigma_1(j^{\varepsilon}_s,\theta^{\varepsilon}_s,z)\right|^{2p}_H\nu_1(dz)dsdt\\
	\leq & C_p\int_{\delta}^{\tau_{M}\wedge T}\int_{t-\delta}^t\left(1+\mathbb{E}\left|j^{\varepsilon}_s\right|^{2p}_H+\mathbb{E}\left|j^{\varepsilon}_{t-\delta}\right|^{2p}_H+\mathbb{E}\left|\theta^{\varepsilon}_s\right|^{2p}_H\right)dsdt\\
	\leq & C_T\delta.
\end{align*}
For $J_5$, using B-D-G's inequality, we have
\begin{align*}
	&\mathbb{E}\int_{\delta}^{\tau_{M}\wedge T}J_5dt\\
	\leq &C\mathbb{E}\int_{\delta}^{\tau_{M}\wedge T}\left(\int_{t-\delta}^t\int_{\left|z\right|_H<1}\left|j^{\varepsilon}_s-j^{\varepsilon}_{t-\delta}\right|^{4p-2}_H\left|\sigma_1(j^{\varepsilon}_s,\theta^{\varepsilon}_s,z)\right|^{2}_H\eta_1(ds,dz)\right)^{\frac{1}{2}}dt\\
	\leq & C_T\mathbb{E}\left(\int_{\delta}^{\tau_{M}\wedge T}\int_{t-\delta}^t\int_{\left|z\right|_H<1}\left|j^{\varepsilon}_s-j^{\varepsilon}_{t-\delta}\right|^{4p-2}_H\left|\sigma_1(j^{\varepsilon}_s,\theta^{\varepsilon}_s,z)\right|^{2}_H\eta_1(ds,dz)dt\right)^{\frac{1}{2}}\\
	\leq & C_{p,T}\mathbb{E}\left(\int_{\delta}^{\tau_{M}\wedge T}\int_{t-\delta}^t\int_{\left|z\right|_H<1}\left|j^{\varepsilon}_s\right|^{4p-2}_H\left|\sigma_1(j^{\varepsilon}_s,\theta^{\varepsilon}_s,z)\right|^{2}_H\eta_1(ds,dz)dt\right)^{\frac{1}{2}}\\
	&+C_{p,T}\mathbb{E}\left(\int_{\delta}^{\tau_{M}\wedge T}\int_{t-\delta}^t\int_{\left|z\right|_H<1}\left|j^{\varepsilon}_{t-\delta}\right|^{4p-2}_H\left|\sigma_1(j^{\varepsilon}_s,\theta^{\varepsilon}_s,z)\right|^{2}_H\eta_1(ds,dz)dt\right)^{\frac{1}{2}}\\
	\leq & C_{p,T}\mathbb{E}\left(\int_{0}^{\tau_{M}\wedge T}\int_{s}^{s+\delta}\int_{\left|z\right|_H<1}\left|j^{\varepsilon}_s\right|^{4p-2}_H\left|\sigma_1(j^{\varepsilon}_s,\theta^{\varepsilon}_s,z)\right|^{2}_H\eta_1(dt,dz)ds\right)^{\frac{1}{2}}\\
	&+C_{p,T}\mathbb{E}\left(\int_{\delta}^{\tau_{M}\wedge T}\int_{t-\delta}^t\int_{\left|z\right|_H<1}\left|j^{\varepsilon}_{t-\delta}\right|^{4p-2}_H\left|\sigma_1(j^{\varepsilon}_s,\theta^{\varepsilon}_s,z)\right|^{2}_H\eta_1(ds,dz)dt\right)^{\frac{1}{2}}\\
	\leq & C_{p,T}\mathbb{E}\left(\sup_{s\in[0,\tau_{M}\wedge T]}\left|j^{\varepsilon}_s\right|^{2p}_H\int_{0}^{\tau_{M}\wedge T}\int_{s}^{s+\delta}\int_{\left|z\right|_H<1}\left|j^{\varepsilon}_s\right|^{2p-2}_H\left|\sigma_1(j^{\varepsilon}_s,\theta^{\varepsilon}_s,z)\right|^{2}_H\eta_1(dt,dz)ds\right)^{\frac{1}{2}}\\
	&+C_{p,T}\mathbb{E}\left(\sup_{t\in[\delta,\tau_{M}\wedge T]}\left|j^{\varepsilon}_{t-\delta}\right|^{2p}_H\int_{\delta}^{\tau_{M}\wedge T}\int_{t-\delta}^t\int_{\left|z\right|_H<1}\left|j^{\varepsilon}_{t-\delta}\right|^{2p-2}_H\left|\sigma_1(j^{\varepsilon}_s,\theta^{\varepsilon}_s,z)\right|^{2}_H\eta_1(ds,dz)dt\right)^{\frac{1}{2}}\\
	\leq & C_{p,T}\left(\mathbb{E}\sup_{s\in[0,\tau_{M}\wedge T]}\left|j^{\varepsilon}_s\right|^{2p}_H\right)^{\frac{1}{2}}\left(\mathbb{E}\int_{0}^{\tau_{M}\wedge T}\int_{s}^{s+\delta}\int_{\left|z\right|_H<1}\left|j^{\varepsilon}_s\right|^{2p-2}_H\left|\sigma_1(j^{\varepsilon}_s,\theta^{\varepsilon}_s,z)\right|^{2}_H\eta_1(dt,dz)ds\right)^{\frac{1}{2}}\\
	&+C_{p,T}\left(\mathbb{E}\sup_{t\in[\delta,\tau_{M}\wedge T]}\left|j^{\varepsilon}_{t-\delta}\right|^{2p}_H\right)^{\frac{1}{2}}\left(\mathbb{E}\int_{\delta}^{\tau_{M}\wedge T}\int_{t-\delta}^t\int_{\left|z\right|_H<1}\left|j^{\varepsilon}_{t-\delta}\right|^{2p-2}_H\left|\sigma_1(j^{\varepsilon}_s,\theta^{\varepsilon}_s,z)\right|^{2}_H\eta_1(ds,dz)dt\right)^{\frac{1}{2}}\\
	\leq &C_{p,T}\left(\mathbb{E}\int_{0}^{\tau_{M}\wedge T}\int_{s}^{s+\delta}\left(1+\left|j^{\varepsilon}_s\right|^{2p}_H+\left|\theta^{\varepsilon}_s\right|^{2p}_H\right)dtds\right)^{\frac{1}{2}}\\
	&+ C_{p,T}\left(\mathbb{E}\int_{\delta}^{\tau_{M}\wedge T}\int_{t-\delta}^t\left(1+\left|j^{\varepsilon}_{t-\delta}\right|^{2p}_H+\left|j^{\varepsilon}_s\right|^{2p}_H+\left|\theta^{\varepsilon}_s\right|^{2p}_H\right)dsdt\right)^{\frac{1}{2}}\\
	\leq & C_{p,T}\delta^{\frac{1}{2}}.
\end{align*}
Hence, (\ref{e8.5.5}) hold.
\end{proof}

According to \cite[Proposition 10]{Justin},  (\ref{e4}) and (\ref{e8}) implies the following result:
\begin{theorem}\label{th3}
For any $t\in [0,\tau _M\wedge T]$ and $\left ( j_0,\theta_0\right )\in H^1\times H$, the family $\left \{\mathscr{L}\left ( j^{\varepsilon}\right )\right \}_{\varepsilon \in(0,1]}$ is tight in $\mathcal{D}\left ( \left [ 0,\tau _M\wedge T\right ],H\right )$, where $\mathscr{L}\left ( j^{\varepsilon}\right )$ denotes the distribution of $j^{\varepsilon}$ and $\mathcal{D}\left ( \left [ 0,\tau _M\wedge T\right ],H\right )$ represents the space of c\`{a}dl\`{a}g functions from $[0,\tau _M\wedge T]\rightarrow H$.
\end{theorem}

%% -------------------------------------------------------------------

\section{Ergodicity of the frozen equation}\label{s4}
In this section, we focus on the frozen equation for the second equation of (\ref{e2}). For any initial value $\theta \in H$ and fixed $u\in H^1$, we consider the following equation:
\begin{equation}\label{e9}
\left\{
\begin{aligned}%对齐
&\frac{\partial\tilde{\theta}_t}{\partial t}+(u\cdot \triangledown )\tilde{\theta}_t=\triangle\tilde{\theta}_t+\int_{\left|z\right|_{H}<1}\sigma_2(j,\tilde{\theta}_t,z)\dot{\tilde{\eta} }_2(t,dz),\\
&\tilde{\theta}_0=\theta,
\end{aligned}
\right.
\end{equation}
where $j:=\triangledown^{\perp }\cdot u$, compensated Poisson random measure $\tilde{\eta}_2(t,x)=\eta_2(t,x)-\nu_2(x)\cdot t$ with intensity measure $\nu_2$.

It is easy to known that Eq. (\ref{e9}) has a unique solution $\tilde{\theta}_t^{u,\theta }$, and it is easy to check that $\theta_t=\tilde{\theta}_{\frac{t}{\varepsilon }}^{u^{\varepsilon} ,\theta_0 }$, which is the solution for the second equation of (\ref{e2}), if we take a time-scale transformation $t\rightarrow \varepsilon t$ and $u=u^{\varepsilon}$. Let $\left \{P_t^{u}\right \}_{t\geqslant 0}$ be the Markov transition semigroup of $\tilde{\theta}_t^{u,\theta }$, i.e. for any $\varphi \in \mathcal{B}_b(H)$,
\begin{align*}
P_t^{u}\varphi(\theta ):=\mathbb{E}\varphi(\tilde{\theta}_t^{u,\theta }),\quad \theta\in H,\quad t\geqslant 0.
\end{align*}

For Eq.(\ref{e9}), we declare the following lemma:
\begin{lemma}\label{l3}
For any $(\theta ,u)\in H\times H^1$, and for some $p\geq 1$, we have
\begin{align}\label{e10}
\mathbb{E}\left |\tilde{\theta }^{u,\theta}_t\right |^{2p}_H+\mathbb{E}\int_0^t\left |\tilde{\theta }^{u,\theta}_s\right |^{2p-2}_H\left|\triangledown\tilde{\theta }^{u,\theta}_s \right|^2_Hds\leqslant \left | \theta \right |^{2p}_He^{-\lambda_{p,\gamma} t}+C_p\left ( 1+ \left | j\right |^{2p}_H\right ),
\end{align}
where $\lambda_{p,\gamma}>0$ has already appeared in (\ref{e5.5}).
In addition, for any $\theta_1,\theta_2\in H$,
\begin{align}\label{e11}
\mathbb{E}\left|\tilde{\theta }^{u,\theta_1}_t-\tilde{\theta }^{u,\theta_2}_t\right|^{2p}_H\leqslant \left |\theta_1-\theta_2 \right |^{2p}_He^{-\lambda_p t}.
\end{align}
Furthermore, for any $u_1,u_2\in H^1$,
\begin{align}\label{e12}
\nonumber\mathbb{E}\left|\tilde{\theta }^{u_1,\theta}_t-\tilde{\theta }^{u_2,\theta}_t\right|^{2p}_H\leq& C\int_0^te^{-\lambda_{p,\gamma}(t-s)}\mathbb{E}\left|\tilde{\theta }^{u_1,\theta}_s-\tilde{\theta }^{u_2,\theta}_s\right|^{2p-2}_H\left|\triangledown\tilde{\theta }^{u_2,\theta}_s\right|^{2}_{H}ds\cdot \left|j_1-j_2\right|^{2}_{H}\\
	&+C\kappa\left(\left|j_1-j_2\right|^{2p}_H\right).
\end{align}
\end{lemma}
\begin{proof}
By It\^{o}'s formula, we have
\begin{align*}
	&\left|\tilde{\theta }^{u,\theta}_t\right|^{2p}_H\\
	=&\left|\theta\right|^{2p}_H+2p\int_0^t\left|\tilde{\theta }^{u,\theta}_s\right|^{2(p-1)}_H\left \langle \triangle \tilde{\theta }^{u,\theta}_s,\tilde{\theta }^{u,\theta}_s \right \rangle ds\\
	&+2p\int_0^t\int_{\left|z\right|_H<1}\left|\tilde{\theta }^{u,\theta}_s\right|^{2(p-1)}_H\left(\tilde{\theta }^{u,\theta}_s,\sigma_2(j,\tilde{\theta }^{u,\theta}_s,z)\right)\tilde{\eta}_2(ds,dz)\\
	&+\int_0^t\int_{\left|z\right|_H<1}\left [\left|\tilde{\theta }^{u,\theta}_s+\sigma_2(j,\tilde{\theta }^{u,\theta}_s,z)\right|^{2p}_H-\left|\tilde{\theta }^{u,\theta}_s\right|^{2p}_H-2p\left|\tilde{\theta }^{u,\theta}_s\right|^{2(p-1)}_H\left(\tilde{\theta }^{u,\theta}_s,\sigma_2(j,\tilde{\theta }^{u,\theta}_s,z)\right)  \right ] \eta_2(ds,dz).
\end{align*}
Thereby, we have
\begin{align*}
\mathbb{E}\left|\tilde{\theta }^{u,\theta}_t\right|^{2p}_H=&\left|\theta\right|^{2p}_H+2p\mathbb{E}\int_0^t\left|\tilde{\theta }^{u,\theta}_s\right|^{2(p-1)}_H\left \langle \triangle \tilde{\theta }^{u,\theta}_s,\tilde{\theta }^{u,\theta}_s \right \rangle ds+\mathbb{E}\int_0^t\int_{\left|z\right|_H<1}\left [\left|\tilde{\theta }^{u,\theta}_s+\sigma_2(j,\tilde{\theta }^{u,\theta}_s,z)\right|^{2p}_H\right.\\
	&\left.-\left|\tilde{\theta }^{u,\theta}_s\right|^{2p}_H-2p\left|\tilde{\theta }^{u,\theta}_s\right|^{2(p-1)}_H\left(\tilde{\theta }^{u,\theta}_s,\sigma_2(j,\tilde{\theta }^{u,\theta}_s,z)\right)  \right ] \nu_2(dz)ds.
\end{align*}
And so
\begin{align*}
	&\frac{d}{dt}\mathbb{E}\left|\tilde{\theta }^{u,\theta}_t\right|^{2p}_H+2p\gamma\mathbb{E}\left|\tilde{\theta }^{u,\theta}_t\right|^{2p-2}_H\left|\triangledown\tilde{\theta }^{u,\theta}_t \right|^2_H\\
	\leq & -2p(1-\gamma)\lambda\mathbb{E}\left|\tilde{\theta }^{u,\theta}_t\right|^{2p}_H+M_p\mathbb{E}\int_{\left|z\right|_H<1}\left(\left|\tilde{\theta }^{u,\theta}_t\right|^{2(p-1)}_H\left|\sigma_2(j,\tilde{\theta }^{u,\theta}_t,z)\right|^2_H+\left|\sigma_2(j,\tilde{\theta }^{u,\theta}_t,z)\right|^{2p}_H\right)\nu_2(dz)\\
	\leq & -2p(1-\gamma)\lambda\mathbb{E}\left|\tilde{\theta }^{u,\theta}_t\right|^{2p}_H+M_p\mathbb{E}\int_{\left|z\right|_H<1}\left(\frac{p-1}{p}\left|\tilde{\theta }^{u,\theta}_t\right|^{2p}_H+\frac{1}{p}\left|\sigma_2(j,\tilde{\theta }^{u,\theta}_t,z)\right|^{2p}_H\right)\nu_2(dz)\\
	&+M_pL_{\sigma_2}\mathbb{E}\left|\tilde{\theta }^{u,\theta}_t\right|^{2p}_H+M_pC_1\left(1+\left|j\right|^{2p}_H\right)\\
	\leq &-\lambda_{p,\gamma}\mathbb{E}\left|\tilde{\theta }^{u,\theta}_t\right|^{2p}_H+M_pC_1\left(1+\left|j\right|^{2p}_H\right).
\end{align*}
Using Gronwall's inequality, we have
\begin{align*}
\mathbb{E}\left |\tilde{\theta }^{u,\theta}_t\right |^{2p}_H+\mathbb{E}\int_0^t\left |\tilde{\theta }^{u,\theta}_s\right |^{2p-2}_H\left|\triangledown\tilde{\theta }^{u,\theta}_s \right|^2_Hds\leqslant \left | \theta \right |^{2p}_He^{-\lambda_{p,\gamma} t}+M_pC_1\left ( 1+ \left | j\right |^{2p}_H\right ).
\end{align*}
Furthermore, for any $\theta_1,\theta_2\in H$, we have
\begin{align*}
&\left|\tilde{\theta }^{u,\theta_1}_t-\tilde{\theta }^{u,\theta_2}_t\right|^{2p}_H\\
=&\left|\theta_1-\theta_2\right|^{2p}_H+2p\int_0^t\left|\tilde{\theta }^{u,\theta_1}_s-\tilde{\theta }^{u,\theta_2}_s\right|^{2p-2}_H\left \langle \triangle\left(\tilde{\theta }^{u,\theta_1}_s-\tilde{\theta }^{u,\theta_2}_s\right),\tilde{\theta }^{u,\theta_1}_s-\tilde{\theta }^{u,\theta_2}_s \right \rangle ds\\
&+2p\int_0^t\int_{\left|z\right|_H<1}\left|\tilde{\theta }^{u,\theta_1}_s-\tilde{\theta }^{u,\theta_2}_s\right|^{2p-2}_H\left(\tilde{\theta }^{u,\theta_1}_s-\tilde{\theta }^{u,\theta_2}_s,\sigma_2(j,\tilde{\theta }^{u,\theta_1}_s,z)-\sigma_2(j,\tilde{\theta }^{u,\theta_2}_s,z)\right)\tilde{\eta}_2(ds,dz)\\
&+\int_0^t\int_{\left|z\right|_H<1}\left[\left|\left(\tilde{\theta }^{u,\theta_1}_s-\tilde{\theta }^{u,\theta_2}_s\right)+\left(\sigma_2(j,\tilde{\theta }^{u,\theta_1}_s,z)-\sigma_2(j,\tilde{\theta }^{u,\theta_2}_s,z)\right)\right|^{2p}_H-\left|\tilde{\theta }^{u,\theta_1}_s-\tilde{\theta }^{u,\theta_2}_s\right|^{2p}_H\right.\\
&\left.-2p\left|\tilde{\theta }^{u,\theta_1}_s-\tilde{\theta }^{u,\theta_2}_s\right|^{2p-2}_H\left(\tilde{\theta }^{u,\theta_1}_s-\tilde{\theta }^{u,\theta_2}_s,\sigma_2(j,\tilde{\theta }^{u,\theta_1}_s,z)-\sigma_2(j,\tilde{\theta }^{u,\theta_2}_s,z)\right)\right]\eta_2(ds,dz).
\end{align*}
According to Assumptions \ref{a3}, \ref{a4} and (\ref{e3}), we have
\begin{align*}
&\frac{d}{dt}\mathbb{E}\left|\tilde{\theta }^{u,\theta_1}_t-\tilde{\theta }^{u,\theta_2}_t\right|^{2p}_H\\
\leq& -2p\lambda \mathbb{E}\left|\tilde{\theta }^{u,\theta_1}_t-\tilde{\theta }^{u,\theta_2}_t\right|^{2p}_H+M_p\mathbb{E}\int_{\left|z\right|_H<1}\left(\left|\tilde{\theta }^{u,\theta_1}_t-\tilde{\theta }^{u,\theta_2}_t\right|^{2p-2}_H\left|\sigma_2(j,\tilde{\theta }^{u,\theta_1}_t,z)-\sigma_2(j,\tilde{\theta }^{u,\theta_2}_t,z)\right|^2_H\right.\\
&\left.+\left|\sigma_2(j,\tilde{\theta }^{u,\theta_1}_t,z)-\sigma_2(j,\tilde{\theta }^{u,\theta_2}_t,z)\right|^{2p}_H\right)\nu_2(dz)\\
\leq &-\left(2p\lambda-\frac{M_p(p-1)}{p}\right)\mathbb{E}\left|\tilde{\theta }^{u,\theta_1}_t-\tilde{\theta }^{u,\theta_2}_t\right|^{2p}_H+\frac{M_p(p+1)}{p}\mathbb{E}\int_{\left|z\right|_H<1}\left|\sigma_2(j,\tilde{\theta }^{u,\theta_1}_t,z)-\sigma_2(j,\tilde{\theta }^{u,\theta_2}_t,z)\right|^{2p}_H\nu_2(dz)\\
\leq &-\lambda_p \mathbb{E}\left|\tilde{\theta }^{u,\theta_1}_t-\tilde{\theta }^{u,\theta_2}_t\right|^{2p}_H.
\end{align*}
Gronwall's inequality implies that
\begin{align*}
\mathbb{E}\left|\tilde{\theta }^{u,\theta_1}_t-\tilde{\theta }^{u,\theta_2}_t\right|^{2p}_H\leqslant \left |\theta_1-\theta_2 \right |^{2p}_He^{-\lambda_p t}.
\end{align*}
Similarly, for any $u_1,u_2\in H^2$, we have
\begin{align*}
&\left|\tilde{\theta }^{u_1,\theta}_t-\tilde{\theta }^{u_2,\theta}_t\right|^{2p}_H\\
=&-2p\int_0^t\left|\tilde{\theta }^{u_1,\theta}_s-\tilde{\theta }^{u_2,\theta}_s\right|^{2p-2}_H\left((u_1\cdot \triangledown)\tilde{\theta }^{u_1,\theta}_s-(u_2\cdot \triangledown)\tilde{\theta }^{u_2,\theta}_s,\tilde{\theta }^{u_1,\theta}_s-\tilde{\theta }^{u_2,\theta}_s\right)ds\\
&+2p\int_0^t\left|\tilde{\theta }^{u_1,\theta}_s-\tilde{\theta }^{u_2,\theta}_s\right|^{2p-2}_H\left \langle \triangle\left(\tilde{\theta }^{u_1,\theta}_s-\tilde{\theta }^{u_2,\theta}_s\right),\tilde{\theta }^{u_1,\theta}_s-\tilde{\theta }^{u_2,\theta}_s \right \rangle ds\\
&+2p\int_0^t\int_{\left|z\right|_H<1}\left|\tilde{\theta }^{u_1,\theta}_s-\tilde{\theta }^{u_2,\theta}_s\right|^{2p-2}_H\left(\tilde{\theta }^{u_1,\theta}_s-\tilde{\theta }^{u_2,\theta}_s,\sigma_2(j_1,\tilde{\theta }^{u_1,\theta}_s,z)-\sigma_2(j_2,\tilde{\theta }^{u_2,\theta}_s,z)\right)\tilde{\eta}_2(ds,dz)\\
&+\int_0^t\int_{\left|z\right|_H<1}\left[\left|\left(\tilde{\theta }^{u_1,\theta}_s-\tilde{\theta }^{u_2,\theta}_s\right)+\left(\sigma_2(j_1,\tilde{\theta }^{u_1,\theta}_s,z)-\sigma_2(j_2,\tilde{\theta }^{u_2,\theta}_s,z)\right)\right|^{2p}_H-\left|\tilde{\theta }^{u_1,\theta}_s-\tilde{\theta }^{u_2,\theta}_s\right|^{2p}_H\right.\\
&\left.-2p\left|\tilde{\theta }^{u_1,\theta}_s-\tilde{\theta }^{u_2,\theta}_s\right|^{2p-2}_H\left(\tilde{\theta }^{u_1,\theta}_s-\tilde{\theta }^{u_2,\theta}_s,\sigma_2(j_1,\tilde{\theta }^{u_1,\theta}_s,z)-\sigma_2(j_2,\tilde{\theta }^{u_2,\theta}_s,z)\right)\right]\eta_2(ds,dz).
\end{align*}
By (\ref{e8.5}), we can obtain that
\begin{align*}
	&\frac{d}{dt}\mathbb{E}\left|\tilde{\theta }^{u_1,\theta}_t-\tilde{\theta }^{u_2,\theta}_t\right|^{2p}_H+2p\gamma\mathbb{E}\left|\tilde{\theta }^{u_1,\theta}_t-\tilde{\theta }^{u_2,\theta}_t\right|^{2p-2}_H\left|\triangledown\left(\tilde{\theta }^{u_1,\theta}_t-\tilde{\theta }^{u_2,\theta}_t\right)\right|^2_H\\
	\leq & -2p(1-\gamma)\lambda\mathbb{E}\left|\tilde{\theta }^{u_1,\theta}_t-\tilde{\theta }^{u_2,\theta}_t\right|^{2p}_H+2p\mathbb{E}\left|\tilde{\theta }^{u_1,\theta}_t-\tilde{\theta }^{u_2,\theta}_t\right|^{2p-2}_H\left(\left((u_1-u_2)\cdot \triangledown\right)\left(\tilde{\theta }^{u_1,\theta}_t-\tilde{\theta }^{u_2,\theta}_t\right),\tilde{\theta }^{u_2,\theta}_t\right)\\
	&+M_p\mathbb{E}\int_{\left|z\right|_H<1}\left(\left|\tilde{\theta }^{u_1,\theta}_t-\tilde{\theta }^{u_2,\theta}_t\right|^{2p-2}_H\left|\sigma_2(j_1,\tilde{\theta }^{u_1,\theta}_t,z)-\sigma_2(j_2,\tilde{\theta }^{u_2,\theta}_t,z)\right|^2_H\right.\\
	&\left.+\left|\sigma_2(j_1,\tilde{\theta }^{u_1,\theta}_t,z)-\sigma_2(j_2,\tilde{\theta }^{u_2,\theta}_t,z)\right|^{2p}_H\right)\nu_2(dz)\\
	\leq & -\left(2p(1-\gamma)-\frac{M(p-1)}{p}\right)\lambda\mathbb{E}\left|\tilde{\theta }^{u_1,\theta}_t-\tilde{\theta }^{u_2,\theta}_t\right|^{2p}_H+2p\mathbb{E}\left|\tilde{\theta }^{u_1,\theta}_t-\tilde{\theta }^{u_2,\theta}_t\right|^{2p-2}_H\left|\tilde{\theta }^{u_1,\theta}_t-\tilde{\theta }^{u_2,\theta}_t\right|_{H^1}\\
	&\cdot \left|\triangledown\tilde{\theta }^{u_2,\theta}_t\right|_{H}\left|u_1-u_2\right|_{H^1}+\frac{M_p(p+1)}{p}\mathbb{E}\int_{\left|z\right|_H<1}\left|\sigma_2(j_1,\tilde{\theta }^{u_1,\theta}_t,z)-\sigma_2(j_1,\tilde{\theta }^{u_1,\theta_2}_t,z)\right|^{2p}_H\nu_2(dz).		
\end{align*}
Using Young's inequality, and for $\gamma$ small enough, we have
\begin{align*}
\frac{d}{dt}\mathbb{E}\left|\tilde{\theta }^{u_1,\theta}_t-\tilde{\theta }^{u_2,\theta}_t\right|^{2p}_H\leq &  -\left(2p(1-\gamma)-\frac{M(p-1)}{p}-\frac{M_p(p+1)L_{\sigma_2}}{p}\right)\lambda\mathbb{E}\left|\tilde{\theta }^{u_1,\theta}_t-\tilde{\theta }^{u_2,\theta}_t\right|^{2p}_H\\
&+C_{p,\gamma}2p\mathbb{E}\left|\tilde{\theta }^{u_1,\theta}_t-\tilde{\theta }^{u_2,\theta}_t\right|^{2p-2}_H\left|\triangledown\tilde{\theta }^{u_2,\theta}_t\right|^2_{H}\left|u_1-u_2\right|^2_{H^1}+\kappa\left(\left|j_1-j_2\right|^{2p}_H\right)\\
\leq &-\lambda_{p,\gamma}\mathbb{E}\left|\tilde{\theta }^{u_1,\theta}_t-\tilde{\theta }^{u_2,\theta}_t\right|^{2p}_H+C_{p,\gamma}\mathbb{E}\left|\tilde{\theta }^{u_1,\theta}_t-\tilde{\theta }^{u_2,\theta}_t\right|^{2p-2}_H\left|\tilde{\theta }^{u_2,\theta}_t\right|^{2}_{H^1}\left|j_1-j_2\right|^{2}_{H}\\
&+C\kappa\left(\left|j_1-j_2\right|^{2p}_H\right).
\end{align*}
By Gronwall's inequality and (\ref{e10}), we have
\begin{align*}
	\mathbb{E}\left|\tilde{\theta }^{u_1,\theta}_t-\tilde{\theta }^{u_2,\theta}_t\right|^{2p}_H\leq& C\int_0^te^{-\lambda_{p,\gamma}(t-s)}\mathbb{E}\left|\tilde{\theta }^{u_1,\theta}_s-\tilde{\theta }^{u_2,\theta}_s\right|^{2p-2}_H\left|\triangledown\tilde{\theta }^{u_2,\theta}_s\right|^{2}_{H}ds\cdot \left|j_1-j_2\right|^{2}_{H}\\
	&+C\kappa\left(\left|j_1-j_2\right|^{2p}_H\right).
\end{align*}
\end{proof}
Next, we will introduce the ergodicity of $\left \{P_t^{u}\right \}_{t\geqslant 0}$. We have reached the following conclusion:
\begin{theorem}\label{th4}
For any $u\in H^1$, $\left \{P_t^{u}\right \}_{t\geqslant 0}$ has a unique ergodic invariant measure $\pi^{u}(\cdot )$. Moreover,
\begin{align}\label{e13}
\int _{H}{\left | \theta \right |^{2p}_H\pi ^{u} (d\theta  )}\leqslant C\left ( 1+\left | j\right |^{2p}_H\right ).
\end{align}
\end{theorem}
\begin{proof}
According to the classical Bogoliubov-Krylov argument, (\ref{e10}) imples the existence of invariant measure. Moreover, let $\pi ^{u} (\cdot )$ be an invariant measure of $\left \{P_t^{u}\right \}_{t\geqslant 0}$, we have
\begin{align*}
\int _{H}{\left | \theta \right |^{2p}_H\pi ^{u} (d\theta  )}&=\int _{H}{\mathbb{E}\left | \tilde{\theta }_t^{u,\theta } \right |^{2p}_H\pi ^{u} (d\theta  )}\\
&\leq \int _{H}{\left | \theta \right |^{2p}_He^{-\lambda_{p,\gamma} t}\pi ^{u} (d\theta  )}+ C_p\left ( 1+\left | j\right |^{2p}_H\right ).
\end{align*}
Choose $t=-\frac{\ln \frac{1}{2}}{\lambda_{p,\gamma}}$, and (\ref{e13}) hold.

For the uniqueness of invariant measure, as shown in \cite{Da}, we only need to prove that for any Lipschitz-continuous function $\varphi :H\rightarrow \mathbb{R}$ and any invariant measure $\pi ^{u} (\cdot )$, the following formula holds:
\begin{align}\label{e14}
\left | P_t^{u}\varphi(\theta_1 )-\int _{H}{\varphi(\theta_2)\pi ^{u} (d\theta_2  )}\right |\leqslant CL_{\varphi }\left ( 1+\left | \theta \right |_H+\left | j\right |_H\right )e^{-\frac{\lambda_{p}}{p}t},
\end{align}
where $L_{\varphi }=\mathop{\sup}\limits_{\theta_1,\theta_2\in H}\frac{\left | \varphi (\theta_1)-\varphi (\theta_2)\right |}{\left | \theta_1-\theta_2\right |_H}$.

To this end, by (\ref{e11}) and (\ref{e13}), we have
\begin{align*}
\left | P_t^{u}\varphi(\theta_1 )-\int _{H}{\varphi(\theta_2)\pi ^{u} (d\theta_2  )}\right |&=\left |\int _{H}{P_t^{u}\varphi(\theta_1 )-P_t^{u}\varphi(\theta_2 )\pi ^{u} (d\theta_2  )} \right |\\
&\leq \int _{H}{\mathbb{E}\left |\varphi(\tilde{\theta } _t^{u ,\theta_1 })-\varphi(\tilde{\theta } _t^{u,\theta_2 }) \right |\pi ^{u} (d\theta_2  )}\\
&\leq L_{\varphi }\int _{H}{\mathbb{E}\left |\tilde{\theta } _t^{u,\theta_1 }-\tilde{\theta } _t^{u,\theta_2 } \right |_{H}\pi ^{u} (d\theta_2  )}\\
&\leq L_{\varphi }e^{-\frac{\lambda_{p}}{p}t}\int _{H}{\left | \theta_1 -\theta_2\right |_H\pi ^{u} (d\theta_2  )}\\
&\leq CL_{\varphi }\left ( 1+\left | \theta \right |_H+\left | j\right |_H\right )e^{-\frac{\lambda_{p}}{p}t}.
\end{align*}
\end{proof}

%% -------------------------------------------------------------------
\section{The averaging principle of Eq.(\ref{e2})}\label{s5}
In this section, we introduce the averaged equation for Eq.(\ref{e2}) and prove the first main result of this paper. Namely, for any $T> 0$ and $\left ( j_0,\theta_0\right )\in H\times H$, the slow variables $j^{\varepsilon}$ converges to the solution $\bar{j}$ of the following averaged equations:
\begin{equation}\label{e15}
\left\{
\begin{aligned}%对齐
&\frac{\partial \bar{j}}{\partial t}+\left(\bar{u}\cdot \triangledown\right)\bar{j}=\triangle\bar{j}+\bar{f}(\bar{j},\bar{u})+\partial_x g(\bar{u})+\int_{\left|z\right|_H<1}\bar{\sigma}_1(\bar{j},\bar{u},z)\dot{\tilde{\eta} }_1(t,dz) ,\\
&\bar{j}(0)=j_0,
\end{aligned}
\right.
\end{equation}
where $\bar{j}:=\triangledown^{\perp }\cdot \bar{u}$, $g:H^1\rightarrow H$ defined by
\begin{align*}
g(u):=\int _{H}{\hat{\theta} \pi ^{u} d(\hat{\theta} )},
\end{align*}
and
\begin{align*}
\bar{f}(j,u):=\int _Hf(j,\hat{\theta})\pi^{u}(d\hat{\theta} ),\quad \bar{\sigma}_1(j,u,z):=\int_H\sigma(j,\hat{\theta},z)\pi^{u}(d\hat{\theta} ).
\end{align*}

To establish the well-posedness of (\ref{e15}), we require the non-Lipschitz continuity of $g$, $\bar{f}$ and $\bar{\sigma}_2$. To this end, we introduce the following lemma:
\begin{lemma}\label{l4}
	Let $a>0$ and $f(t)\geq 0$, if $F(t):=\int_0^tf(s)ds$ is bounded with respect to $t$, we have
	\begin{align*}
		\lim_{t\rightarrow +\infty}\int_0^te^{-a(t-s)}f(s)ds=0.
	\end{align*}
\end{lemma}
The proof of this lemma is given in the Appendix.

Next, we derive the non-Lipschitz continuity of the coefficients in Eq.(\ref{e15}). We assert the following property:
\begin{proposition}\label{p1}
	Under Assumptions \ref{a1}-\ref{a4}, for any $u_1,u_2\in H^1$ and for  some $p\geq 1$, we have
	\begin{align*}
		\left|g(u_1)-g(u_2)\right|^{2p}_H\leq C_p\kappa\left(\left|j_1-j_2\right|^{2p}_H\right),
	\end{align*}
and
\begin{align*}
	\left|\bar{f}(j_1,u_1)-\bar{f}(j_2,u_2)\right|^{2p}_H\leq C_p\tilde{\kappa}\left(\left|j_1-j_2\right|^{2p}_H\right),
\end{align*}
where $\tilde{\kappa}(u)=\kappa(u+C\kappa(u))$ is also a concave continuous non-decreasing function. Furthermore, we have
\begin{align*}
	\int_{\left|z\right|_H<1}\left|\bar{\sigma}_1(j_1,u_1,z)-\bar{\sigma}_1(j_2,u_2,z)\right|^{2p}_H\nu_1(dz)\leq C_p\tilde{\kappa}\left(\left|j_1-j_2\right|^{2p}_H\right).
\end{align*}
\end{proposition}
\begin{proof}
	First, review the first term on the right-hand side of the inequality in (\ref{e12}), we have
	\begin{align*}
		&\int_0^te^{-\lambda_{p,\gamma}(t-s)}\mathbb{E}\left|\tilde{\theta }^{u_1,\theta}_s-\tilde{\theta }^{u_2,\theta}_s\right|^{2p-2}_H\left|\triangledown\tilde{\theta }^{u_2,\theta}_s\right|^{2}_{H}ds \\
		\leq &C_p\left\{\int_0^te^{-\lambda_{p,\gamma}(t-s)}\mathbb{E}\left|\tilde{\theta }^{u_2,\theta}_s\right|^{2p-2}_H\left|\triangledown\tilde{\theta }^{u_2,\theta}_s\right|^{2}_{H}ds\right\}_{K_1}+C_p\left\{\int_0^te^{-\lambda_{p,\gamma}(t-s)}\mathbb{E}\left|\tilde{\theta }^{u_1,\theta}_s\right|^{2p-2}_H\left|\triangledown\tilde{\theta }^{u_2,\theta}_s\right|^{2}_{H}ds\right\}_{K_2}.
	\end{align*}
For $K_1$, according to (\ref{e10}) and Lemma \ref{l4}, we have
\begin{align}\label{e16}
	\lim_{t\rightarrow+\infty}\int_0^te^{-\lambda_{p,\gamma}(t-s)}\mathbb{E}\left|\tilde{\theta }^{u_2,\theta}_s\right|^{2p-2}_H\left|\triangledown\tilde{\theta }^{u_2,\theta}_s\right|^{2}_{H}ds=0.
\end{align}
For $K_2$, by It\^{o}'s formula, we have
\begin{align*}
	\frac{d}{dt}\left|\tilde{\theta }^{u_2,\theta}_t\right|^2_H=&2\left \langle \triangle \tilde{\theta }^{u_2,\theta}_t,\tilde{\theta }^{u_2,\theta}_t \right \rangle +2\int_{\left|z\right|_H<1}\left(\tilde{\theta }^{u_2,\theta}_t,\sigma_2(j_2,\tilde{\theta }^{u_2,\theta}_t,z)\right)\dot{\tilde{\eta} }_2(t,dz)\\
	+&\int_{\left|z\right|_H<1}\left|\sigma_2(j_2,\tilde{\theta }^{u_2,\theta}_t,z)\right|^2_H\dot{\eta}_2(t,dz),
\end{align*}
and so
\begin{align*}
	\frac{d}{dt}\left|\tilde{\theta }^{u_2,\theta}_t\right|^2_H=&-2\left|\triangledown \tilde{\theta }^{u_2,\theta}_t\right|^2_H+2\int_{\left|z\right|_H<1}\left(\tilde{\theta }^{u_2,\theta}_t,\sigma_2(j_2,\tilde{\theta }^{u_2,\theta}_t,z)\right)\dot{\tilde{\eta} }_2(t,dz)\\
	+&\int_{\left|z\right|_H<1}\left|\sigma_2(j_2,\tilde{\theta }^{u_2,\theta}_t,z)\right|^2_H\dot{\eta}_2(t,dz).
\end{align*}
Multiply both sides of the above equation by $\left|\tilde{\theta }^{u_1,\theta}_t\right|^{2p-2}_H$, we have
\begin{align*}
	&\left|\tilde{\theta }^{u_1,\theta}_t\right|^{2p-2}_H\frac{d}{dt}\left|\tilde{\theta }^{u_2,\theta}_t\right|^2_H\\
	=&-2\left|\tilde{\theta }^{u_1,\theta}_t\right|^{2p-2}_H\left|\triangledown \tilde{\theta }^{u_2,\theta}_t\right|^2_H+2\int_{\left|z\right|_H<1}\left|\tilde{\theta }^{u_1,\theta}_t\right|^{2p-2}_H\left(\tilde{\theta }^{u_2,\theta}_t,\sigma_2(j_2,\tilde{\theta }^{u_2,\theta}_t,z)\right)\dot{\tilde{\eta} }_2(t,dz)\\
	+&\int_{\left|z\right|_H<1}\left|\tilde{\theta }^{u_1,\theta}_t\right|^{2p-2}_H\left|\sigma_2(j_2,\tilde{\theta }^{u_2,\theta}_t,z)\right|^2_H\dot{\eta}_2(t,dz),
\end{align*}
and so
\begin{align*}
	&\frac{d}{dt}\left|\tilde{\theta }^{u_1,\theta}_t\right|^{2p-2}_H\left|\tilde{\theta }^{u_2,\theta}_t\right|^2_H\\
	=&-2\left|\tilde{\theta }^{u_1,\theta}_t\right|^{2p-2}_H\left|\triangledown \tilde{\theta }^{u_2,\theta}_t\right|^2_H+2\int_{\left|z\right|_H<1}\left|\tilde{\theta }^{u_1,\theta}_t\right|^{2p-2}_H\left(\tilde{\theta }^{u_2,\theta}_t,\sigma_2(j_2,\tilde{\theta }^{u_2,\theta}_t,z)\right)\dot{\tilde{\eta} }_2(t,dz)\\
	&+\int_{\left|z\right|_H<1}\left|\tilde{\theta }^{u_1,\theta}_t\right|^{2p-2}_H\left|\sigma_2(j_2,\tilde{\theta }^{u_2,\theta}_t,z)\right|^2_H\dot{\eta}_2(t,dz)+\left|\tilde{\theta }^{u_2,\theta}_t\right|^2_H\frac{d}{dt}\left|\tilde{\theta }^{u_1,\theta}_t\right|^{2p-2}_H.
\end{align*}
Note that
\begin{align*}
	&\frac{d}{dt}\left|\tilde{\theta }^{u_1,\theta}_t\right|^{2p-2}_H\\
	=&2(p-1)\left|\tilde{\theta }^{u_1,\theta}_t\right|^{2p-4}_H\left \langle \triangle \tilde{\theta }^{u_1,\theta}_t,\tilde{\theta }^{u_1,\theta}_t \right \rangle+2(p-1)\int_{\left|z\right|_H<1}\left|\tilde{\theta }^{u_1,\theta}_t\right|^{2p-4}_H\left(\tilde{\theta }^{u_1,\theta}_t,\sigma_2(j_1,\tilde{\theta }^{u_1,\theta}_t,z)\right)\dot{\tilde{\eta} }_2(t,dz)\\
	&+\int_{\left|z\right|_H<1}\left [\left|\tilde{\theta }^{u_1,\theta}_t+\sigma_2(j_1,\tilde{\theta }^{u_1,\theta}_t,z)\right|^{2p-2}_H-\left|\tilde{\theta }^{u_1,\theta}_t\right|^{2p-2}_H\right.\\
	&\left.-2(p-1)\left|\tilde{\theta }^{u_1,\theta}_t\right|^{2(p-2)}_H\left(\tilde{\theta }^{u_1,\theta}_t,\sigma_2(j_1,\tilde{\theta }^{u_1,\theta}_t,z)\right)  \right ] \dot{\eta}_2(t,dz).
\end{align*}
Therefore
\begin{align*}
	&\frac{d}{dt}\left|\tilde{\theta }^{u_1,\theta}_t\right|^{2p-2}_H\left|\tilde{\theta }^{u_2,\theta}_t\right|^2_H\\
	=&-2\left|\tilde{\theta }^{u_1,\theta}_t\right|^{2p-2}_H\left|\triangledown \tilde{\theta }^{u_2,\theta}_t\right|^2_H+2\int_{\left|z\right|_H<1}\left|\tilde{\theta }^{u_1,\theta}_t\right|^{2p-2}_H\left(\tilde{\theta }^{u_2,\theta}_t,\sigma_2(j_2,\tilde{\theta }^{u_2,\theta}_t,z)\right)\dot{\tilde{\eta} }_2(t,dz)\\
	&+\int_{\left|z\right|_H<1}\left|\tilde{\theta }^{u_1,\theta}_t\right|^{2p-2}_H\left|\sigma_2(j_2,\tilde{\theta }^{u_2,\theta}_t,z)\right|^2_H\dot{\eta}_2(t,dz)\\
	&-2(p-1)\left|\tilde{\theta }^{u_2,\theta}_t\right|^2_H\left|\tilde{\theta }^{u_1,\theta}_t\right|^{2p-4}_H\left|\triangledown \tilde{\theta }^{u_1,\theta}_t\right|^2_H\\
	&+2(p-1)\int_{\left|z\right|_H<1}\left|\tilde{\theta }^{u_2,\theta}_t\right|^2_H\left|\tilde{\theta }^{u_1,\theta}_t\right|^{2p-4}_H\left(\tilde{\theta }^{u_1,\theta}_t,\sigma_2(j_1,\tilde{\theta }^{u_1,\theta}_t,z)\right)\dot{\tilde{\eta} }_2(t,dz)\\
	&+\int_{\left|z\right|_H<1}\left|\tilde{\theta }^{u_2,\theta}_t\right|^2_H\left [\left|\tilde{\theta }^{u_1,\theta}_t+\sigma_2(j_1,\tilde{\theta }^{u_1,\theta}_t,z)\right|^{2p-2}_H-\left|\tilde{\theta }^{u_1,\theta}_t\right|^{2p-2}_H\right.\\
	&\left.-2(p-1)\left|\tilde{\theta }^{u_1,\theta}_t\right|^{2(p-2)}_H\left(\tilde{\theta }^{u_1,\theta}_t,\sigma_2(j_1,\tilde{\theta }^{u_1,\theta}_t,z)\right)  \right ] \dot{\eta}_2(t,dz).
\end{align*}
Hence
\begin{align*}
	&\mathbb{E}\left|\tilde{\theta }^{u_1,\theta}_t\right|^{2p-2}_H\left|\tilde{\theta }^{u_2,\theta}_t\right|^2_H\\
	=&\left|\theta\right|^{2p}_H-2\mathbb{E}\int_0^t\left|\tilde{\theta }^{u_1,\theta}_s\right|^{2p-2}_H\left|\triangledown \tilde{\theta }^{u_2,\theta}_s\right|^2_Hds-2(p-1)\mathbb{E}\int_0^t\left|\tilde{\theta }^{u_2,\theta}_s\right|^2_H\left|\tilde{\theta }^{u_1,\theta}_s\right|^{2p-4}_H\left|\triangledown \tilde{\theta }^{u_1,\theta}_s\right|^2_Hds\\
	&+\mathbb{E}\int_0^t\int_{\left|z\right|_H<1}\left|\tilde{\theta }^{u_1,\theta}_s\right|^{2p-2}_H\left|\sigma_2(j_2,\tilde{\theta }^{u_2,\theta}_s,z)\right|^2_H\nu_2(dz)ds\\
	&+\mathbb{E}\int_0^t\int_{\left|z\right|_H<1}\left|\tilde{\theta }^{u_2,\theta}_s\right|^2_H\left [\left|\tilde{\theta }^{u_1,\theta}_s+\sigma_2(j_1,\tilde{\theta }^{u_1,\theta}_s,z)\right|^{2p-2}_H-\left|\tilde{\theta }^{u_1,\theta}_s\right|^{2p-2}_H\right.\\
	&\left.-2(p-1)\left|\tilde{\theta }^{u_1,\theta}_s\right|^{2(p-2)}_H\left(\tilde{\theta }^{u_1,\theta}_s,\sigma_2(j_1,\tilde{\theta }^{u_1,\theta}_s,z)\right)  \right ] \nu_2(dz)ds.
\end{align*}
According to Assumption \ref{a3}, (\ref{e3}) and (\ref{e10}), it is easy to check that 
\begin{align*}
	&\frac{d}{dt}\mathbb{E}\left|\tilde{\theta }^{u_1,\theta}_t\right|^{2p-2}_H\left|\tilde{\theta }^{u_2,\theta}_t\right|^2_H+\mathbb{E}\left|\tilde{\theta }^{u_1,\theta}_t\right|^{2p-2}_H\left|\triangledown \tilde{\theta }^{u_2,\theta}_t\right|^2_H\\
	\leq& -(2p-1)\lambda\mathbb{E}\left|\tilde{\theta }^{u_1,\theta}_t\right|^{2p-2}_H\left|\tilde{\theta }^{u_2,\theta}_t\right|^2_H+C_p\left(1+\left|\theta\right|^{2p}_H+\left|j_1\right|^{2p}_H+\left|j_2\right|^{2p}_H\right).
\end{align*}
Gronwall's inequality implies that
\begin{align*}
	\mathbb{E}\left|\tilde{\theta }^{u_1,\theta}_t\right|^{2p-2}_H\left|\tilde{\theta }^{u_2,\theta}_t\right|^2_H+\mathbb{E}\int_0^t\left|\tilde{\theta }^{u_1,\theta}_s\right|^{2p-2}_H\left|\triangledown \tilde{\theta }^{u_2,\theta}_s\right|^2_Hds\leq C_p\left(1+\left|\theta\right|^{2p}_H+\left|j_1\right|^{2p}_H+\left|j_2\right|^{2p}_H\right).
\end{align*}
By Lemma \ref{l4}, we obtain that
\begin{align}\label{e17}
	\lim_{t\rightarrow+\infty}\int_0^te^{-\lambda_{p,\gamma}(t-s)}\mathbb{E}\left|\tilde{\theta }^{u_1,\theta}_s\right|^{2p-2}_H\left|\triangledown\tilde{\theta }^{u_2,\theta}_s\right|^{2}_{H}ds=0.
\end{align}
Combining (\ref{e16}) with (\ref{e17}), we obtain that
\begin{align*}
	\lim_{t\rightarrow+\infty}\int_0^te^{-\lambda_{p,\gamma}(t-s)}\mathbb{E}\left|\tilde{\theta }^{u_1,\theta}_s-\tilde{\theta }^{u_2,\theta}_s\right|^{2p-2}_H\left|\triangledown\tilde{\theta }^{u_2,\theta}_s\right|^{2}_{H}ds=0.
\end{align*}
Thus, we have demonstrated that
\begin{align}\label{e18}
	\lim_{t\rightarrow+\infty}\mathbb{E}\left|\tilde{\theta }^{u_1,\theta}_t-\tilde{\theta }^{u_2,\theta}_t\right|^{2p}_H\leq C\kappa\left(\left|j_1-j_2\right|^{2p}_H\right).
\end{align}

Next, by (\ref{e14}), for any $u_1,u_2\in H^1$, we have
\begin{align*}
	&\left|g(u_1)-g(u_2)\right|^{2p}_H\\
	=&\left|\int_H\hat{\theta}\pi^{u_1}(d\hat{\theta})-\int_H\hat{\theta}\pi^{u_2}(d\hat{\theta})\right|^{2p}_H\\
	\leq & C_p\left|\int_H\hat{\theta}\pi^{u_1}(d\hat{\theta})-P^{u_1}_t\theta\right|^{2p}_H+C_p\left|P^{u_1}_t\theta-P^{u_2}_t\theta\right|^{2p}_H+C_p\left|P^{u_2}_t\theta-\int_H\hat{\theta}\pi^{u_2}(d\hat{\theta})\right|^{2p}_H\\
	\leq & C_p\left|\int_H\left(P^{u_1}_t\hat{\theta}-P^{u_1}_t\theta\right)\pi^{u_1}(d\hat{\theta})\right|^{2p}_H+C_p\mathbb{E}\left|\tilde{\theta}^{u_1,\theta}_t-\tilde{\theta}^{u_2,\theta}_t\right|^{2p}_H\\
	&+C_p\left|\int_H\left(P^{u_2}_t\hat{\theta}-P^{u_2}_t\theta\right)\pi^{u_1}(d\hat{\theta})\right|^{2p}_H\\
	\leq &C_p\left(1+\left|\theta\right|^{2p}_H+\left|j_1\right|^{2p}_H+\left|j_2\right|^{2p}_H\right)e^{-\lambda_p t}+C_p\mathbb{E}\left|\tilde{\theta}^{u_1,\theta}_t-\tilde{\theta}^{u_2,\theta}_t\right|^{2p}_H.
\end{align*}
Using (\ref{e18}), and let $t\rightarrow +\infty $, we get
\begin{align*}
	\left|g(u_1)-g(u_2)\right|^{2p}_H\leq C_p\kappa\left(\left|j_1-j_2\right|^{2p}_H\right).
\end{align*}

Subsequently, by Assumption \ref{a1}, we have
\begin{align*}
	&\left|\bar{f}(j_1,u_1)-\bar{f}(j_2,u_2)\right|^{2p}_H\\
	=&\left|\int_H f(j_1,\hat{\theta})\pi^{u_1}d(\hat{\theta})-\int_H f(j_2,\hat{\theta})\pi^{u_2}(d\hat{\theta})\right|^{2p}_H\\
	\leq & C_p\left|\int_H f(j_1,\hat{\theta})\pi^{u_1}d(\hat{\theta})-P^{u_1}_tf(j_1,\theta)\right|^{2p}_H+C_p\left|P^{u_1}_tf(j_1,\theta)-P^{u_2}_tf(j_2,\theta)\right|^{2p}_H\\
	&+C_p\left|\int_H f(j_2,\hat{\theta})\pi^{u_2}d(\hat{\theta})-P^{u_2}_tf(j_2,\theta)\right|^{2p}_H\\
	\leq & C_p\left|\int_H \left(P^{u_1}_tf(j_1,\hat{\theta})-P^{u_1}_tf(j_1,\theta)\right)\pi^{u_1}d(\hat{\theta})\right|^{2p}_H+C_p\mathbb{E}\left|f(j_1,\tilde{\theta}^{u_1,\theta})-f(j_2,\tilde{\theta}^{u_2,\theta})\right|^{2p}_H\\
	&+C_p\left|\int_H \left(P^{u_2}_tf(j_2,\hat{\theta})-P^{u_2}_tf(j_2,\theta)\right)\pi^{u_2}d(\hat{\theta})\right|^{2p}_H\\
	\leq & C_p\int_H \mathbb{E}\left|f(j_1,\tilde{\theta}^{u_1,\hat{\theta}}_t)-f(j_1,\tilde{\theta}^{u_1,\theta}_t)\right|^{2p}_H\pi^{u_1}(d\hat{\theta})+C_p\int_H \mathbb{E}\left|f(j_2,\tilde{\theta}^{u_2,\hat{\theta}}_t)-f(j_2,\tilde{\theta}^{u_2,\theta}_t)\right|^{2p}_H\pi^{u_2}(d\hat{\theta})\\
	&+C_p\kappa\left(\left|j_1-j_2\right|^{2p}_H+\mathbb{E}\left|\tilde{\theta}^{u_1,\theta}-\tilde{\theta}^{u_2,\theta}\right|^{2p}_H\right)\\
	\leq &C_p\kappa\left(\int_H\mathbb{E}\left|\tilde{\theta}^{u_1,\hat{\theta}}_t-\tilde{\theta}^{u_1,\theta}_t\right|^{2p}_H\pi^{u_1}(d\hat{\theta})\right)+C_p\kappa\left(\int_H\mathbb{E}\left|\tilde{\theta}^{u_2,\hat{\theta}}_t-\tilde{\theta}^{u_2,\theta}_t\right|^{2p}_H\pi^{u_2}(d\hat{\theta})\right)\\
	&+C_p\kappa\left(\left|j_1-j_2\right|^{2p}_H+\mathbb{E}\left|\tilde{\theta}^{u_1,\theta}-\tilde{\theta}^{u_2,\theta}\right|^{2p}_H\right)\\
	\leq &C_p\kappa\left(C\left(1+\left|\theta\right|^{2p}_H+\left|j_1\right|^{2p}_H+\left|j_2\right|^{2p}_H\right)e^{-\lambda_p t}\right)+C_p\kappa\left(\left|j_1-j_2\right|^{2p}_H+\mathbb{E}\left|\tilde{\theta}^{u_1,\theta}-\tilde{\theta}^{u_2,\theta}\right|^{2p}_H\right).
\end{align*}
Let $t\rightarrow +\infty $, we get
\begin{align*}
	\left|\bar{f}(j_1,u_1)-\bar{f}(j_2,u_2)\right|^{2p}_H\leq C_p\tilde{\kappa}\left(\left|j_1-j_2\right|^{2p}_H\right),
\end{align*}
where $\tilde{\kappa}(u)=\kappa(u+C\kappa(u))$.

At last, according to Assumption \ref{a2}, we have
\begin{align*}
	&\int_{\left|z\right|_H<1}\left|\bar{\sigma}_1(j_1,u_1,z)-\bar{\sigma}_1(j_2,u_2,z)\right|^{2p}_H\nu_1(dz)\\
	=& \int_{\left|z\right|_H<1}\left|\int_H\sigma_1(j_1,\hat{\theta},z)\pi^{u_1}(d\hat{\theta})-\int_H\sigma_1(j_2,\hat{\theta},z)\pi^{u_2}(d\hat{\theta})\right|^{2p}_H\nu_1(dz)\\
	\leq & C_p\int_{\left|z\right|_H<1}\left|\int_H\sigma_1(j_1,\hat{\theta},z)\pi^{u_1}(d\hat{\theta})-P^{u_1}_t\sigma(j_1,\theta,z)\right|^{2p}_H\nu_1(dz)\\
	&+C_p\int_{\left|z\right|_H<1}\left|P^{u_1}_t\sigma_1(j_1,\theta,z)-P^{u_2}_t\sigma_1(j_2,\theta,z)\right|^{2p}_H\nu_1(dz)\\
	&+C_p\int_{\left|z\right|_H<1}\left|\int_H\sigma_1(j_2,\hat{\theta},z)\pi^{u_2}(d\hat{\theta})-P^{u_2}_t\sigma(j_2,\theta,z)\right|^{2p}_H\nu_1(dz)\\
	\leq &C_p\kappa\left(\int_H\mathbb{E}\left|\tilde{\theta}^{u_1,\hat{\theta}}_t-\tilde{\theta}^{u_1,\theta}_t\right|^{2p}_H\pi^{u_1}(d\hat{\theta})\right)+C\kappa\left(\int_H\mathbb{E}\left|\tilde{\theta}^{u_2,\hat{\theta}}_t-\tilde{\theta}^{u_2,\theta}_t\right|^{2p}_H\pi^{u_2}(d\hat{\theta})\right)\\
	&+C\kappa\left(\left|j_1-j_2\right|^{2p}_H+\mathbb{E}\left|\tilde{\theta}^{u_1,\theta}-\tilde{\theta}^{u_2,\theta}\right|^{2p}_H\right)\\
	\leq &C_p\kappa\left(C\left(1+\left|\theta\right|^{2p}_H+\left|j_1\right|^{2p}_H+\left|j_2\right|^{2p}_H\right)e^{-\lambda_p t}\right)+C_p\kappa\left(\left|j_1-j_2\right|^{2p}_H+\mathbb{E}\left|\tilde{\theta}^{u_1,\theta}-\tilde{\theta}^{u_2,\theta}\right|^{2p}_H\right).
\end{align*}
Let $t\rightarrow +\infty $, we get
\begin{align*}
	\int_{\left|z\right|_H<1}\left|\bar{\sigma}_1(j_1,u_1,z)-\bar{\sigma}_1(j_2,u_2,z)\right|^{2p}_H\nu_1(dz)\leq C_p\tilde{\kappa}\left(\left|j_1-j_2\right|^{2p}_H\right).
\end{align*}

\end{proof}

Similar to Theorems \ref{th1} and \ref{th2}, we declare the following Theorem directly.
\begin{theorem}\label{th5}
For any $j_0\in H$, Eq.(\ref{e15}) is well-posed. Furthermore, for any $T>0$ and some $p\geq 1$, we have
\begin{align}\label{e19}
\mathbb{E}\sup_{t\in[0,T]}\left|\bar{j}\right|^{2p}_H+\int_0^T\mathbb{E}\left|\bar{j}\right|^{2p-2}_H\left|\triangledown\bar{j}\right|^2_Hds\leq C_T\left(1+\left|j_0\right|^{2p}_H\right).
\end{align}
\end{theorem}
\subsection{Weak convergence of the averaging principle}
In this subsection, we prove the following conclusion, which is the first main result of this article.
\begin{theorem}\label{th6}
	Suppose that Assumptions \ref{a1}-\ref{a4} hold. For any $T>0$, $\eta>0$ and $(j_0,\theta_0)\in H^1\times H$, we have
\begin{align}\label{e21}
\lim _{\varepsilon \rightarrow 0}\mathbb{P}\left ( \mathop{\sup}\limits_{t\in[0,T]}\left|j^{\varepsilon}-\bar{j}\right|_H>\eta \right )=0.
\end{align}
\end{theorem}

Let's consider the solution of equation (\ref{e2}) in weak sense, i.e. for any test function $\phi\in C_0^2(\mathbb{T}^2)$,
\begin{align*}
	\left(j^{\varepsilon},\phi\right)+\int_0^t\left(\left(u^{\varepsilon}\cdot \triangledown \right)j^{\varepsilon},\phi\right)ds=&\left(j_0,\phi\right)+\int_0^t\left \langle j^{\varepsilon},\triangle\phi\right \rangle ds+\int_0^t\left(f(j^{\varepsilon},\theta^{\varepsilon}),\phi\right)ds\\
	&+\int_0^t\left(\partial_x\theta^{\varepsilon},\phi\right)ds+\int_0^t\int_{\left|z\right|_H<1}\left(\sigma_1(j^{\varepsilon},\theta^{\varepsilon},z),\phi\right)\tilde{\eta}_1(ds,dz).
\end{align*}
Hence, we have
\begin{align*}
	\left(j^{\varepsilon},\phi\right)+\int_0^t\left(\left(u^{\varepsilon}\cdot \triangledown \right)j^{\varepsilon},\phi\right)ds=&\left(j_0,\phi\right)+\int_0^t\left \langle j^{\varepsilon},\triangle\phi\right \rangle ds+\int_0^t\left(\bar{f}(j^{\varepsilon},u^{\varepsilon}),\phi\right)ds\\
	&+\int_0^t\left(\partial_xg(u^{\varepsilon}),\phi\right)ds+\int_0^t\int_{\left|z\right|_H<1}\left(\bar{\sigma}_1(j^{\varepsilon},u^{\varepsilon},z),\phi\right)\tilde{\eta}_1(ds,dz)\\
	&+R^{\varepsilon }_{\phi}(t).
\end{align*}
where the remainder $R^{\varepsilon }_{\phi}(t)$ is given by
\begin{align*}
R^{\varepsilon }_{\phi}(t)=&\left\{\int_0^t\left(f(j^{\varepsilon},\theta^{\varepsilon}),\phi\right)ds-\int_0^t\left(\bar{f}(j^{\varepsilon},u^{\varepsilon}),\phi\right)ds\right\}_{P_1}+\left\{\int_0^t\left(\partial_x\theta^{\varepsilon},\phi\right)ds-\int_0^t\left(\partial_xg(u^{\varepsilon}),\phi\right)ds\right\}_{P_2}\\
&+\left\{\int_0^t\int_{\left|z\right|_H<1}\left(\sigma_1(j^{\varepsilon},\theta^{\varepsilon},z),\phi\right)\tilde{\eta}_1(ds,dz)-\int_0^t\int_{\left|z\right|_H<1}\left(\bar{\sigma}_1(j^{\varepsilon},u^{\varepsilon},z),\phi\right)\tilde{\eta}_1(ds,dz)\right\}_{P_3}.
\end{align*}
Furthermore, we need the following Lemma.
\begin{lemma}\label{l5}
Under Assumptions \ref{a1}-\ref{a4}, for any $T>0$, we have
\begin{align}\label{e22}
\lim _{\varepsilon \rightarrow 0}\mathbb{E}\mathop{\sup}\limits_{t\in[0,\tau_M \wedge T]}\left|R^{\varepsilon }_{\phi}(t)\right|=0,
\end{align}
where stopping time $\tau_M:=\inf\left\{t\geq 0; \left|j^{\varepsilon}\right|_H>M\right\}$, which already appeared in Theorem \ref{th2}.
\end{lemma}
\begin{proof}
For $P_1$, according to Assumption \ref{a1}, (\ref{e11}) and (\ref{e13}), we have
\begin{align*}
	\mathbb{E}\sup_{t\in [0,\tau_M \wedge T]}\left|P_1\right|\leq& C\mathbb{E}\sup_{t\in [0,\tau_M \wedge T]}\int_0^t\left|f(j^{\varepsilon},\theta^{\varepsilon})-\bar{f}(j^{\varepsilon},u^{\varepsilon})\right|_Hds\\
	=& C\mathbb{E}\sup_{t\in [0,\tau_M \wedge T]}\int_0^t\left|f(j^{\varepsilon},\tilde{\theta}^{u^{\varepsilon},\theta_0}_{\frac{s}{\varepsilon}})-\int_Hf(j^{\varepsilon},\hat{\theta})\pi^{u^{\varepsilon}}d(\hat{\theta})\right|_Hds\\
	= &C\mathbb{E}\sup_{t\in [0,\tau_M \wedge T]}\int_0^t\left|\int_Hf(j^{\varepsilon},\tilde{\theta}^{u^{\varepsilon},\theta_0}_{\frac{s}{\varepsilon}})-f(j^{\varepsilon},\hat{\theta})\pi^{u^{\varepsilon}}d(\hat{\theta})\right|_Hds\\
    \leq & C\mathbb{E}\sup_{t\in [0,\tau_M \wedge T]}\int_0^t\int_H\left|f(j^{\varepsilon},\tilde{\theta}^{u^{\varepsilon},\theta_0}_{\frac{s}{\varepsilon}})-f(j^{\varepsilon},\hat{\theta})\right|_H\pi^{u^{\varepsilon}}d(\hat{\theta})ds\\
    \leq & C\mathbb{E}\int_0^{\tau_M \wedge T}\int_H\kappa\left(\left|\tilde{\theta}^{u^{\varepsilon},\theta_0}_{\frac{s}{\varepsilon}}-\hat{\theta}\right|_H\right)\pi^{u^{\varepsilon}}d(\hat{\theta})ds\\
    \leq &C\int_0^{\tau_M \wedge T}\int_H\kappa\left(\mathbb{E}\left|\tilde{\theta}^{u^{\varepsilon},\theta_0}_{\frac{s}{\varepsilon}}-\tilde{\theta}^{u^{\varepsilon},\hat{\theta}}_{\frac{s}{\varepsilon}}\right|_H\right)\pi^{u^{\varepsilon}}d(\hat{\theta})ds\\
    \leq & C\int_0^{\tau_M \wedge T}\kappa\left(e^{-\frac{\lambda_ps}{2p\varepsilon}}\int_H\left|\theta_0-\hat{\theta}\right|_H\pi^{u^{\varepsilon}}d(\hat{\theta})\right)ds\\
    \leq &C\int_0^{\tau_M \wedge T}\kappa\left(C_Me^{-\frac{\lambda_ps}{2p\varepsilon}}\right)ds\\
    \leq &C_{T}\kappa\left(\frac{1}{\tau_M \wedge T}\int_0^{\tau_M \wedge T}e^{-\frac{\lambda_ps}{2p\varepsilon}}ds\right)\\
    \leq & C_{T}\kappa\left(C_{T,M,p}\varepsilon\right).
\end{align*}
For $P_2$, we have
\begin{align*}
	\mathbb{E}\sup_{t\in [0,\tau_M \wedge T]}\left|P_2\right|=&\mathbb{E}\sup_{t\in [0,\tau_M \wedge T]}\left|\int_0^t\left(\theta^{\varepsilon}-g(u^{\varepsilon}),\partial_x\phi\right)ds\right|\\
	\le & C\mathbb{E}\sup_{t\in [0,\tau_M \wedge T]}\int_0^t\left|\theta^{\varepsilon}-g(u^{\varepsilon})\right|_Hds\\
	=& C\mathbb{E}\sup_{t\in [0,\tau_M \wedge T]}\int_0^t\left|\tilde{\theta}^{u^{\varepsilon},\theta_0}_{\frac{s}{\varepsilon}}-\int_H\hat{\theta}\pi^{u^{\varepsilon}}d(\hat{\theta})\right|_Hds\\
	\leq &C\mathbb{E}\sup_{t\in [0,\tau_M \wedge T]}\int_0^t\int_H\left|\tilde{\theta}^{u^{\varepsilon},\theta_0}_{\frac{s}{\varepsilon}}-\hat{\theta}\right|_H\pi^{u^{\varepsilon}}d(\hat{\theta})ds\\
	 \leq &C\int_0^{\tau_M \wedge T}\int_H\mathbb{E}\left|\tilde{\theta}^{u^{\varepsilon},\theta_0}_{\frac{s}{\varepsilon}}-\tilde{\theta}^{u^{\varepsilon},\hat{\theta}}_{\frac{s}{\varepsilon}}\right|_H\pi^{u^{\varepsilon}}d(\hat{\theta})ds\\
	   \leq & C\int_0^{\tau_M \wedge T}e^{-\frac{\lambda_p s}{2p\varepsilon}}\int_H\left|\theta_0-\hat{\theta}\right|_H\pi^{u^{\varepsilon}}d(\hat{\theta})ds\\
    \leq &C_M\int_0^{\tau_M \wedge T}e^{-\frac{\lambda_ps}{2p\varepsilon}}ds\\
    \leq & C_{M,p}\varepsilon.
\end{align*}
For $P_3$, by Assumption\ref{a3} and B-D-G's inequality, we have
\begin{align*}
	\mathbb{E}\sup_{t\in [0,\tau_M \wedge T]}\left|P_3\right|\leq & C\mathbb{E}\left(\int_0^{\tau_M \wedge T}\int_{\left|z\right|_H<1}\left|\sigma_1(j^{\varepsilon},\theta^{\varepsilon},z)-\bar{\sigma}_1(j^{\varepsilon},u^{\varepsilon},z)\right|^2_H\eta_1(ds,dz)\right)^{\frac{1}{2}}\\
	\leq &C\left(\mathbb{E}\int_0^{\tau_M \wedge T}\int_{\left|z\right|_H<1}\left|\sigma_1(j^{\varepsilon},\tilde{\theta}^{u^{\varepsilon},\theta_0}_{\frac{s}{\varepsilon}},z)-\int_H\sigma_1(j^{\varepsilon},\hat{\theta},z)\pi^{u^{\varepsilon}}d(\hat{\theta})\right|^2_H\eta_1(ds,dz)\right)^{\frac{1}{2}}\\
	\leq&C\left(\mathbb{E}\int_0^{\tau_M \wedge T}\int_{\left|z\right|_H<1}\int_H\left|\sigma_1(j^{\varepsilon},\tilde{\theta}^{u^{\varepsilon},\theta_0}_{\frac{s}{\varepsilon}},z)-\sigma_1(j^{\varepsilon},\hat{\theta},z)\right|^2_H\pi^{u^{\varepsilon}}d(\hat{\theta})\eta_1(ds,dz)\right)^{\frac{1}{2}}\\
	=&C\left(\int_0^{\tau_M \wedge T}\int_H\mathbb{E}\int_{\left|z\right|_H<1}\left|\sigma_1(j^{\varepsilon},\tilde{\theta}^{u^{\varepsilon},\theta_0}_{\frac{s}{\varepsilon}},z)-\sigma_1(j^{\varepsilon},\tilde{\theta}^{u^{\varepsilon},\hat{\theta}}_{\frac{s}{\varepsilon}},z)\right|^2_H\eta_1(ds,dz)\pi^{u^{\varepsilon}}d(\hat{\theta})\right)^{\frac{1}{2}}\\
	\leq &C\left(\int_0^{\tau_M \wedge T}\int_H\kappa\left(\mathbb{E}\left|\tilde{\theta}^{u^{\varepsilon},\theta_0}_{\frac{s}{\varepsilon}}-\tilde{\theta}^{u^{\varepsilon},\hat{\theta}}_{\frac{s}{\varepsilon}}\right|^2_H\right)\pi^{u^{\varepsilon}}d(\hat{\theta})ds\right)^{\frac{1}{2}}\\
	\leq &C\left(\int_0^{\tau_M \wedge T}\int_H\kappa\left(e^{-\frac{\lambda_ps}{p\varepsilon}}\left|\theta_0-\hat{\theta}\right|^2_H\right)\pi^{u^{\varepsilon}}d(\hat{\theta})ds\right)^{\frac{1}{2}}\\
	\leq &C\left(\int_0^{\tau_M \wedge T}\kappa\left(e^{-\frac{\lambda_ps}{p\varepsilon}}\int_H\left|\theta_0-\hat{\theta}\right|^2_H\pi^{u^{\varepsilon}}d(\hat{\theta})\right)ds\right)^{\frac{1}{2}}\\
	\leq & C\left(\int_0^{\tau_M \wedge T}\kappa\left(C_Me^{-\frac{\lambda_ps}{p\varepsilon}}\right)ds\right)^{\frac{1}{2}}\\
	\leq &C_T\left(\kappa\left(\frac{1}{\tau_M \wedge T}\int_0^{\tau_M \wedge T}e^{-\frac{\lambda_ps}{p\varepsilon}}ds\right)\right)^{\frac{1}{2}}\\
	\leq & C_{T}\left[\kappa\left(C_{T,M,p}\varepsilon\right)\right]^{\frac{1}{2}}.
\end{align*}

Let $\varepsilon\rightarrow 0$, we obtain that
\begin{align*}
	\lim _{\varepsilon \rightarrow 0}\mathbb{E}\mathop{\sup}\limits_{t\in[0,\tau_M \wedge T]}\left|R^{\varepsilon }_{\phi}(t)\right|=0.
\end{align*}
\end{proof}

\textbf{Proof of Theorem \ref{th6}}. According to Theorem \ref{th3}, the distribution family $\left \{\mathscr{L} (  j^{\varepsilon}  )\right \}_{\epsilon \in(0,1]}$ is tight in $\mathcal{D}\left ([0,T\wedge \tau _M];H\right )$. By the Skorohod representation theorem, for any two subsequences $\left \{\varepsilon ^1_n\right \}_{n\in\mathbb{N}}$ and $\left \{\varepsilon ^2_n\right \}_{n\in\mathbb{N}}$, there exists two sequences $\left \{\hat{j}^{\varepsilon^1_n}\right \}_{n\in \mathbb{N}}$, $\left \{\hat{j}^{\varepsilon^2_n}\right \}_{n\in \mathbb{N}}$ and stochastic processes $\hat{j}^1$, $\hat{j}^2\in \mathcal{D}\left ([0,T\wedge \tau _M];H\right )$ on some probability space $\left ( \hat{\Omega },\hat{\mathcal{F}},\hat{\mathbb{P}}\right )$ such that
\begin{align*}
\left \{j^{\varepsilon^k_n}\right \}_{\varepsilon^k _n\in (0,1]} \xlongequal{in\ law} \quad\left \{\hat{j}^{\varepsilon^k_n}\right \}_{n\in \mathbb{N}},\ k=1,2,
\end{align*}
and
\begin{align*}
\hat{j}^{\varepsilon^k_n}\rightarrow \hat{j}^k,\  \hat{\mathbb{P}}-a.s.\ on\  \mathcal{D}\left ([0,T\wedge \tau _M];H\right ) ,\ as\ n \rightarrow \infty,\ k=1,2.
\end{align*}
As shown by Cerrai and Freidlin in \cite{Cerrai1}, (\ref{e22}) implies that both $\hat{j}^1$ and $\hat{j}^2$ fulfill averaged equations (\ref{e15}). By uniqueness of the solution of equations (\ref{e15}), we have $\hat{j}^1=\hat{j}^2$. Due to the well known argument form Gy\"{o}ngy and Krylov \cite{Gyongy}, there exists some $\bar{j}\in \mathcal{D}([0,T\wedge \tau _M],H)$ such that
\begin{align}\label{e23}
\lim _{\varepsilon \rightarrow 0}\mathbb{P}\left ( \mathop{\sup}\limits_{t\in[0,\tau _M\wedge T]}\left | j^{\varepsilon}-\bar{j}_t\right |_H>\eta \right )=0.
\end{align}
By Chebyshev's inequality, (\ref{e4}) and (\ref{e19}), we have
\begin{align*}
\mathbb{P}\left ( \sup_{t\in [0,T]}\left | j^{\varepsilon}-\bar{j}\right |_H\chi \left \{t>\tau _M\right \}>\eta \right )\leq & \frac{\mathbb{E}\left ( \sup_{t\in [0,T]}\left | j^{\varepsilon}-\bar{j}\right |_H\chi \left \{t>\tau _M\right \} \right )}{\eta }\\
\leq & \frac{1}{\eta }\left ( \mathbb{E}\sup_{t\in [0,T]}\left | j^{\varepsilon}-\bar{j}\right |^2_H\right )^{\frac{1}{2}}\left (\mathbb{P}\left \{t>\tau _M\right \} \right )^{\frac{1}{2}}\\
\leq & \frac{C_T}{\eta }\frac{\left ( \mathbb{E}\sup_{t\in [0,T]}\left | j^{\varepsilon}\right |^{2p}_H\right )^{\frac{1}{2}}}{M^{p}}\\
\leq & \frac{C_{T,p}}{\eta }\cdot M^{-p},
\end{align*}
where $\chi \left \{\cdot \right \}$ is the indicator function.

Combining (\ref{e23}), and Let $\varepsilon \rightarrow 0$ firstly, $M\rightarrow \infty $ secondly, we have
\begin{align*}
\lim _{\varepsilon \rightarrow 0}\mathbb{P}\left ( \mathop{\sup}\limits_{t\in[0,T]}\left | j^{\varepsilon}-\bar{j}\right |_H>\eta \right )=0.
\end{align*}

\rightline{$\qedsymbol$} 
%% -------------------------------------------------------------------
\subsection{Strong convergence of the averaging principle}
We will prove the second main result in this section, that is, slow variables can strongly converge to the solution of the averaged equations (\ref{e15}) in the $p$th-mean. Previously, it needs to be emphasized that we will make $\sigma _1(j,\theta,z)\equiv \sigma _1(j,z)$ and so $\bar{\sigma}_1(j,u,z)\equiv\sigma_1(j,z)$. Certain counterexamples illustrate that strong convergence is not ture when the noise term of the slow equations depends on the fast variable, see \cite{Dror,Di}.

The strong convergence argumentation method originated from \cite{Khasminskii2} by Khasminskii. In particular, for any $T>0$ and $t\in \left [ k\delta ,min\left \{(k+1)\delta ,T\right \}\right ]$, $k=0,1,2,\cdots $, we construct auxiliary process
\begin{align*}
\hat{\theta}^{\varepsilon}_t+\frac{1}{\varepsilon}\int_{k\delta}^t\left(u^{\varepsilon}_{k\delta}\cdot \triangledown\right)\hat{\theta}^{\varepsilon}_sds=\hat{\theta}^{\varepsilon}_{k\delta}+\frac{1}{\varepsilon}\int_{k\delta}^t\triangle\hat{\theta}^{\varepsilon}_sds+\int_{k\delta}^t\int_{\left|z\right|_{H}<1}\sigma_2(j^{\varepsilon}_{k\delta},\hat{\theta}^{\varepsilon}_s,z)\tilde{\eta}^{\varepsilon}_2(ds,dz)
\end{align*}
with $\hat{\theta }^{\varepsilon}_0=\theta_0$. This implies that
\begin{align*}
\hat{\theta}^{\varepsilon}_t+\frac{1}{\varepsilon}\int_{0}^t\left(u^{\varepsilon}_{s(\delta)}\cdot \triangledown\right)\hat{\theta}^{\varepsilon}_sds=\theta_0+\frac{1}{\varepsilon}\int_{0}^t\triangle\hat{\theta}^{\varepsilon}_sds+\int_{0}^t\int_{\left|z\right|_{H}<1}\sigma_2(j^{\varepsilon}_{s(\delta)},\hat{\theta}^{\varepsilon}_s,z)\tilde{\eta}^{\varepsilon}_2(ds,dz),
\end{align*}
where $s-\delta\leq s(\delta):=\left[\frac{s}{\delta}\right]\delta<s$, and $\left[\cdot\right]$ denotes the floor function. In addition, we define
\begin{align*}
\hat{j}^{\varepsilon}_t+\int_0^t\left(u^{\varepsilon}_{s}\cdot \triangledown\right)\hat{j}^{\varepsilon}_sds=&j_0+\int_0^t\triangle \hat{j}^{\varepsilon}_sds+\int_0^tf(j^{\varepsilon}_{s(\delta)},\hat{\theta}^{\varepsilon}_s)ds+\int_0^t\partial_x\hat{\theta}^{\varepsilon}_sds\\
&+\int_0^t\int_{\left|z\right|_H<1}\sigma_1(j^{\varepsilon}_s,z)\tilde{\eta}_1(ds,dz),
\end{align*}
where $\hat{j}^{\varepsilon}:=\triangledown^{\perp }\cdot \hat{u}^{\varepsilon}$. It is easy to obtain that for any $T>0$, $(j_0,\theta_0)\in H\times H$, and for some $p\geq 1$, 
\begin{align}\label{e24}
	\sup_{t\in[0,T]}\mathbb{E}\left|\hat{\theta}^{\varepsilon}_t\right|^{2p}_H+\mathbb{E}\int_0^T\left|\hat{\theta}^{\varepsilon}_s\right|^{2p-2}_H\left|\triangledown \hat{\theta}^{\varepsilon}_s\right|^2_Hds\leq C_T\left(1+\left|\theta_0\right|^{2p}_H+\left|j_0\right|^{2p}_H\right),
\end{align}
and
\begin{align}\label{e25}
	\mathbb{E}\sup_{t\in[0,T]}\left|\hat{j}^{\varepsilon}_t\right|^{2p}_H+\mathbb{E}\int_0^T\left|\hat{j}^{\varepsilon}_s\right|^{2p-2}_H\left|\triangledown \hat{j}^{\varepsilon}_s\right|^2_Hds\leq C_T\left(1+\left|\theta_0\right|^{2p}_H+\left|j_0\right|^{2p}_H\right).
\end{align}

We claim the following Lemma:
\begin{lemma}\label{l6}
Under Assumptions \ref{a1}-\ref{a4}, for any $T>0$, $\left ( j_0,\theta_0\right )\in H\times H$ and for some $p\geq 1$, we have
\begin{align}\label{e25}
	\mathbb{E}\int_0^{\tau _{M,L}\wedge T}\left|\hat{\theta}^{\varepsilon}_t-\theta^{\varepsilon}_t\right|^{2p}_Hdt\leq \mathcal{O}\left(\delta \right),
\end{align}
where $\tau _{M,L}:=\tau_M \wedge \inf\left\{t\geq 0; \left|\theta^{\varepsilon}_t\right|_H+\left|\hat{\theta}^{\varepsilon}_t\right|_H>L\right\}$ and $\mathcal{O}\left(\delta \right):=C_{p,T}\kappa\left(C_{T,M}\delta^{\frac{1}{2}}\right)+C_{p,T,M,L}\delta^{\frac{1}{2p}}$.
\end{lemma}
\begin{proof}
By It\^{o}'s formula, we have
\begin{align*}
&\left|\hat{\theta}^{\varepsilon}_t-\theta^{\varepsilon}_t\right|^{2p}_H+\frac{2p}{\varepsilon}\int_0^t\left|\hat{\theta}^{\varepsilon}_s-\theta^{\varepsilon}_s\right|^{2p-2}_H\left(\left(u^{\varepsilon}_{s(\delta)}\cdot\triangledown\right)\hat{\theta}^{\varepsilon}_s-\left(u^{\varepsilon}_s\cdot \triangledown\right)\theta^{\varepsilon}_s,\hat{\theta}^{\varepsilon}_s-\theta^{\varepsilon}_s\right)ds\\
=&\frac{2p}{\varepsilon}\int_0^t\left|\hat{\theta}^{\varepsilon}_s-\theta^{\varepsilon}_s\right|^{2p-2}_H\left \langle \triangle\left(\hat{\theta}^{\varepsilon}_s-\theta^{\varepsilon}_s\right),\hat{\theta}^{\varepsilon}_s-\theta^{\varepsilon}_s \right \rangle ds\\
&+2p\int_0^t\int_{\left|z\right|_H<1}\left|\hat{\theta}^{\varepsilon}_s-\theta^{\varepsilon}_s\right|^{2p-2}_H\left(\hat{\theta}^{\varepsilon}_s-\theta^{\varepsilon}_s,\sigma_2(j^{\varepsilon}_{s(\delta)},\hat{\theta}^{\varepsilon}_s,z)-\sigma_2(j^{\varepsilon}_s,\theta^{\varepsilon}_s,z)\right)\tilde{\eta}^{\varepsilon}_2(ds,dz)\\
&+\int_0^t\int_{\left|z\right|_H<1}\left[\left|\left(\hat{\theta}^{\varepsilon}_s-\theta^{\varepsilon}_s\right)+\left(\sigma_2(j^{\varepsilon}_{s(\delta)},\hat{\theta}^{\varepsilon}_s,z)-\sigma_2(j^{\varepsilon}_s,\theta^{\varepsilon}_s,z)\right)\right|^{2p}_H-\left|\hat{\theta}^{\varepsilon}_s-\theta^{\varepsilon}_s\right|^{2p}_H\right.\\
&\left.-2p\left|\hat{\theta}^{\varepsilon}_s-\theta^{\varepsilon}_s\right|^{2p-2}_H\left(\hat{\theta}^{\varepsilon}_s-\theta^{\varepsilon}_s,\sigma_2(j^{\varepsilon}_{s(\delta)},\hat{\theta}^{\varepsilon}_s,z)-\sigma_2(j^{\varepsilon}_s,\theta^{\varepsilon}_s,z)\right)\right]\eta^{\varepsilon}_2(ds,dz).
\end{align*}
By (\ref{e8.5}) and Young's inequality, we have
\begin{align*}
	&\frac{d}{dt}\mathbb{E}\left|\hat{\theta}^{\varepsilon}_t-\theta^{\varepsilon}_t\right|^{2p}_H+\frac{p\gamma}{\varepsilon}\mathbb{E}\left|\hat{\theta}^{\varepsilon}_t-\theta^{\varepsilon}_t\right|^{2p-2}_H\left|\triangledown\left(\hat{\theta}^{\varepsilon}_t-\theta^{\varepsilon}_t\right)\right|^{2}_H\\
	\leq& -\frac{1}{\varepsilon}\left(2p\lambda\left(1-\frac{\gamma}{2}\right)-\frac{M_p(p-1)}{p}\right)\mathbb{E}\left|\hat{\theta}^{\varepsilon}_t-\theta^{\varepsilon}_t\right|^{2p}_H+\frac{2p}{\varepsilon}\mathbb{E}\left|\hat{\theta}^{\varepsilon}_t-\theta^{\varepsilon}_t\right|^{2p-2}_H\left|j^{\varepsilon}_t-j^{\varepsilon}_{t(\delta)}\right|_H\left|\hat{\theta}^{\varepsilon}_t-\theta^{\varepsilon}_t\right|_{H^1}\left|\theta^{\varepsilon}_t\right|_{H^1}\\
	&+\frac{M_p(p+1)}{\varepsilon p}\mathbb{E}\int_{\left|z\right|_H<1}\left|\sigma_2(j^{\varepsilon}_{t(\delta)},\hat{\theta}^{\varepsilon}_t,z)-\sigma_2(j^{\varepsilon}_t,\theta^{\varepsilon}_t,z)\right|^{2p}_H\nu_2(dz),
\end{align*}
and
\begin{align*}
\frac{d}{dt}\mathbb{E}\left|\hat{\theta}^{\varepsilon}_t-\theta^{\varepsilon}_t\right|^{2p}_H\leq &-\frac{\lambda_{p,\gamma}}{\varepsilon}\mathbb{E}\left|\hat{\theta}^{\varepsilon}_t-\theta^{\varepsilon}_t\right|^{2p}_H+\frac{C_p}{\varepsilon}\kappa\left(\mathbb{E}\left|j^{\varepsilon}_t-j^{\varepsilon}_{t(\delta)}\right|^{2p}_H\right)\\
&+\frac{C_p}{\varepsilon}\mathbb{E}\left|\hat{\theta}^{\varepsilon}_t-\theta^{\varepsilon}_t\right|^{2p-2}_H\left|j^{\varepsilon}_t-j^{\varepsilon}_{t(\delta)}\right|^2_H\left|\triangledown\theta^{\varepsilon}_t\right|^2_{H}
\end{align*}
with $\gamma$ small enough. Using Gronwall's inequality , we have
\begin{align*}
\mathbb{E}\left|\hat{\theta}^{\varepsilon}_t-\theta^{\varepsilon}_t\right|^{2p}_H\leq & \frac{C_p}{\varepsilon}\int_0^te^{-\frac{\lambda_{p,\gamma}}{\varepsilon}(t-s)}\kappa\left(\mathbb{E}\left|j^{\varepsilon}_s-j^{\varepsilon}_{s(\delta)}\right|^{2p}_H\right)ds\\
&+\frac{C_p}{\varepsilon}\int_0^te^{-\frac{\lambda_{p,\gamma}}{\varepsilon}(t-s)}\mathbb{E}\left|\hat{\theta}^{\varepsilon}_s-\theta^{\varepsilon}_s\right|^{2p-2}_H\left|j^{\varepsilon}_s-j^{\varepsilon}_{s(\delta)}\right|^2_H\left|\triangledown\theta^{\varepsilon}_s\right|^2_{H}ds.
\end{align*}
According to Young’s convolution inequality and (\ref{e8.5.5}), we have
\begin{align*}
	&\mathbb{E}\int_0^{\tau _{M,L}\wedge T}\left|\hat{\theta}^{\varepsilon}_t-\theta^{\varepsilon}_t\right|^{2p}_Hdt\\
	\leq & \frac{C_p}{\varepsilon}\int_0^{\tau _{M,L}\wedge T}\int_0^te^{-\frac{\lambda_{p,\gamma}}{\varepsilon}(t-s)}\kappa\left(\mathbb{E}\left|j^{\varepsilon}_s-j^{\varepsilon}_{s(\delta)}\right|^{2p}_H\right)dsdt\\
	&+\frac{C_p}{\varepsilon}\int_0^{\tau _{M,L}\wedge T}\int_0^te^{-\frac{\lambda_{p,\gamma}}{\varepsilon}(t-s)}\mathbb{E}\left|\hat{\theta}^{\varepsilon}_s-\theta^{\varepsilon}_s\right|^{2p-2}_H\left|j^{\varepsilon}_s-j^{\varepsilon}_{s(\delta)}\right|^2_H\left|\triangledown\theta^{\varepsilon}_s\right|^2_{H}dsdt\\
	\leq &\frac{C_p}{\varepsilon}\int_0^{\tau _{M,L}\wedge T}e^{-\frac{\lambda_{p,\gamma}}{\varepsilon}t}dt\cdot \left(\int_0^{\tau _{M,L}\wedge T}\kappa\left(\mathbb{E}\left|j^{\varepsilon}_t-j^{\varepsilon}_{t(\delta)}\right|^{2p}_H\right)dt\right.\\
	&\left.+\mathbb{E}\int_0^{\tau _{M,L}\wedge T}\left|\hat{\theta}^{\varepsilon}_t-\theta^{\varepsilon}_t\right|^{2p-2}_H\left|j^{\varepsilon}_t-j^{\varepsilon}_{t(\delta)}\right|^2_H\left|\triangledown\theta^{\varepsilon}_t\right|^2_{H}dt\right)\\
	\leq &C_{p,T}\kappa\left(\frac{1}{\tau _{M,L}\wedge T}\mathbb{E}\int_0^{\tau _{M,L}\wedge T}\left|j^{\varepsilon}_t-j^{\varepsilon}_{t(\delta)}\right|^{2p}_Hdt\right)+C_{p,T,L}\mathbb{E}\int_0^{\tau _{M,L}\wedge T}\left|j^{\varepsilon}_t-j^{\varepsilon}_{t(\delta)}\right|^2_H\left|\triangledown\theta^{\varepsilon}_t\right|^2_{H}dt\\
	\leq &C_{p,T}\kappa\left(C_{T,M}\delta^{\frac{1}{2}}\right)+C_{p,T,L}\mathbb{E}\int_0^{\tau _{M,L}\wedge T}\left|j^{\varepsilon}_t-j^{\varepsilon}_{t(\delta)}\right|^2_H\left|\triangledown\theta^{\varepsilon}_t\right|^2_{H}dt.
\end{align*}
Next, we shall address the final term on the right-hand side of the above inequality. Specifically, due to
\begin{align*}
	\frac{d}{dt}\left|\theta^{\varepsilon}_t\right|^2_H+\frac{2}{\varepsilon}\left|\triangledown \theta^{\varepsilon}_t\right|^2_H=2\int_{\left|z\right|_H<1}\left(\theta^{\varepsilon}_t,\sigma_2(j^{\varepsilon}_t,\theta^{\varepsilon}_t,z)\right)\dot{\tilde{\eta}}_2(t,dz)+\int_{\left|z\right|_H<1}\left|\sigma_2(j^{\varepsilon}_t,\theta^{\varepsilon}_t,z)\right|^2_H\dot{\eta}_2(\frac{1}{\varepsilon}t,dz).
\end{align*}
It is easy to obtain that
\begin{align*}
	\left \langle \left|\theta^{\varepsilon}_t\right|^2_H \right \rangle _t=\int_0^t\int_{\left|z\right|_H<1}\left(2\left(\theta^{\varepsilon}_s,\sigma_2(j^{\varepsilon}_s,\theta^{\varepsilon}_s,z)\right)+\left|\sigma_2(j^{\varepsilon}_s,\theta^{\varepsilon}_s,z)\right|^2_H\right)^2\eta_2(\frac{1}{\varepsilon}ds,dz),
\end{align*}
where $\left \langle\cdot \right \rangle _t$ denotes the quadratic variation, and
\begin{align*}
	&\left|j^{\varepsilon}_t-j^{\varepsilon}_{t(\delta)}\right|^2_H\frac{d}{dt}\left|\theta^{\varepsilon}_t\right|^2_H+\frac{2}{\varepsilon}\left|j^{\varepsilon}_t-j^{\varepsilon}_{t(\delta)}\right|^2_H\left|\triangledown \theta^{\varepsilon}_t\right|^2_H\\
	=&2\int_{\left|z\right|_H<1}\left|j^{\varepsilon}_t-j^{\varepsilon}_{t(\delta)}\right|^2_H\left(\theta^{\varepsilon}_t,\sigma_2(j^{\varepsilon}_t,\theta^{\varepsilon}_t,z)\right)\dot{\tilde{\eta}}_2(t,dz)+\int_{\left|z\right|_H<1}\left|j^{\varepsilon}_t-j^{\varepsilon}_{t(\delta)}\right|^2_H\left|\sigma_2(j^{\varepsilon}_t,\theta^{\varepsilon}_t,z)\right|^2_H\dot{\eta}_2(\frac{1}{\varepsilon}t,dz).
\end{align*}
Therefore, by B-D-G's inequality, we have
\begin{align*}
	&\mathbb{E}\int_0^{\tau _{M,L}\wedge T}\left|j^{\varepsilon}_t-j^{\varepsilon}_{t(\delta)}\right|^2_H\left|\triangledown \theta^{\varepsilon}_t\right|^2_Hdt\\
	=&-\frac{\varepsilon}{2}\mathbb{E}\int_0^{\tau _{M,L}\wedge T}\left|j^{\varepsilon}_t-j^{\varepsilon}_{t(\delta)}\right|^2_Hd\left|\theta^{\varepsilon}_t\right|^2_H+\varepsilon\mathbb{E}\int_0^{\tau _{M,L}\wedge T}\int_{\left|z\right|_H<1}\left|j^{\varepsilon}_t-j^{\varepsilon}_{t(\delta)}\right|^2_H\left|\sigma_2(j^{\varepsilon}_t,\theta^{\varepsilon}_t,z)\right|^2_H\eta_2(\frac{1}{\varepsilon}dt,dz)\\
	\leq & C\varepsilon\mathbb{E}\left(\int_0^{\tau _{M,L}\wedge T}\int_{\left|z\right|_H<1}\left|j^{\varepsilon}_t-j^{\varepsilon}_{t(\delta)}\right|^4_H\left(2\left(\theta^{\varepsilon}_t,\sigma_2(j^{\varepsilon}_t,\theta^{\varepsilon}_t,z)\right)+\left|\sigma_2(j^{\varepsilon}_t,\theta^{\varepsilon}_t,z)\right|^2_H\right)^2\eta_2(\frac{1}{\varepsilon}dt,dz)\right)^{\frac{1}{2}}\\
	&+\mathbb{E}\int_0^{\tau _{M,L}\wedge T}\int_{\left|z\right|_H<1}\left|j^{\varepsilon}_t-j^{\varepsilon}_{t(\delta)}\right|^2_H\left|\sigma_2(j^{\varepsilon}_t,\theta^{\varepsilon}_t,z)\right|^2_H\nu_2(dz)dt\\
	\leq &C_{M,L}\mathbb{E}\int_0^{\tau _{M,L}\wedge T}\left|j^{\varepsilon}_t-j^{\varepsilon}_{t(\delta)}\right|^2_Hdt\\
	\leq & C_{T,M,L}\delta^{\frac{1}{2p}}.
\end{align*}
Hence
\begin{align*}
	\mathbb{E}\int_0^{\tau _{M,L}\wedge T}\left|\hat{\theta}^{\varepsilon}_t-\theta^{\varepsilon}_t\right|^{2p}_Hdt\leq C_{p,T}\kappa\left(C_{T,M}\delta^{\frac{1}{2}}\right)+C_{p,T,M,L}\delta^{\frac{1}{2p}}.
\end{align*}
\end{proof}
Furthermore, we can get the following Lemma:
\begin{lemma}\label{l7}
Let $\sigma _1(j,\theta,z)\equiv \sigma _1(j,z)$	, for any $T>0$, $\left ( j_0,\theta_0\right )\in H\times H$ and for some $p\geq 1$, we have
\begin{align}\label{e26}
\nonumber	&\mathbb{E}\sup_{t\in[0,\tau _{M,L}\wedge T]}\left|j^{\varepsilon}_t-\hat{j}^{\varepsilon}_t\right|^{2p}_H+\mathbb{E}\int_0^{\tau _{M,L}\wedge T}\left|j^{\varepsilon}_s-\hat{j}^{\varepsilon}_s\right|^{2p-2}_H\left|\triangledown\left(j^{\varepsilon}_s-\hat{j}^{\varepsilon}_s\right)\right|^2_Hds\\
	\leq &\mathcal{O}\left(\delta\right)+C_{p,T,M}\kappa\left(C_{T,M}\delta^{\frac{1}{2}}+\mathcal{O}\left(\delta\right)\right).
\end{align}
\end{lemma}
\begin{proof}
	Due to 
\begin{align*}
	\left(j^{\varepsilon}_t-\hat{j}^{\varepsilon}_t\right)&+\int_0^t\left(u^{\varepsilon}_s\cdot \triangledown\right)\left(j^{\varepsilon}_s-\hat{j}^{\varepsilon}_s\right)=\int_0^t\triangle\left(j^{\varepsilon}_s-\hat{j}^{\varepsilon}_s\right)ds\\
	&+\int_0^tf(j^{\varepsilon}_s,\theta^{\varepsilon}_s)-f(j^{\varepsilon}_{s(\delta)},\hat{\theta}^{\varepsilon}_s)ds+\int_0^t\partial_x\theta^{\varepsilon}_s-\partial_x\hat{\theta}^{\varepsilon}_sds,
\end{align*}
we have
\begin{align*}
	\left|j^{\varepsilon}_t-\hat{j}^{\varepsilon}_t\right|^{2p}_H=&2p\int_0^t\left|j^{\varepsilon}_s-\hat{j}^{\varepsilon}_s\right|^{2p-2}_H\left \langle\triangle\left(j^{\varepsilon}_s-\hat{j}^{\varepsilon}_s\right),j^{\varepsilon}_s-\hat{j}^{\varepsilon}_s  \right \rangle ds\\
	&+2p\int_0^t\left|j^{\varepsilon}_s-\hat{j}^{\varepsilon}_s\right|^{2p-2}_H\left(f(j^{\varepsilon}_s,\theta^{\varepsilon}_s)-f(j^{\varepsilon}_{s(\delta)},\hat{\theta}^{\varepsilon}_s),j^{\varepsilon}_s-\hat{j}^{\varepsilon}_s\right)ds\\
	&+2p\int_0^t\left|j^{\varepsilon}_s-\hat{j}^{\varepsilon}_s\right|^{2p-2}_H\left(\partial_x\theta^{\varepsilon}_s-\partial_x\hat{\theta}^{\varepsilon}_s ,j^{\varepsilon}_s-\hat{j}^{\varepsilon}_s\right)ds,
\end{align*}
and
\begin{align*}
	&\left|j^{\varepsilon}_t-\hat{j}^{\varepsilon}_t\right|^{2p}_H+2p\int_0^t\left|j^{\varepsilon}_s-\hat{j}^{\varepsilon}_s\right|^{2p-2}_H\left|\triangledown\left(j^{\varepsilon}_s-\hat{j}^{\varepsilon}_s\right)\right|^2_Hds\\
	=&2p\int_0^t\left|j^{\varepsilon}_s-\hat{j}^{\varepsilon}_s\right|^{2p-2}_H\left(f(j^{\varepsilon}_s,\theta^{\varepsilon}_s)-f(j^{\varepsilon}_{s(\delta)},\hat{\theta}^{\varepsilon}_s),j^{\varepsilon}_s-\hat{j}^{\varepsilon}_s\right)ds\\
	&-2p\int_0^t\left|j^{\varepsilon}_s-\hat{j}^{\varepsilon}_s\right|^{2p-2}_H\left(\theta^{\varepsilon}_s-\hat{\theta}^{\varepsilon}_s ,\partial_xj^{\varepsilon}_s-\partial_x\hat{j}^{\varepsilon}_s\right)ds.
\end{align*}
Therefore, using Young's inequality, we have
\begin{align*}
	&\mathbb{E}\sup_{t\in[0,\tau _{M,L}\wedge T]}\left|j^{\varepsilon}_t-\hat{j}^{\varepsilon}_t\right|^{2p}_H+2p\mathbb{E}\int_0^{\tau _{M,L}\wedge T}\left|j^{\varepsilon}_s-\hat{j}^{\varepsilon}_s\right|^{2p-2}_H\left|\triangledown\left(j^{\varepsilon}_s-\hat{j}^{\varepsilon}_s\right)\right|^2_Hds\\
	\leq & C_p\mathbb{E}\int_0^{\tau _{M,L}\wedge T}\left|j^{\varepsilon}_s-\hat{j}^{\varepsilon}_s\right|^{2p}_Hds+C_p\mathbb{E}\int_0^{\tau _{M,L}\wedge T}\left|f(j^{\varepsilon}_s,\theta^{\varepsilon}_s)-f(j^{\varepsilon}_{s(\delta)},\hat{\theta}^{\varepsilon}_s)\right|^{2p}_{H}ds\\
	&+p\mathbb{E}\int_0^{\tau _{M,L}\wedge T}\left|j^{\varepsilon}_s-\hat{j}^{\varepsilon}_s\right|^{2p-2}_H\left|\triangledown\left(j^{\varepsilon}_s-\hat{j}^{\varepsilon}_s\right)\right|^2_Hds+C_p\mathbb{E}\int_0^{\tau _{M,L}\wedge T}\left|\theta^{\varepsilon}_s-\hat{\theta}^{\varepsilon}_s\right|^{2p}_Hds.
\end{align*}
According to (\ref{e8.5.5}) and (\ref{e25}), we have 
\begin{align*}
	&\mathbb{E}\sup_{t\in[0,\tau _{M,L}\wedge T]}\left|j^{\varepsilon}_t-\hat{j}^{\varepsilon}_t\right|^{2p}_H+p\mathbb{E}\int_0^{\tau _{M,L}\wedge T}\left|j^{\varepsilon}_s-\hat{j}^{\varepsilon}_s\right|^{2p-2}_H\left|\triangledown\left(j^{\varepsilon}_s-\hat{j}^{\varepsilon}_s\right)\right|^2_Hds\\
	\le & C_p\mathbb{E}\int_0^{\tau _{M,L}\wedge T}\left|j^{\varepsilon}_s-\hat{j}^{\varepsilon}_s\right|^{2p}_Hds+C_p\mathbb{E}\int_0^{\tau _{M,L}\wedge T}\left|\theta^{\varepsilon}_s-\hat{\theta}^{\varepsilon}_s\right|^{2p}_Hds\\
	&+C_p\mathbb{E}\int_0^{\tau _{M,L}\wedge T}\kappa\left(\left|j^{\varepsilon}_s-j^{\varepsilon}_{s(\delta)}\right|^{2p}_H+\left|\theta^{\varepsilon}_s-\hat{\theta}^{\varepsilon}_s\right|^{2p}_H\right)ds\\
	\leq &C_p\int_0^{\tau _{M,L}\wedge T}\mathbb{E}\sup_{r\in[0,\tau _{M,L}\wedge s]}\left|j^{\varepsilon}_r-\hat{j}^{\varepsilon}_r\right|^{2p}_Hds+\mathcal{O}_{p,T,M,L}\left(\delta\right)\\
	&+ C_{p,T}\kappa\left(\frac{1}{\tau _{M,L}\wedge T}\int_0^{\tau _{M,L}\wedge T}\left|j^{\varepsilon}_s-j^{\varepsilon}_{s(\delta)}\right|^{2p}_H+\left|\theta^{\varepsilon}_s-\hat{\theta}^{\varepsilon}_s\right|^{2p}_Hds\right)\\
	\leq &C_p\int_0^{\tau _{M,L}\wedge T}\mathbb{E}\sup_{r\in[0,\tau _{M,L}\wedge s]}\left|j^{\varepsilon}_r-\hat{j}^{\varepsilon}_r\right|^{2p}_Hds+\mathcal{O}\left(\delta\right)+C_{p,T,M}\kappa\left(C_{T,M}\delta^{\frac{1}{2}}+\mathcal{O}\left(\delta\right)\right).
\end{align*}
Gronwall's inequality implies that 
\begin{align*}
	&\mathbb{E}\sup_{t\in[0,\tau _{M,L}\wedge T]}\left|j^{\varepsilon}_t-\hat{j}^{\varepsilon}_t\right|^{2p}_H+\mathbb{E}\int_0^{\tau _{M,L}\wedge T}\left|j^{\varepsilon}_s-\hat{j}^{\varepsilon}_s\right|^{2p-2}_H\left|\triangledown\left(j^{\varepsilon}_s-\hat{j}^{\varepsilon}_s\right)\right|^2_Hds\\
	\leq & \mathcal{O}\left(\delta\right)+C_{p,T,M}\kappa\left(C_{T,M}\delta^{\frac{1}{2}}+\mathcal{O}\left(\delta\right)\right).
\end{align*}
\end{proof}

Next, we present the estimate of $\hat{j}^{\varepsilon}_t-\bar{j}_t$, which is the following Lemma:
\begin{lemma}\label{l8}
Let $\sigma _1(j,\theta,z)\equiv \sigma _1(j,z)$	, for any $T>0$, $\left ( j_0,\theta_0\right )\in H\times H$ and for some $p\geq 1$, we have
\begin{align}\label{e27}
	\mathbb{E}\sup_{t\in[0,\tau_{M,L,N}\wedge T]}\left|\hat{j}^{\varepsilon}_t-\bar{j}_t\right|^{2p}_H\leq \hat{\Omega}^{-1}\left(\hat{\Omega}\left(\mathcal{R}(\varepsilon,\delta)\right)+C_{p,T,N}T\right),
\end{align}
where $\tau _{M,L,N}:=\tau_{M,L} \wedge \inf\left\{t\geq 0; \left|\hat{j}^{\varepsilon}_t\right|_H+\left|\bar{j}_t\right|_H>N\right\}$. The definitions of $\mathcal{R}(\varepsilon,\delta)$ and $\Omega(\cdot)$ are given in (\ref{e28}) and (\ref{e29}).
\end{lemma}
\begin{proof}
By It\^{o}'s formula, we have
\begin{align*}
&\left|\hat{j}^{\varepsilon}_t-\bar{j}_t\right|^{2p}_H+2p\int_0^t\left|\hat{j}^{\varepsilon}_s-\bar{j}_s\right|^{2p-2}_H\left(\left(u^{\varepsilon}_s\cdot \triangledown\right)\hat{j}^{\varepsilon}_s-\left(\bar{u}_s\cdot \triangledown\right)\bar{j}_s,\hat{j}^{\varepsilon}_s-\bar{j}_s\right)ds\\
=& 2p\int_0^t\left|\hat{j}^{\varepsilon}_s-\bar{j}_s\right|^{2p-2}_H\left \langle \triangle\left(\hat{j}^{\varepsilon}_s-\bar{j}_s\right),\hat{j}^{\varepsilon}_s-\bar{j}_s \right \rangle ds\\
&+2p\int_0^t\left|\hat{j}^{\varepsilon}_s-\bar{j}_s\right|^{2p-2}_H\left(f(j^{\varepsilon}_{s(\delta)},\hat{\theta}^{\varepsilon}_s)-\bar{f}(\bar{j}_s,\bar{u}_s),\hat{j}^{\varepsilon}_s-\bar{j}_s\right)ds\\
&+2p\int_0^t\left|\hat{j}^{\varepsilon}_s-\bar{j}_s\right|^{2p-2}_H\left(\partial_x\hat{\theta}^{\varepsilon}_s-\partial_xg(\bar{u}_s) ,\hat{j}^{\varepsilon}_s-\bar{j}_s\right)ds\\
&+2p\int_0^t\int_{\left|z\right|_H<1}\left|\hat{j}^{\varepsilon}_s-\bar{j}_s\right|^{2p-2}_H\left(\sigma_1(j^{\varepsilon}_s,z)-\sigma_1(\bar{j}_s,z),\hat{j}^{\varepsilon}_s-\bar{j}_s\right)\tilde{\eta}_1(ds,dz)\\
&+\int_0^t\int_{\left|z\right|_H<1}\left[\left|(\hat{j}^{\varepsilon}_s-\bar{j}_s)+(\sigma_1(j^{\varepsilon}_s,z)-\sigma_1(\bar{j}_s,z))\right|^{2p}_H-\left|\hat{j}^{\varepsilon}_s-\bar{j}_s\right|^{2p}_H\right.\\
&\left.-2p\left|\hat{j}^{\varepsilon}_s-\bar{j}_s\right|^{2p-2}_H\left(\sigma_1(j^{\varepsilon}_s,z)-\sigma_1(\bar{j}_s,z),\hat{j}^{\varepsilon}_s-\bar{j}_s\right)\right]\eta_1(ds,dz).
\end{align*}
Therefore
\begin{align*}
	&\left|\hat{j}^{\varepsilon}_t-\bar{j}_t\right|^{2p}_H+2p\int_0^t\left|\hat{j}^{\varepsilon}_s-\bar{j}_s\right|^{2p-2}_H\left|\triangledown(\hat{j}^{\varepsilon}_s-\bar{j}_s)\right|^2_Hds\\
	\leq &2p\int_0^t\left|\hat{j}^{\varepsilon}_s-\bar{j}_s\right|^{2p-2}_H\left(\left(\left(u^{\varepsilon}_s-\hat{u}^{\varepsilon}_s\right)\cdot \triangledown\right)\left(\hat{j}^{\varepsilon}_s-\bar{j}_s\right),\bar{j}_s\right)ds\\
	&+2p\int_0^t\left|\hat{j}^{\varepsilon}_s-\bar{j}_s\right|^{2p-2}_H\left(\left(\left(\hat{u}^{\varepsilon}_s-\bar{u}_s\right)\cdot \triangledown\right)\left(\hat{j}^{\varepsilon}_s-\bar{j}_s\right),\bar{j}_s\right)ds\\
	&+C_p\int_0^t\left|\hat{j}^{\varepsilon}_s-\bar{j}_s\right|^{2p}_Hds+C_p\int_0^t\left|f(j^{\varepsilon}_{s(\delta)},\hat{\theta}^{\varepsilon}_s)-\bar{f}(\bar{j}_s,\bar{u}_s)\right|^{2p}_Hds\\
	&+p\int_0^t\left|\hat{j}^{\varepsilon}_s-\bar{j}_s\right|^{2p-2}_H\left|\triangledown(\hat{j}^{\varepsilon}_s-\bar{j}_s)\right|^2_Hds+C_p\int_0^t\left|\hat{\theta}^{\varepsilon}_s-g(\bar{u}_s) \right|^{2p}_Hds\\
	&+2p\int_0^t\int_{\left|z\right|_H<1}\left|\hat{j}^{\varepsilon}_s-\bar{j}_s\right|^{2p-2}_H\left(\sigma_1(j^{\varepsilon}_s,z)-\sigma_1(\bar{j}_s,z),\hat{j}^{\varepsilon}_s-\bar{j}_s\right)\tilde{\eta}_1(ds,dz)\\
	&+M_p\int_0^t\int_{\left|z\right|_H<1}\left(\left|\hat{j}^{\varepsilon}_s-\bar{j}_s\right|^{2p-2}_H\left|\sigma_1(j^{\varepsilon}_s,z)-\sigma_1(\bar{j}_s,z)\right|^2_H+\left|\hat{j}^{\varepsilon}_s-\bar{j}_s\right|^{2p}_H\right)\eta_1(ds,dz).
\end{align*}
According to (\ref{e8.5}) and Young's inequality, we have
\begin{align*}
	&\left|\hat{j}^{\varepsilon}_t-\bar{j}_t\right|^{2p}_H+p\int_0^t\left|\hat{j}^{\varepsilon}_s-\bar{j}_s\right|^{2p-2}_H\left|\triangledown(\hat{j}^{\varepsilon}_s-\bar{j}_s)\right|^2_Hds\\
	\leq & 2p\int_0^t\left|\hat{j}^{\varepsilon}_s-\bar{j}_s\right|^{2p-2}_H\left|j^{\varepsilon}_s-\hat{j}^{\varepsilon}_s\right|_{H^1}\left|\hat{j}^{\varepsilon}_s-\bar{j}_s\right|_{H^1}\left|\bar{j}_s\right|_Hds\\
	&+2p\int_0^t\left|\hat{j}^{\varepsilon}_s-\bar{j}_s\right|^{2p-2}_H\left|\hat{j}^{\varepsilon}_s-\bar{j}_s\right|_H\left|\hat{j}^{\varepsilon}_s-\bar{j}_s\right|_{H^1}\left|\bar{j}_s\right|_{H^1}ds\\
	&+C_p\int_0^t\left|\hat{j}^{\varepsilon}_s-\bar{j}_s\right|^{2p}_Hds+C_p\int_0^t\left|f(j^{\varepsilon}_{s(\delta)},\hat{\theta}^{\varepsilon}_s)-\bar{f}(\bar{j}_s,\bar{u}_s)\right|^{2p}_Hds+C_p\int_0^t\left|\hat{\theta}^{\varepsilon}_s-g(\bar{u}_s) \right|^{2p}_Hds\\
	&+2p\int_0^t\int_{\left|z\right|_H<1}\left|\hat{j}^{\varepsilon}_s-\bar{j}_s\right|^{2p-2}_H\left(\sigma_1(j^{\varepsilon}_s,z)-\sigma_1(\bar{j}_s,z),\hat{j}^{\varepsilon}_s-\bar{j}_s\right)\tilde{\eta}_1(ds,dz)\\
	&+C_p\int_0^t\int_{\left|z\right|_H<1}\left|\sigma_1(j^{\varepsilon}_s,z)-\sigma_1(\bar{j}_s,z)\right|^{2p}_H\eta_1(ds,dz)\\
	\leq &2p\gamma\int_0^t\left|\hat{j}^{\varepsilon}_s-\bar{j}_s\right|^{2p-2}_H\left|\hat{j}^{\varepsilon}_s-\bar{j}_s\right|^2_{H^1}ds+C_{p,\gamma}\int_0^t\left|\hat{j}^{\varepsilon}_s-\bar{j}_s\right|^{2p-2}_H\left|j^{\varepsilon}_s-\hat{j}^{\varepsilon}_s\right|^2_{H^1}\left|\bar{j}_s\right|^2_Hds\\
	&+C_{p,\gamma}\int_0^t\left|\hat{j}^{\varepsilon}_s-\bar{j}_s\right|^{2p}_H\left|\bar{j}_s\right|^2_{H^1}ds+C_p\int_0^t\left|\hat{j}^{\varepsilon}_s-\bar{j}_s\right|^{2p}_Hds\\
	&+C_p\int_0^t\left|f(j^{\varepsilon}_{s(\delta)},\hat{\theta}^{\varepsilon}_s)-\bar{f}(\bar{j}_s,\bar{u}_s)\right|^{2p}_Hds+C_p\int_0^t\left|\hat{\theta}^{\varepsilon}_s-g(\bar{u}_s) \right|^{2p}_Hds\\
	&+2p\int_0^t\int_{\left|z\right|_H<1}\left|\hat{j}^{\varepsilon}_s-\bar{j}_s\right|^{2p-2}_H\left(\sigma_1(j^{\varepsilon}_s,z)-\sigma_1(\bar{j}_s,z),\hat{j}^{\varepsilon}_s-\bar{j}_s\right)\tilde{\eta}_1(ds,dz)\\
	&+C_p\int_0^t\int_{\left|z\right|_H<1}\left|\sigma_1(j^{\varepsilon}_s,z)-\sigma_1(\bar{j}_s,z)\right|^{2p}_H\eta_1(ds,dz).
\end{align*}
Choose $\gamma$ small enough, we get
\begin{align*}
	\left|\hat{j}^{\varepsilon}_t-\bar{j}_t\right|^{2p}_H\leq & C_p\int_0^t\left|\hat{j}^{\varepsilon}_s-\bar{j}_s\right|^{2p}_Hds+C_p\left\{\int_0^t\left|\hat{j}^{\varepsilon}_s-\bar{j}_s\right|^{2p-2}_H\left|j^{\varepsilon}_s-\hat{j}^{\varepsilon}_s\right|^2_{H^1}\left|\bar{j}_s\right|^2_Hds\right\}_{Q_1}\\
	&+C_p\left\{\int_0^t\left|\hat{j}^{\varepsilon}_s-\bar{j}_s\right|^{2p}_H\left|\bar{j}_s\right|^2_{H^1}ds\right\}_{Q_2}+C_p\left\{\int_0^t\left|f(j^{\varepsilon}_{s(\delta)},\hat{\theta}^{\varepsilon}_s)-\bar{f}(\bar{j}_s,\bar{u}_s)\right|^{2p}_Hds\right\}_{Q_3}\\
	&+C_p\left\{\int_0^t\left|\hat{\theta}^{\varepsilon}_s-g(\bar{u}_s) \right|^{2p}_Hds\right\}_{Q_4}+C_p\left\{\int_0^t\int_{\left|z\right|_H<1}\left|\sigma_1(j^{\varepsilon}_s,z)-\sigma_1(\bar{j}_s,z)\right|^{2p}_H\eta_1(ds,dz)\right\}_{Q_5}\\
	&+2p\left\{\int_0^t\int_{\left|z\right|_H<1}\left|\hat{j}^{\varepsilon}_s-\bar{j}_s\right|^{2p-2}_H\left(\sigma_1(j^{\varepsilon}_s,z)-\sigma_1(\bar{j}_s,z),\hat{j}^{\varepsilon}_s-\bar{j}_s\right)\tilde{\eta}_1(ds,dz)\right\}_{Q_6}.
\end{align*}

We will discuss $Q_1$ to $Q_{6}$ one by one. Specifically, for $Q_1$, according to (\ref{e26}), we have
\begin{align*}
 \mathbb{E}\sup_{t\in[0,\tau_{M,L,N}\wedge T]}Q_1\leq C_{p,T,N}\mathcal{O}\left(\delta\right)+C_{p,T,M,N}\kappa\left(C_{T,M}\delta^{\frac{1}{2}}+\mathcal{O}\left(\delta\right)\right).
\end{align*}
For $Q_2$, note that
\begin{align*}
\frac{d}{dt}\left|\bar{j}_t\right|^2_H+2\left|\triangledown \bar{j}_t\right|^2_H=&2\left(\bar{f}(\bar{j}_t,\bar{u}_t),\bar{j}_t\right)+2\left(\partial_x g(\bar{u}) ,\bar{j}_t\right)+2\int_{\left|z\right|_H<1}\left(\bar{j}_t,\sigma_1(\bar{j}_t,z)\right)\dot{\tilde{\eta}}_1(t,dz)\\
&+\int_{\left|z\right|_H<1}\left|\sigma_1(\bar{j}_t,z)\right|^2_H\dot{\eta}_1(t,dz),
\end{align*}
we have
\begin{align*}
	\left \langle \left|\bar{j}_t\right|^2_H \right \rangle _t=\int_0^t\int_{\left|z\right|_H<1}\left(2\left(\bar{j}_t,\sigma_1(\bar{j}_t,z)\right)+\left|\sigma_1(\bar{j}_t,z)\right|^2_H\right)^2\eta_1(ds,dz),
\end{align*}
and
\begin{align*}
\frac{d}{dt}\left|\bar{j}_t\right|^2_H+\left|\triangledown \bar{j}_t\right|^2_H	\leq & C\left(1+\left|\bar{j}_t\right|^2_H\right)+2\int_{\left|z\right|_H<1}\left(\bar{j}_t,\sigma_1(\bar{j}_t,z)\right)\dot{\tilde{\eta}}_1(t,dz)\\
&+\int_{\left|z\right|_H<1}\left|\sigma_1(\bar{j}_t,z)\right|^2_H\dot{\eta}_1(t,dz),
\end{align*}
and so
\begin{align*}
	&\mathbb{E}\int_0^{\tau_{M,L,N}\wedge T}\left|\hat{j}^{\varepsilon}_s-\bar{j}_s\right|^{2p}_H\left|\triangledown\bar{j}_s\right|^2_Hds\\
	\leq & C\mathbb{E}\int_0^{\tau_{M,L,N}\wedge T}\left(1+\left|\bar{j}_s\right|^2_H\right)\left|\hat{j}^{\varepsilon}_s-\bar{j}_s\right|^{2p}_Hds\\
	&+2\mathbb{E}\int_0^{\tau_{M,L,N}\wedge T}\int_{\left|z\right|_H<1}\left(\bar{j}_s,\sigma_1(\bar{j}_s,z)\right)\left|\hat{j}^{\varepsilon}_s-\bar{j}_s\right|^{2p}_H\tilde{\eta}_1(ds,dz)\\
	&+\mathbb{E}\int_0^{\tau_{M,L,N}\wedge T}\int_{\left|z\right|_H<1}\left|\sigma_1(\bar{j}_s,z)\right|^2_H\left|\hat{j}^{\varepsilon}_s-\bar{j}_s\right|^{2p}_H\eta_1(ds,dz)-\mathbb{E}\int_0^{\tau_{M,L,N}\wedge T}\left|\hat{j}^{\varepsilon}_s-\bar{j}_s\right|^{2p}_Hd\left|\bar{j}_s\right|^2_H\\
	\leq & C_N\mathbb{E}\int_0^{\tau_{M,L,N}\wedge T}\left|\hat{j}^{\varepsilon}_s-\bar{j}_s\right|^{2p}_Hds+C\mathbb{E}\left(\int_0^{\tau_{M,L,N}\wedge T}\int_{\left|z\right|_H<1}\left|\hat{j}^{\varepsilon}_s-\bar{j}_s\right|^{4p}_H\left|\bar{j}_s\right|^2_H\left|\sigma_1(\bar{j}_s,z)\right|^2_H\eta_1(ds,dz)\right)^{\frac{1}{2}}\\
	&+C\mathbb{E}\left(\int_0^{\tau_{M,L,N}\wedge T}\int_{\left|z\right|_H<1}\left|\hat{j}^{\varepsilon}_s-\bar{j}_s\right|^{4p}_Hd\left \langle \left|\bar{j}_s\right|^2_H \right \rangle _s\right)^{\frac{1}{2}}\\
	\leq & C_N\mathbb{E}\int_0^{\tau_{M,L,N}\wedge T}\left|\hat{j}^{\varepsilon}_s-\bar{j}_s\right|^{2p}_Hds\\
	&+C\mathbb{E}\left(\int_0^{\tau_{M,L,N}\wedge T}\int_{\left|z\right|_H<1}\left|\hat{j}^{\varepsilon}_s-\bar{j}_s\right|^{4p}_H\left(2\left(\bar{j}_s,\sigma_1(\bar{j}_s,z)\right)+\left|\sigma_1(\bar{j}_s,z)\right|^2_H\right)^2\eta_1(ds,dz)\right)^{\frac{1}{2}}\\
	\leq & C_N\mathbb{E}\int_0^{\tau_{M,L,N}\wedge T}\left|\hat{j}^{\varepsilon}_s-\bar{j}_s\right|^{2p}_Hds.
\end{align*}
Hence
\begin{align*}
	\mathbb{E}\sup_{t\in[0,\tau_{M,L,N}\wedge T]}Q_2\leq  C_N\mathbb{E}\int_0^{\tau_{M,L,N}\wedge T}\left|\hat{j}^{\varepsilon}_s-\bar{j}_s\right|^{2p}_Hds.
\end{align*}
For $Q_3$, according to Proposition \ref{p1}, (\ref{e8.5.5}) and (\ref{e26}), we have
\begin{align*}
	\mathbb{E}\sup_{t\in[0,\tau_{M,L,N}\wedge T]}Q_3\leq & C_p\mathbb{E}\sup_{t\in[0,\tau_{M,L,N}\wedge T]}\int_0^t\left|f(j^{\varepsilon}_{s(\delta)},\hat{\theta}^{\varepsilon}_s)-\bar{f}(j^{\varepsilon}_{s(\delta)},u^{\varepsilon}_{s(\delta)})\right|^{2p}_Hds\\
	&+C_p\mathbb{E}\sup_{t\in[0,\tau_{M,L,N}\wedge T]}\int_0^t\left|\bar{f}(j^{\varepsilon}_{s(\delta)},u^{\varepsilon}_{s(\delta)})-\bar{f}(j^{\varepsilon}_s,u^{\varepsilon}_s)\right|^{2p}_Hds\\
	&+C_p\mathbb{E}\sup_{t\in[0,\tau_{M,L,N}\wedge T]}\int_0^t\left|\bar{f}(j^{\varepsilon}_s,u^{\varepsilon}_s)-\bar{f}(\hat{j}^{\varepsilon}_s,\hat{u}^{\varepsilon}_s)\right|^{2p}_Hds\\
	&+C_p\mathbb{E}\sup_{t\in[0,\tau_{M,L,N}\wedge T]}\int_0^t\left|\bar{f}(\hat{j}^{\varepsilon}_s,\hat{u}^{\varepsilon}_s)-\bar{f}(\bar{j}_s,\bar{u}_s)\right|^{2p}_Hds\\
	\leq & C_p\mathbb{E}\sup_{t\in[0,\tau_{M,L,N}\wedge T]}\int_0^t\left|f(j^{\varepsilon}_{s(\delta)},\hat{\theta}^{\varepsilon}_s)-\bar{f}(j^{\varepsilon}_{s(\delta)},u^{\varepsilon}_{s(\delta)})\right|^{2p}_Hds\\
	&+C_p\mathbb{E}\int_0^{\tau_{M,L,N}\wedge T}\tilde{\kappa}\left(\left|j^{\varepsilon}_s-j^{\varepsilon}_{s(\delta)}\right|^{2p}_H\right) ds+C_p\mathbb{E}\int_0^{\tau_{M,L,N}\wedge T}\tilde{\kappa}\left(\left|j^{\varepsilon}_s-\hat{j}^{\varepsilon}_s\right|^{2p}_H\right) ds\\
	&+C_p\mathbb{E}\int_0^{\tau_{M,L,N}\wedge T}\tilde{\kappa}\left(\left|\hat{j}^{\varepsilon}_s-\bar{j}_s\right|^{2p}_H\right) ds\\
	\leq & C_{p,T}\int_0^{\tau_{M,L,N}\wedge T}\tilde{\kappa}\left(\mathbb{E}\sup_{r\in[0,\tau_{M,L,N}\wedge s]}\left|\hat{j}^{\varepsilon}_s-\bar{j}_s\right|^{2p}_H\right) ds\\
	&+C_p\mathbb{E}\sup_{t\in[0,\tau_{M,L,N}\wedge T]}\int_0^t\left|f(j^{\varepsilon}_{s(\delta)},\hat{\theta}^{\varepsilon}_s)-\bar{f}(j^{\varepsilon}_{s(\delta)},u^{\varepsilon}_{s(\delta)})\right|^{2p}_Hds\\
	&+C_{p,T}\tilde{\kappa}\left(\frac{1}{\tau_{M,L,N}\wedge T}\int_0^{\tau_{M,L,N}\wedge T}\left|j^{\varepsilon}_s-j^{\varepsilon}_{s(\delta)}\right|^{2p}_Hds\right)\\
	&+C_{p,T}\tilde{\kappa}\left(\frac{1}{\tau_{M,L,N}\wedge T}\int_0^{\tau_{M,L,N}\wedge T}\left|j^{\varepsilon}_s-\hat{j}^{\varepsilon}_s\right|^{2p}_Hds\right)\\
	\leq & C_{p,T}\int_0^{\tau_{M,L,N}\wedge T}\tilde{\kappa}\left(\mathbb{E}\sup_{r\in[0,\tau_{M,L,N}\wedge s]}\left|\hat{j}^{\varepsilon}_s-\bar{j}_s\right|^{2p}_H\right) ds\\
	&+C_p\left\{\mathbb{E}\sup_{t\in[0,\tau_{M,L,N}\wedge T]}\int_0^t\left|f(j^{\varepsilon}_{s(\delta)},\hat{\theta}^{\varepsilon}_s)-\bar{f}(j^{\varepsilon}_{s(\delta)},u^{\varepsilon}_{s(\delta)})\right|^{2p}_Hds\right\}_{Q_{3,1}}\\
	&+C_{p,T}\tilde{\kappa}\left(\mathcal{O}\left(\delta\right)+C_{p,T,M}\kappa\left(C_{T,M}\delta^{\frac{1}{2}}+\mathcal{O}\left(\delta\right)\right)\right)\\
	&+C_{p,T}\tilde{\kappa}\left(C_{T,M}\delta^{\frac{1}{2}}\right).
\end{align*}
For $Q_{3,1}$, we have
\begin{align*}
	Q_{3,1}=&\mathbb{E}\sup_{t\in[0,\tau_{M,L,N}\wedge T]}\sum_{k=0}^{\left[\frac{t}{\delta}\right]-1}\int_{k\delta}^{(k+1)\delta}\left|f(j^{\varepsilon}_{k\delta},\hat{\theta}^{\varepsilon}_s)-\bar{f}(j^{\varepsilon}_{k\delta},u^{\varepsilon}_{k\delta})\right|^{2p}_Hds\\
	&+\mathbb{E}\sup_{t\in[0,\tau_{M,L,N}\wedge T]}\int_{\left[\frac{t}{\delta}\right]\delta}^t\left|f(j^{\varepsilon}_{\left[\frac{t}{\delta}\right]\delta},\hat{\theta}^{\varepsilon}_s)-\bar{f}(j^{\varepsilon}_{\left[\frac{t}{\delta}\right]\delta},u^{\varepsilon}_{\left[\frac{t}{\delta}\right]\delta})\right|^{2p}_Hds\\
	\leq &\left[\frac{T}{\delta}\right]\max_{0\leq k\leq \left[\frac{t}{\delta}\right]-1}\mathbb{E}\int_{k\delta}^{(k+1)\delta}\left|f(j^{\varepsilon}_{k\delta},\hat{\theta}^{\varepsilon}_s)-\bar{f}(j^{\varepsilon}_{k\delta},u^{\varepsilon}_{k\delta})\right|^{2p}_Hds+C_T\delta\\
	\overset{\mathrm{s=r +k\delta }}{\leq}&\left[\frac{T}{\delta}\right]\max_{0\leq k\leq \left[\frac{t}{\delta}\right]-1}\mathbb{E}\int_{0}^{\delta}\left|f(j^{\varepsilon}_{k\delta},\hat{\theta}^{\varepsilon}_{r +k\delta})-\bar{f}(j^{\varepsilon}_{k\delta},u^{\varepsilon}_{k\delta})\right|^{2p}_Hdr+C_T\delta.
\end{align*}
Note that $\hat{\theta}^{\varepsilon}_{r +k\delta}=\tilde{\theta}^{u^{\varepsilon}_{k\delta},\theta_0}_{\frac{r +k\delta}{\varepsilon}}$, by (\ref{e11}), we have
\begin{align*}
	&\mathbb{E}\int_{0}^{\delta}\left|f(j^{\varepsilon}_{k\delta},\hat{\theta}^{\varepsilon}_{r +k\delta})-\bar{f}(j^{\varepsilon}_{k\delta},u^{\varepsilon}_{k\delta})\right|^{2p}_Hdr\\
	=&\mathbb{E}\int_{0}^{\delta}\left|f(j^{\varepsilon}_{k\delta},\tilde{\theta}^{u^{\varepsilon}_{k\delta},\theta_0}_{\frac{r +k\delta}{\varepsilon}})-\int_H f(j^{\varepsilon}_{k\delta},\theta)\pi^{u^{\varepsilon}_{k\delta}}(d\theta)\right|^{2p}_Hdr\\
	\leq &C_p\mathbb{E}\int_{0}^{\delta}\int_H \left|f(j^{\varepsilon}_{k\delta},\tilde{\theta}^{u^{\varepsilon}_{k\delta},\theta_0}_{\frac{r +k\delta}{\varepsilon}})-f(j^{\varepsilon}_{k\delta},\theta)\right|^{2p}_H\pi^{u^{\varepsilon}_{k\delta}}(d\theta)dr\\
	=& C_p\mathbb{E}\int_{0}^{\delta}\int_H \left|f(j^{\varepsilon}_{k\delta},\tilde{\theta}^{u^{\varepsilon}_{k\delta},\theta_0}_{\frac{r +k\delta}{\varepsilon}})-f(j^{\varepsilon}_{k\delta},\tilde{\theta}^{u^{\varepsilon}_{k\delta},\theta}_{\frac{r +k\delta}{\varepsilon}})\right|^{2p}_H\pi^{u^{\varepsilon}_{k\delta}}(d\theta)dr\\
	\leq & C_p\int_{0}^{\delta}\int_H \kappa\left(\mathbb{E}\left|\tilde{\theta}^{u^{\varepsilon}_{k\delta},\theta_0}_{\frac{r +k\delta}{\varepsilon}}-\tilde{\theta}^{u^{\varepsilon}_{k\delta},\theta}_{\frac{r +k\delta}{\varepsilon}}\right|^{2p}_H\right)\pi^{u^{\varepsilon}_{k\delta}}(d\theta)dr\\
	\leq &C_p\int_{0}^{\delta}\int_H\kappa\left(e^{-\frac{\lambda_p(r+k\delta)}{\varepsilon}}\left|\theta_0-\theta\right|^{2p}_H\right)\pi^{u^{\varepsilon}_{k\delta}}(d\theta)dr\\
	\leq &C_p\int_{0}^{\delta}\kappa\left(e^{-\frac{\lambda_p(r+k\delta)}{\varepsilon}}\int_H\left|\theta_0-\theta\right|^{2p}_H\pi^{u^{\varepsilon}_{k\delta}}(d\theta)\right)dr\\
	\leq & C_p\int_{0}^{\delta}\kappa\left(Ce^{-\frac{\lambda_p(r+k\delta)}{\varepsilon}}\right)dr\\
	\leq &C_p\delta\kappa\left(\frac{C}{\delta}\int_{0}^{\delta}e^{-\frac{\lambda_p(r+k\delta)}{\varepsilon}}dr\right)\\
	\leq &C_p\delta\kappa\left(\frac{C\varepsilon}{\delta}\right).
\end{align*}
Therefore
\begin{align*}
	Q_{3,1}\leq C_{p,T}\kappa\left(\frac{C\varepsilon}{\delta}\right)+C_T\delta.
\end{align*}
Hence
\begin{align*}
	\mathbb{E}\sup_{t\in[0,\tau_{M,L,N}\wedge T]}Q_3\leq & C_{p,T}\int_0^{\tau_{M,L,N}\wedge T}\tilde{\kappa}\left(\mathbb{E}\sup_{r\in[0,\tau_{M,L,N}\wedge s]}\left|\hat{j}^{\varepsilon}_s-\bar{j}_s\right|^{2p}_H\right) ds\\
		&+C_{p,T}\tilde{\kappa}\left(\mathcal{O}\left(\delta\right)+C_{p,T,M}\kappa\left(C_{T,M}\delta^{\frac{1}{2}}+\mathcal{O}\left(\delta\right)\right)\right)\\
	&+C_{p,T}\tilde{\kappa}\left(C_{T,M}\delta^{\frac{1}{2}}\right)+C_{p,T}\kappa\left(\frac{C\varepsilon}{\delta}\right)+C_T\delta.
\end{align*}
Similar to $Q_3$, yet more readily than $Q_3$, we can obtain that
\begin{align*}
	\mathbb{E}\sup_{t\in[0,\tau_{M,L,N}\wedge T]}Q_4\leq & C_{p,T}\int_0^{\tau_{M,L,N}\wedge T}\kappa\left(\mathbb{E}\sup_{r\in[0,\tau_{M,L,N}\wedge s]}\left|\hat{j}^{\varepsilon}_s-\bar{j}_s\right|^{2p}_H\right) ds\\
		&+C_{p,T}\kappa\left(\mathcal{O}\left(\delta\right)+C_{p,T,M}\kappa\left(C_{T,M}\delta^{\frac{1}{2}}+\mathcal{O}\left(\delta\right)\right)\right)\\
	&+C_{p,T}\kappa\left(C_{T,M}\delta^{\frac{1}{2}}\right)+C_{p,T}\kappa\left(\frac{C\varepsilon}{\delta}\right)+C_T\delta.
\end{align*}
For $Q_5$, we have
\begin{align*}
	\mathbb{E}\sup_{t\in[0,\tau_{M,L,N}\wedge T]}Q_5=&\mathbb{E}\int_0^{\tau_{M,L,N}\wedge T}\int_{\left|z\right|_H<1}\left|\sigma_1(j^{\varepsilon}_s,z)-\sigma_1(\bar{j}_s,z)\right|^{2p}_H\nu_1(dz)ds\\
	\leq & C_p\mathbb{E}\int_0^{\tau_{M,L,N}\wedge T}\int_{\left|z\right|_H<1}\left|\sigma_1(j^{\varepsilon}_s,z)-\sigma_1(\hat{j}^{\varepsilon}_s,z)\right|^{2p}_H\nu_1(dz)ds\\
	&+C_p\mathbb{E}\int_0^{\tau_{M,L,N}\wedge T}\int_{\left|z\right|_H<1}\left|\sigma_1(\hat{j}^{\varepsilon}_s,z)-\sigma_1(\bar{j}_s,z)\right|^{2p}_H\nu_1(dz)ds\\
	\leq &C_p\mathbb{E}\int_0^{\tau_{M,L,N}\wedge T}\kappa\left(\left|j^{\varepsilon}_s-\hat{j}^{\varepsilon}_s\right|^{2p}_H\right)ds+C_p\mathbb{E}\int_0^{\tau_{M,L,N}\wedge T}\kappa\left(\left|\hat{j}^{\varepsilon}_s-\bar{j}_s\right|^{2p}_H\right)ds\\
	\leq & C_{p,T}\int_0^{\tau_{M,L,N}\wedge T}\kappa\left(\mathbb{E}\sup_{r\in[0,\tau_{M,L,N}\wedge s]}\left|\hat{j}^{\varepsilon}_s-\bar{j}_s\right|^{2p}_H\right) ds\\
	&+C_{p,T}\kappa\left(\mathcal{O}\left(\delta\right)+C_{p,T,M}\kappa\left(C_{T,M}\delta^{\frac{1}{2}}+\mathcal{O}\left(\delta\right)\right)\right).
\end{align*}
For $Q_6$, using B-D-G's inequality, we have
\begin{align*}
	&\mathbb{E}\sup_{t\in[0,\tau_{M,L,N}\wedge T]}Q_6\\
	\leq & C\mathbb{E}\left(\int_0^{\tau_{M,L,N}\wedge T}\left|\hat{j}^{\varepsilon}_s-\bar{j}_s\right|^{4p-2}_H\left|\sigma_1(j^{\varepsilon}_s,z)-\sigma_1(\bar{j}_s,z)\right|^2_H\eta_1(ds,dz)\right)^{\frac{1}{2}}\\
	\leq & C\mathbb{E}\left(\sup_{t\in[0,\tau_{M,L,N}\wedge T]}\left|\hat{j}^{\varepsilon}_s-\bar{j}_s\right|^{2p}_H\int_0^{\tau_{M,L,N}\wedge T}\left|\hat{j}^{\varepsilon}_s-\bar{j}_s\right|^{2p-2}_H\left|\sigma_1(j^{\varepsilon}_s,z)-\sigma_1(\bar{j}_s,z)\right|^2_H\eta_1(ds,dz)\right)^{\frac{1}{2}}\\
	\leq & C\gamma\mathbb{E}\sup_{t\in[0,\tau_{M,L,N}\wedge T]}\left|\hat{j}^{\varepsilon}_s-\bar{j}_s\right|^{2p}_H+C_{p,\gamma}\int_0^{\tau_{M,L,N}\wedge T}\mathbb{E}\sup_{r\in[0,\tau_{M,L,N}\wedge s]}\left|\hat{j}^{\varepsilon}_s-\bar{j}_s\right|^{2p}_Hds\\
	&+C_{p,\gamma}\mathbb{E}\int_0^{\tau_{M,L,N}\wedge T}\int_{\left|z\right|_H<1}\left|\sigma_1(j^{\varepsilon}_s,z)-\sigma_1(\bar{j}_s,z)\right|^{2p}_H\nu_1(dz)ds\\
	\leq &  C\gamma\mathbb{E}\sup_{t\in[0,\tau_{M,L,N}\wedge T]}\left|\hat{j}^{\varepsilon}_t-\bar{j}_t\right|^{2p}_H+C_{p,\gamma}\int_0^{\tau_{M,L,N}\wedge T}\mathbb{E}\sup_{r\in[0,\tau_{M,L,N}\wedge s]}\left|\hat{j}^{\varepsilon}_s-\bar{j}_s\right|^{2p}_Hds\\
	&+ C_{p,\gamma,T}\int_0^{\tau_{M,L,N}\wedge T}\kappa\left(\mathbb{E}\sup_{r\in[0,\tau_{M,L,N}\wedge s]}\left|\hat{j}^{\varepsilon}_s-\bar{j}_s\right|^{2p}_H\right) ds\\
	&+C_{p,\gamma,T}\kappa\left(\mathcal{O}\left(\delta\right)+C_{p,T,M}\kappa\left(C_{T,M}\delta^{\frac{1}{2}}+\mathcal{O}\left(\delta\right)\right)\right).
\end{align*}
In summary, choose $\gamma$ small enough, we have
\begin{align*}
	\mathbb{E}\sup_{t\in[0,\tau_{M,L,N}\wedge T]}\left|\hat{j}^{\varepsilon}_t-\bar{j}_t\right|^{2p}_H\leq & C_{p,T,N}\int_0^{\tau_{M,L,N}\wedge T}\hat{\kappa}\left(\mathbb{E}\sup_{r\in[0,\tau_{M,L,N}\wedge s]}\left|\hat{j}^{\varepsilon}_s-\bar{j}_s\right|^{2p}_H\right)ds\\
	&+C_{p,T}\hat{\kappa}\left(\mathcal{O}\left(\delta\right)+C_{p,T,M}\kappa\left(C_{T,M}\delta^{\frac{1}{2}}+\mathcal{O}\left(\delta\right)\right)\right)\\
	&+C_{p,T}\hat{\kappa}\left(C_{T,M}\delta^{\frac{1}{2}}\right)+C_{p,T}\kappa\left(\frac{C\varepsilon}{\delta}\right)+C_T\delta,
\end{align*}
where $\hat{\kappa}(u):=u+\kappa(u)+\tilde{\kappa}(u)$ is also a concave continuous non-decreasing function. Let 
\begin{align}\label{e28}
\nonumber	\mathcal{R}(\varepsilon,\delta):=&C_{p,T}\hat{\kappa}\left(\mathcal{O}_{p,T,M,L}\left(\delta\right)+C_{p,T,M}\kappa\left(C_{T,M}\delta^{\frac{1}{2}}+\mathcal{O}_{p,T,M,L}\left(\delta\right)\right)\right)\\
	&+C_{p,T}\hat{\kappa}\left(C_{T,M}\delta^{\frac{1}{2}}\right)+C_{p,T}\kappa\left(\frac{C\varepsilon}{\delta}\right)+C_T\delta.
\end{align}
Applying the Bihari's inequality (see \cite[Section 3]{Bihari}), we have
\begin{align*}
	\mathbb{E}\sup_{t\in[0,\tau_{M,L,N}\wedge T]}\left|\hat{j}^{\varepsilon}_t-\bar{j}_t\right|^{2p}_H\leq \hat{\Omega}^{-1}\left(\hat{\Omega}\left(\mathcal{R}(\varepsilon,\delta)\right)+C_{p,T,N}T\right),
\end{align*}
where
\begin{align}\label{e29}
	\hat{\Omega} (v):=\int_0^v\frac{du}{\hat{\kappa}(u)},\ \forall v>0.
\end{align}
\end{proof}

At the end of this subsection, we claim the following main result:
\begin{theorem}\label{th7}
Let $\sigma _1(j,\theta,z)\equiv \sigma _1(j,z)$	, under Assumptions \ref{a1}-\ref{a4}, for any $T>0$, $\left ( j_0,\theta_0\right )\in H\times H$ and for some $p\geq 1$, we have
\begin{align}\label{e30}
\lim_{\varepsilon \rightarrow 0}\mathbb{E}\sup_{t\in[0,T]}\left |j^{\varepsilon}_t-\bar{j}_t \right |^{2p}_{H}=0.
\end{align}
\end{theorem}
\begin{proof}
By (\ref{e26}) and (\ref{e27}), we have
\begin{align*}
	\mathbb{E}\sup_{t\in[0,\tau_{M,L,N}\wedge T]}\left |j^{\varepsilon}_t-\bar{j}_t \right |^{2p}_{H}\leq & C_p\mathbb{E}\sup_{t\in[0,\tau_{M,L,N}\wedge T]}\left |j^{\varepsilon}_t-\hat{j}^{\varepsilon}_t \right |^{2p}_{H}+\mathbb{E}\sup_{t\in[0,\tau_{M,L,N}\wedge T]}\left |\hat{j}^{\varepsilon}_t-\bar{j}_t \right |^{2p}_{H}\\
	\leq & \mathcal{O}\left(\delta\right)+C_{p,T,M}\kappa\left(C_{T,M}\delta^{\frac{1}{2}}+\mathcal{O}\left(\delta\right)\right)\\
	&+ \hat{\Omega}^{-1}\left(\hat{\Omega}\left(\mathcal{R}(\varepsilon,\delta)\right)+C_{p,T,N}T\right).
\end{align*}
Let $\delta =\varepsilon ^{\frac{1}{2}}$, note that
\begin{align*}
	\mathcal{O}\left(\varepsilon ^{\frac{1}{2}}\right)\rightarrow 0,\quad and\quad \mathcal{R}\left(\varepsilon,\varepsilon ^{\frac{1}{2}}\right)\rightarrow 0,
\end{align*}
as $\varepsilon\rightarrow 0$. According to \cite[Lemma 2.3]{Gou}, we have
\begin{align*}
	\hat{\Omega}^{-1}\left(\hat{\Omega}\left(\mathcal{R}(\varepsilon,\delta)\right)+C_{p,T,N}T\right)\rightarrow 0, \quad as\quad \mathcal{R}\left(\varepsilon,\varepsilon ^{\frac{1}{2}}\right)\rightarrow 0.
\end{align*}
Hence
\begin{align*}
	\lim_{\varepsilon\rightarrow 0}\mathbb{E}\sup_{t\in[0,\tau_{M,L,N}\wedge T]}\left |j^{\varepsilon}_t-\bar{j}_t \right |^{2p}_{H}=0.
\end{align*}
Since (\ref{e4}), (\ref{e5}), (\ref{e19}), (\ref{e24}) and (\ref{e25}) implies that $\tau_{M,L,N}\rightarrow \infty$ as $M,L,N\rightarrow\infty$, let $\varepsilon \rightarrow 0$ firstly, $M,L,N\rightarrow \infty $ secondly,  we obtain
\begin{align*}
\lim_{\varepsilon \rightarrow 0}\mathbb{E}\sup_{t\in[0,T]}\left | j^{\varepsilon}_t-\bar{j}_t\right |^{2p}_{H}=0.
\end{align*}
\end{proof}

\section{Example and numerical simulation}\label{s6}
In this section, we provide an example to verify the theoretical results of this paper.
\subsection{Example}
Let us consider the following two-scale system:
\begin{equation}\label{e31}
	\left\{
\begin{aligned}
&\frac{\partial j^{\varepsilon}}{\partial t} + (u^{\varepsilon}\cdot \nabla)j^{\varepsilon} 
= \triangle j^{\varepsilon} - \frac{1}{2}\mathrm{sgn}(j^{\varepsilon})|j^{\varepsilon}|^{2/3} - \frac{3}{10}\mathrm{sgn}(\theta^{\varepsilon})|\theta^{\varepsilon}|^{1/3}+ \partial_x\theta^{\varepsilon} \\
&\qquad\qquad\qquad\qquad + \int_{|z|_H<1}\frac{\left(1+(j^{\varepsilon})^2\right)^{1/3}}{\sqrt{1+|z|_H^2}}\dot{\tilde{\eta}}_1(t,dz),\\
&\frac{\partial\theta^{\varepsilon}}{\partial t} + \frac{1}{\varepsilon}(u^{\varepsilon}\cdot \nabla)\theta^{\varepsilon} 
= \frac{1}{\varepsilon}\triangle\theta^{\varepsilon} - \int_{|z|_H<1}\frac{1}{2}\theta^{\varepsilon} e^{-|z|_H^2} \dot{\tilde{\eta}}^{\varepsilon}_2(t,dz),\\
&j^{\varepsilon}(0) = j_0, \quad \theta^{\varepsilon}(0) = \theta_0,
\end{aligned}
\right.
\end{equation}
where $\mathrm{sgn}(\cdot)$ denotes the sign function, and $z = \sum_{k=1}^{\infty} z_k e_k$ is the expansion of $z$ in the basis $\{e_k\}$ of $H$. The compensated Poisson random measures $\tilde{\eta}_1(t,z)=\eta_1(t,z)-\nu_1(z)\cdot t$ and $\tilde{\eta}^{\varepsilon}_2(t,z)=\eta_2(\frac{1}{\varepsilon}t,z)-\frac{1}{\varepsilon}\nu_2(z)\cdot t$, and the intensity measures $\nu_i(i=1,2)$ defined by
\begin{align*}
	\nu_i(dz) = \frac{C_{\nu_i}}{|z|_H^{1+\beta_i}}dz, \quad i=1,2,\quad for\quad |z|_H < 1,
\end{align*}
where $\beta_i\in(0,1)$ and $C_{\nu_i}$ is normalization constant.

Compared with the original equation (\ref{e2}), it is readily apparent that
\begin{align*}
	&f(j,\theta)== -\frac{1}{2}\mathrm{sgn}(j)\left|j\right|^{2/3} - \frac{3}{10}\mathrm{sgn}(\theta)\left|\theta\right|^{1/3},\\
	&\sigma_1(j,z) = \left(1+j^2\right)^{1/3}\frac{1}{\sqrt{1+|z|_H^2}},\\
	&\sigma_2(j,\theta,z) \equiv\sigma_2(\theta,z)= -\frac{1}{2}\theta e^{-|z|_H^2},
\end{align*}
and satisfies Assumptions \ref{a1}-\ref{a4} with $p=1$ and $\kappa(u)=Cu^{\frac{2}{3}}$.

For and fixed $u \in H^1$, the frozen equation is:
\begin{equation}\label{e32}
\left\{
\begin{aligned}
&\frac{\partial\tilde{\theta}_t}{\partial t} + (u\cdot \nabla)\tilde{\theta}_t = \triangle\tilde{\theta}_t - \int_{|z|_H<1}\frac{1}{2}\tilde{\theta}_t e^{-|z|_H^2} \dot{\tilde{\eta}}_2(t,dz),\\
&\tilde{\theta}_0 = \theta.
\end{aligned}
\right.
\end{equation}
It is easy to check that (\ref{e32}) admits a unique ergodic invariant measure $\pi^{u}$ on $H$. In addition, due to $\nu_2(\cdot)$ is symmetric, we have $\pi^{u}$ is symmetric. Hence, the integral of the odd function with respect to $\pi^{u}$ vanishes. Therefore
\begin{align*}
	\bar{f}(j,u)= -\frac{1}{2}\mathrm{sgn}(j)|j|^{2/3} - \frac{3}{10}\int_H \mathrm{sgn}(\hat{\theta})|\hat{\theta}|^{1/3}\pi^{u}(d\hat{\theta})=-\frac{1}{2}\mathrm{sgn}(j)|j|^{2/3},
\end{align*}
and
\begin{align*}
	g(u)=\int_H \hat{\theta} \, \pi^{u}(d\hat{\theta}) = 0.
\end{align*}
And so, the averaged equation give by
\begin{equation}\label{e33}
\left\{
\begin{aligned}
&\frac{\partial \bar{j}}{\partial t} + (\bar{u} \cdot \nabla)\bar{j} 
= \triangle \bar{j} - \frac{1}{2}\mathrm{sgn}(\bar{j})|\bar{j}|^{2/3} + \int_{|z|_H<1}\frac{(1+\bar{j}^2)^{1/3}}{\sqrt{1+|z|_H^2}}\dot{\tilde{\eta}}_1(t,dz),\\
&\bar{j}(0) = j_0.
\end{aligned}
\right.
\end{equation}

\subsection{Numerical simulation}

Let $\beta_1=0.8$, $\beta_2=0.6$, for $\varepsilon=1,0.5,0.25,0.1$, define
\begin{align*}
	error(\varepsilon):=\sup_{t\in[0,1]}\left|j^{\varepsilon}-\bar{j}\right|_H,
\end{align*}
and choose 100 samples with respect to $(j^{\varepsilon},\theta^{\varepsilon})$ and $\bar{j}$, we get the error distribution for 100 samples with different $\varepsilon$, see Figure \ref{f1}. 
\begin{figure}[H]
\centering
\includegraphics[width=0.78\textwidth]{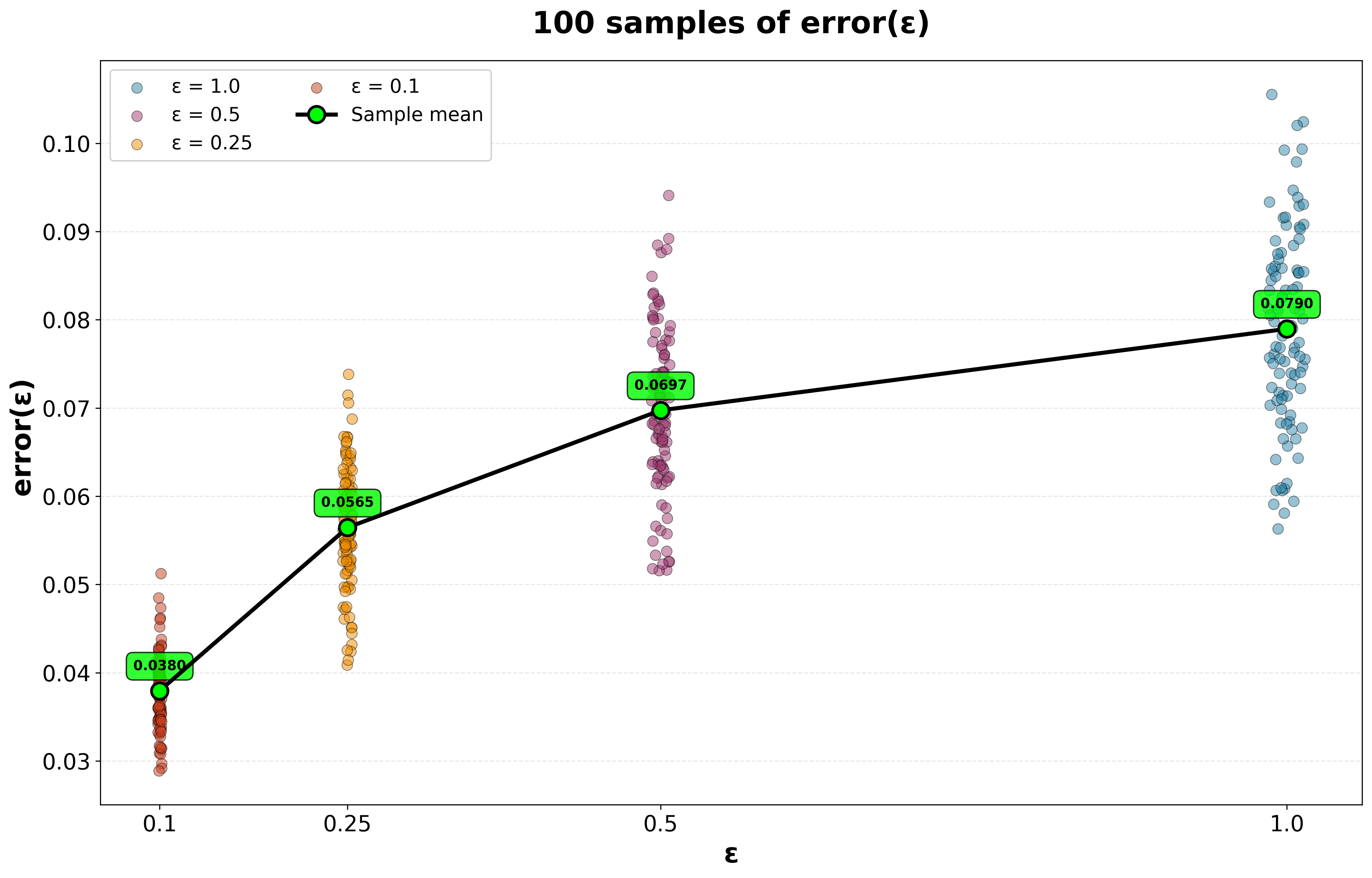}
\caption{100 samples of error($\varepsilon$)} % 图片下方解释
        \label{f1} % 便于文章中引用图片
\end{figure}
%\FloatBarrier% 图片不悬浮

As can be easy seen from the Figure $\ref{f1}$, the numerical simulation results for 100 samples are consistent with Theorem \ref{th6}.

Moreover, the mean squared error of $\varepsilon$ is shown in Figure \ref{f2}.
\begin{figure}[H]
\centering
\includegraphics[width=0.78\textwidth]{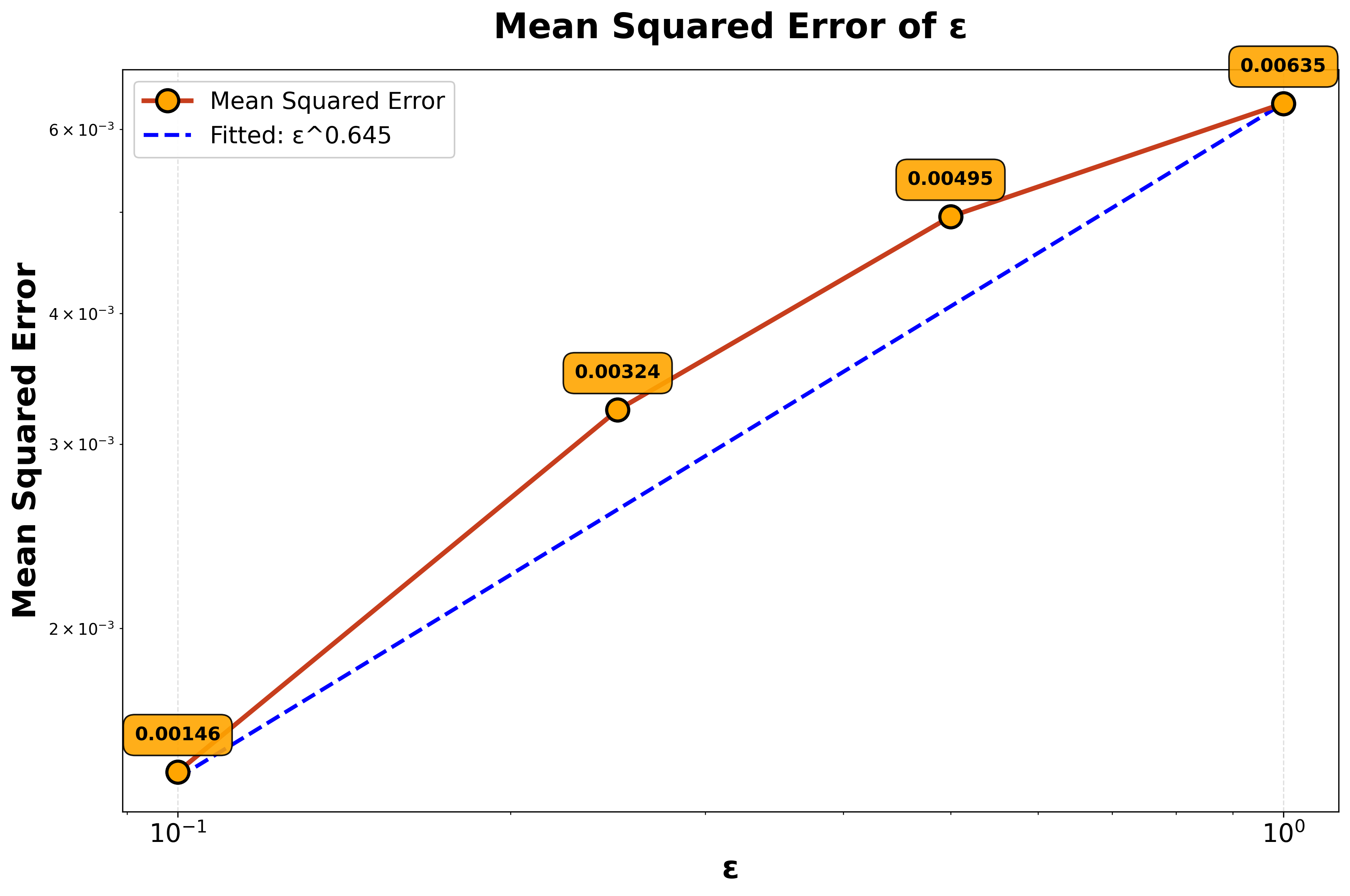}
\caption{mean squared error of $\varepsilon$} % 图片下方解释
        \label{f2} % 便于文章中引用图片
\end{figure}
%\FloatBarrier% 图片不悬浮

Observing Figure \ref{f2}, the simulation results are consistent with the conclusion of Theorem \ref{th7} with $p=1$. In addition, in this example, the mean squared error obtained through power law fitting is
\begin{align*}
	\text{mean squared error}\approx C\varepsilon^{0.645}.
\end{align*}

%% -------------------------------------------------------------------

\section*{Appendix}
\appendix
\renewcommand{\appendixname}{Appendix~\Alph{section}}
\addcontentsline{toc}{chapter}{Appendix}
\subsection*{Well-posedness of (\ref{e2})}

For convenience, let $\varepsilon=1$, according to (\ref{e4}) and (\ref{xe11}), it is straightforward to prove the existence of solution by the standard Galerkin's approximation. We turn to consider the problem of uniqueness.

Let $(j_1,\theta_1)$ and $(j_2,\theta_2)$ be two solution of (\ref{e2}) with the same initial value $(j_0,\theta_0)$. For any $T>0$ and $R>0$, we define the stopping time
\begin{align*}
	\tau_R:=\left\{t\leq 0;\left|j_2\right|_H+\left|\theta_2\right|_H>R\right\}.
\end{align*}
By It\^{o}'s formula, we have
\begin{align*}
	\left|j_1-j_2\right|^2_H&+2\int_0^t\left(\left(u_1\cdot\triangledown\right)j_1-\left(u_2\cdot \triangledown\right)j_2,j_1-j_2\right)ds=2\int_0^t\left(\triangle\left(j_1-j_2\right),j_1-j_2\right)ds\\
	&+2\int_0^t\left(f(j_1,\theta_1)-f(j_2,\theta_2),j_1-j_2\right)ds+2\int_0^t\left(\partial_x\theta_1-\partial_x\theta_2,j_1-j_2\right)ds\\
	&+\int_0^t\int_{\left|z\right|_H<1}\left|\sigma_1(j_1,\theta_1,z)-\sigma_1(j_2,\theta_2,z)\right|^2_H\eta_1(ds,dz)\\
	&+2\int_0^t\int_{\left|z\right|_H<1}\left(\sigma_1(j_1,\theta_1,z)-\sigma_1(j_2,\theta_2,z),j_1-j_2\right)\tilde{\eta}_1(ds,dz).
\end{align*}
Therefore
\begin{align*}
	\left|j_1-j_2\right|^2_H&+2\int_0^t\left|\triangledown\left( j_1-j_2\right)\right|^2_Hds=2\left\{\int_0^t\left(\left(\left(u_1-u_2\right)\cdot \triangledown\right)\left(j_1-j_2\right),j_2\right)ds\right\}_{H_1}\\
	&+2\left\{\int_0^t\left(f(j_1,\theta_1)-f(j_2,\theta_2),j_1-j_2\right)ds\right\}_{H_2}-2\left\{\int_0^t\left(\theta_1-\theta_2,\partial_x j_1-\partial_x j_2\right)ds\right\}_{H_3}\\
	&+\left\{\int_0^t\int_{\left|z\right|_H<1}\left|\sigma_1(j_1,\theta_1,z)-\sigma_1(j_2,\theta_2,z)\right|^2_H\eta_1(ds,dz)\right\}_{H_4}\\
	&+2\left\{\int_0^t\int_{\left|z\right|_H<1}\left(\sigma_1(j_1,\theta_1,z)-\sigma_1(j_2,\theta_2,z),j_1-j_2\right)\tilde{\eta}_1(ds,dz)\right\}_{H_5}.
\end{align*}
By H\"{o}lder's inequality, Ladyzhenskaya's inequality and Young's inequality, we derive that
\begin{align*}
	\mathbb{E}\sup_{t\in[0,\tau_R\wedge T]}H_1\leq & \mathbb{E}\int_0^{\tau_R\wedge T}\left|u_1-u_2\right|_{L^4}\left|\triangledown\left(j_1-j_2\right)\right|_H\left|j_2\right|_{L^4}ds\\
	\leq & \mathbb{E}\int_0^{\tau_R\wedge T}\left|\triangledown\left(u_1-u_2\right)\right|_H\left|\triangledown\left(j_1-j_2\right)\right|_H\left|\triangledown j_2\right|_Hds\\
		\leq & \mathbb{E}\int_0^{\tau_R\wedge T}\left|j_1-j_2\right|_H\left|\triangledown\left(j_1-j_2\right)\right|_H\left|\triangledown j_2\right|_Hds\\
		\leq & \gamma \mathbb{E}\int_0^{\tau_R\wedge T}\left|\triangledown\left(j_1-j_2\right)\right|^2_Hds+C_{\gamma} \mathbb{E}\int_0^{\tau_R\wedge T}\left|j_1-j_2\right|^2_H\left|\triangledown j_2\right|^2_Hds.
\end{align*}
Note that
\begin{align*}
	\frac{d}{dt}\left|j_2\right|^2_H+2\left|\triangledown j_2\right|^2_H=&2\left(f(j_2,\theta_2),j_2\right)-2\left(\theta_2,\partial_x j_2\right)+\int_{\left|z\right|_H<1}\left|\sigma_1(j_2,\theta_2,z)\right|^2_H\dot{\eta}_1(t,dz)\\
	&+2\int_{\left|z\right|_H<1}\left(\sigma_1(j_2,\theta_2,z),j_2\right)\dot{\tilde{\eta}}_1(t,dz),
\end{align*}
we have
\begin{align*}
	\left \langle \left|j_2\right|^2_H \right \rangle _t=\int_0^t\int_{\left|z\right|_H<1}\left(\left|\sigma_1(j_2,\theta_2,z)\right|^2_H+2\left(\sigma_1(j_2,\theta_2,z),j_2\right)\right)^2\eta_1(ds,dz),
\end{align*}
and
\begin{align*}
	\frac{d}{dt}\left|j_2\right|^2_H+\left|\triangledown j_2\right|^2_H\leq &C\left(1+\left|j_2\right|^2_H+\left|\theta_2\right|^2_H\right)+\int_{\left|z\right|_H<1}\left|\sigma_1(j_2,\theta_2,z)\right|^2_H\dot{\eta}_1(t,dz)\\
	&+2\int_{\left|z\right|_H<1}\left(\sigma_1(j_2,\theta_2,z),j_2\right)\dot{\tilde{\eta}}_1(t,dz).
\end{align*}
Using B-D-G's inequality, we have
\begin{align*}
&\mathbb{E}\int_0^{\tau_R\wedge T}\left|j_1-j_2\right|^2_H\left|\triangledown j_2\right|^2_Hds\\
\leq &-	\mathbb{E}\int_0^{\tau_R\wedge T}\left|j_1-j_2\right|^2_Hd\left|j_2\right|^2_H+C\mathbb{E}\int_0^{\tau_R\wedge T}\left(1+\left|j_2\right|^2_H+\left|\theta_2\right|^2_H\right)\left|j_1-j_2\right|^2_Hds\\
&+\mathbb{E}\int_0^{\tau_R\wedge T}\int_{\left|z\right|_H<1}\left|\sigma_1(j_2,\theta_2,z)\right|^2_H\left|j_1-j_2\right|^2_H\eta_1(ds,dz)\\
&+2\mathbb{E}\int_0^{\tau_R\wedge T}\int_{\left|z\right|_H<1}\left(\sigma_1(j_2,\theta_2,z),j_2\right)\left|j_1-j_2\right|^2_H\tilde{\eta}_1(ds,dz)\\
\leq & C\mathbb{E}\left(\int_0^{\tau_R\wedge T}\int_{\left|z\right|_H<1}\left|j_1-j_2\right|^4_H\left(\left|\sigma_1(j_2,\theta_2,z)\right|^2_H+2\left(\sigma_1(j_2,\theta_2,z),j_2\right)\right)^2\eta_1(ds,dz)\right)^{\frac{1}{2}}\\
&+C_{R}\mathbb{E}\int_0^{\tau_R\wedge T}\left|j_1-j_2\right|^2_Hds+\mathbb{E}\int_0^{\tau_R\wedge T}\int_{\left|z\right|_H<1}\left|\sigma_1(j_2,\theta_2,z)\right|^2_H\left|j_1-j_2\right|^2_H\nu_1(dz)ds\\
&+C\mathbb{E}\left(\int_0^{\tau_R\wedge T}\int_{\left|z\right|_H<1}\left|\left(\sigma_1(j_2,\theta_2,z),j_2\right)\right|^2\left|j_1-j_2\right|^4_H\eta_1(ds,dz)\right)^{\frac{1}{2}}\\
\leq & C\gamma\mathbb{E}\sup_{t\in [0,\tau_R\wedge T]}\left|j_1-j_2\right|^2_H+C_{R,\gamma}\mathbb{E}\int_0^{\tau_R\wedge T}\left|j_1-j_2\right|^2_Hds
\end{align*}
Hence
\begin{align*}
	\mathbb{E}\sup_{t\in[0,\tau_R\wedge T]}H_1\leq &  C\gamma\mathbb{E}\sup_{t\in [0,\tau_R\wedge T]}\left|j_1-j_2\right|^2_H+\gamma \mathbb{E}\int_0^{\tau_R\wedge T}\left|\triangledown\left(j_1-j_2\right)\right|^2_Hds\\
	&+C_{R,\gamma}\mathbb{E}\int_0^{\tau_R\wedge T}\left|j_1-j_2\right|^2_Hds
\end{align*}
According to Assumption \ref{a1}, we have
\begin{align*}
	\mathbb{E}\sup_{t\in[0,\tau_R\wedge T]}H_2\leq C\mathbb{E}\int_0^{\tau_R\wedge T}\left|j_1-j_2\right|^2_Hds+ C\mathbb{E}\int_0^{\tau_R\wedge T}\kappa\left(\left|j_1-j_2\right|^2_H+\left|\theta_1-\theta_2\right|^2_H\right)ds.
\end{align*}
It is easy to obtain that
\begin{align*}
	\mathbb{E}\sup_{t\in[0,\tau_R\wedge T]}H_3\leq \gamma \mathbb{E}\int_0^{\tau_R\wedge T}\left|\triangledown\left(j_1-j_2\right)\right|^2_Hds+C_{\gamma}\mathbb{E}\int_0^{\tau_R\wedge T}\left|\theta_1-\theta_2\right|^2_Hds
\end{align*}
According to Assumption \ref{a2}, we can get
\begin{align*}
		\mathbb{E}\sup_{t\in[0,\tau_R\wedge T]}H_4\leq C\mathbb{E}\int_0^{\tau_R\wedge T}\kappa\left(\left|j_1-j_2\right|^2_H+\left|\theta_1-\theta_2\right|^2_H\right)ds,
\end{align*}
and
\begin{align*}
	\mathbb{E}\sup_{t\in[0,\tau_R\wedge T]}H_5\leq  C\gamma\mathbb{E}\sup_{t\in [0,\tau_R\wedge T]}\left|j_1-j_2\right|^2_H+ C_{\gamma}\mathbb{E}\int_0^{\tau_R\wedge T}\kappa\left(\left|j_1-j_2\right|^2_H+\left|\theta_1-\theta_2\right|^2_H\right)ds.
\end{align*}
Choose $\gamma$ small enough, we have
\begin{align*}
	\mathbb{E}\sup_{t\in [0,\tau_R\wedge T]}\left|j_1-j_2\right|^2_H\leq & C_{R}\mathbb{E}\int_0^{\tau_R\wedge T}\left|j_1-j_2\right|^2_H+\left|\theta_1-\theta_2\right|^2_Hds\\
	&+C\mathbb{E}\int_0^{\tau_R\wedge T}\kappa\left(\left|j_1-j_2\right|^2_H+\left|\theta_1-\theta_2\right|^2_H\right)ds.
\end{align*}
A similar derivation process, we can get
\begin{align*}
\mathbb{E}\sup_{t\in [0,\tau_R\wedge T]}\left|\theta_1-\theta_2\right|^2_H\leq & C_{R}\mathbb{E}\int_0^{\tau_R\wedge T}\left|\theta_1-\theta_2\right|^2_Hds\\
	&+C\mathbb{E}\int_0^{\tau_R\wedge T}\kappa\left(\left|j_1-j_2\right|^2_H+\left|\theta_1-\theta_2\right|^2_H\right)ds.
\end{align*}
Therefore
\begin{align*}
	\mathbb{E}\sup_{t\in [0,\tau_R\wedge T]}\left(\left|j_1-j_2\right|^2_H+\left|\theta_1-\theta_2\right|^2_H\right)\leq C_R\int_0^{\tau_R\wedge T}\check{\kappa} \left(\left|j_1-j_2\right|^2_H+\left|\theta_1-\theta_2\right|^2_H\right)ds,
\end{align*}
where $\check{\kappa}(u):=u+\kappa(u)$. By Bihari's inequality, we have
\begin{align*}
	\mathbb{E}\sup_{t\in [0,\tau_R\wedge T]}\left(\left|j_1-j_2\right|^2_H+\left|\theta_1-\theta_2\right|^2_H\right)\leq \check{\Omega}^{-1}\left(\check{\Omega}(0)+C_RT\right),
\end{align*}
where
\begin{align*}
	\check{\Omega}(v):=\int_0^v\frac{du}{\check{\kappa}(u)}, \forall v\geq 0.
\end{align*}
According to \cite[Lemma 2.3]{Gou}, we have
\begin{align*}
	\mathbb{E}\sup_{t\in [0,\tau_R\wedge T]}\left(\left|j_1-j_2\right|^2_H+\left|\theta_1-\theta_2\right|^2_H\right)\equiv 0.
\end{align*}
Since (\ref{e4}) and (\ref{xe11}) implies that $\tau_R\rightarrow \infty$ as $R\rightarrow \infty$, the uniqueness holds by dominated convergence theorem and letting $R\rightarrow \infty$.
\rightline{$\qedsymbol$}

\subsection*{Proof of Lemma \ref{l4}}

Note that $F$ is increasing and bounded with respect to $t$, we have
\begin{align*}
	\lim_{t\rightarrow+\infty}F(t)=\sup_{t\geq 0}F(t):=M.
\end{align*}
Therefore, for any $\varepsilon>0$, there exists some $\delta>0$, such that
\begin{align*}
	e^{-a\delta}M<\frac{\varepsilon}{2}.
\end{align*}
Let
\begin{align*}
	I(t):=\int_0^te^{-a(t-s)}f(s)ds.
\end{align*}
For sufficiently large $t$, we have
\begin{align*}
	I(t)=\left\{\int_0^{t-\delta}e^{-a(t-s)}f(s)ds\right\}_{I_1(t)}+\left\{\int_{t-\delta}^{t}e^{-a(t-s)}f(s)ds\right\}_{I_2(t)}.
\end{align*}
For $I_1(t)$, due to $\delta\leq t-s\leq t$, we have
\begin{align*}
	I_1(t)\leq e^{-a\delta}\int_0^{t-\delta}f(s)ds\leq e^{-a\delta}M<\frac{\varepsilon}{2}.
\end{align*}
For $I_2(t)$, we have
\begin{align*}
	I_2(t)\leq \int_{t-\delta}^{t}f(s)ds=F(t)-F(t-\delta)<\frac{\varepsilon}{2},
\end{align*}
for sufficiently large $t$. Hence, Lemma \ref{l4} hold.

\rightline{$\qedsymbol$} 

\section*{Acknowledgments}
The authors would like to acknowledge the following funding sources for their support:
\begin{itemize}
	\item \textbf{Yangyang Shi}: National Natural Science Foundation of China (No. 12401304);
	\item \textbf{Dong Su}:  the Higher Education Institutions Basic Science (Natural Science) Research Project of Jiangsu Province (No. 2025KJB110004);
	\item \textbf{Hui Liu}: National Natural Science Foundation of China (No. 12271293), Natural Science Foundation of Shandong Province (No. ZR2023MA002). 
\end{itemize}

\section*{Conflicts of Interest}
The authors declare no conflicts of interest.
%% -------------------------------------------------------------------

%\bibliographystyle{plain}
%\bibliography{ref1}

\begin{thebibliography}{99}

\bibitem{Zaussinger}{\textsc{F. Zaussinger and H. Spruit}, Semiconvection: numerical simulations, {\it Astron. Astrophys.}, {\bf 554} (2010) 1--13.}

\bibitem{Thomas}{\textsc{T. Bridges and D. Ratliff}, Double criticality and the two-way Boussinesq equation in stratified shallow water hydrodynamics, {\it Phys. Fluids}, {\bf 28} (2016) 062103.}

\bibitem{Ahmet}{\textsc{A. Bekir, A. Cevikel and E. Zahran}, New impressive representations for the soliton behaviors arising from the (2+1)-Boussinesq equation, {\it J. Ocean. Eng. Sci.} (2022).}

\bibitem{Farlin}{\textsc{J. Farlin and P. Maloszewski}, On the use of spring baseflow recession for a more accurate parameterization of aquifer transit time distribution functions, {\it Hydrol. Earth Syst. Sci.}, {\bf 17} (2013) 1825--1831.}

\bibitem{Knight}{\textsc{J. Knight and D. Rassam}, Groundwater head responses due to random stream stage fluctuations using basis splines, {\it Water Resour. Res.}, {\bf 43} (2007) W06435.}

\bibitem{Yang}{\textsc{W. Yang and C. Li}, General Propagation Lattice Boltzmann Model for the Boussinesq Equation, {\it Entropy}, {\bf 24} (2022) 486.}

\bibitem{Pu}{\textsc{X. Pu and B. Guo}, Global well-posedness of the stochastic 2D Boussinesq equations with partial viscosity, {\it Acta Math. Sci. Ser. B (Engl. Ed.)}, {\bf 31} (2011) 1968--1984.}

\bibitem{Lihuai}{\textsc{L. Du and T. Zhang}, Local and global existence of pathwise solution for the stochastic Boussinesq equations with multiplicative noises, {\it Stochastic Process. Appl.}, {\bf 130} (2020) 1545--1567.}

\bibitem{Diego}{\textsc{D. Alonso-Or\'{a}n and A. Bethencourt de Le\'{o}n}, On the well-posedness of stochastic Boussinesq equations with transport noise, {\it J. Nonlinear Sci.}, {\bf 30} (2020) 175--224.}

\bibitem{Jinyi}{\textsc{J. Sun, N. Li and M. Yang}, Global existence of mild solutions for 3D stochastic Boussinesq system in Besov spaces, {\it Math. Nachr.}, {\bf 298} (2025) 1105--1126.}

\bibitem{Shang}{\textsc{S. Wu and J. Huang}, Well-posedness and limit behavior of stochastic fractional Boussinesq equation driven by nonlinear noise, {\it Phys. D}, {\bf 461} (2024) Paper No. 134104, 22.}

\bibitem{Quyuan}{\textsc{Q. Lin, R. Liu and W. Wang}, Global existence for the stochastic Boussinesq equations with transport noise and small rough data, {\it SIAM J. Math. Anal.}, {\bf 56} (2024) 501--528.}

\bibitem{Yamazaki}{\textsc{K. Yamazaki}, Global martingale solution for the stochastic Boussinesq system with zero dissipation, {\it Stoch. Anal. Appl.}, {\bf 34} (2016) 404--426.}

\bibitem{Dejun}{\textsc{D. Luo}, Convergence of stochastic 2D inviscid Boussinesq equations with transport noise to a deterministic viscous system, {\it Nonlinearity}, {\bf 34} (2021) 8311--8330.}

\bibitem{Millet}{\textsc{J. Duan and A. Millet}, Large deviations for the Boussinesq equations under random influences, {\it Stochastic Process. Appl.}, {\bf 119} (2009) 2052--2081.}

\bibitem{Chengfeng}{\textsc{C. Sun, H. Gao, J. Duan and B. Schmalfu\ss}, Rare events in the Boussinesq system with fluctuating dynamical boundary conditions, {\it J. Differential Equations}, {\bf 248} (2010) 1269--1296.}

\bibitem{Yan}{\textsc{Y. Zheng and J. Huang}, Ergodicity of stochastic Boussinesq equations driven by L\'evy processes, {\it Sci. China Math.}, {\bf 56} (2013) 1195--1212.}

\bibitem{Jianhua}{\textsc{J. Huang, Y. Zheng, T. Shen and C. Guo}, Asymptotic properties of the 2D stochastic fractional Boussinesq equations driven by degenerate noise, {\it J. Differential Equations}, {\bf 310} (2022) 362--403.}

\bibitem{Haoran}{\textsc{H. Dai, B. You and T. Caraballo}, Asymptotical behavior of the 2D stochastic partial dissipative Boussinesq system with memory, {\it Commun. Nonlinear Sci. Numer. Simul.}, {\bf 149} (2025) Paper No. 108916, 15.}

\bibitem{Bertram}{\textsc{R. Bertram and J. Rubin}, Multi-timescale systems and fast-slow analysis, {\it Math. Biosci.}, {\bf 287} (2017) 105--121.}

\bibitem{Weinan}{\textsc{W. E and B. Engquist}, Multiscale modeling and computation, {\it Notices Amer. Math. Soc.}, {\bf 50} (2003) 1062--1070.}

\bibitem{Schutter}{\textsc{E. De Schutter}, Modeling intracellular calcium dynamics, in {\it Computational modeling methods for neuroscientists}, Comput. Neurosci., MIT Press, Cambridge, MA (2010) 93--105.}

\bibitem{Khasminskii1}{\textsc{R. Khasminskii}, A limit theorem for solutions of differential equations with a random right hand part, {\it Teor. Verojatnost. i Primenen}, {\bf 11} (1966) 444--462.}

\bibitem{Khasminskii2}{\textsc{R. Khasminskii}, On the principle of averaging the It\^{o}'s stochastic differential equations, {\it Kybernetika (Prague)}, {\bf 4} (1968) 260--279.}

\bibitem{Cerrai2}{\textsc{S. Cerrai and A. Lunardi}, Averaging principle for nonautonomous slow-fast systems of stochastic reaction-diffusion equations: the almost periodic case, {\it SIAM J. Math. Anal.}, {\bf 49} (2017) 2843--2884.}

\bibitem{Duan1}{\textsc{J. Duan and W. Wang}, Effective dynamics of stochastic partial differential equations, Elsevier Insights, Elsevier, Amsterdam (2014).}

\bibitem{Golec}{\textsc{J. Golec}, Stochastic averaging principle for systems with pathwise uniqueness, {\it Stochastic Anal. Appl.}, {\bf 13} (1995) 307--322.}

\bibitem{Di}{\textsc{D. Liu}, Strong convergence of principle of averaging for multiscale stochastic dynamical systems, {\it Commun. Math. Sci.}, {\bf 8} (2010) 999--1020.}

\bibitem{Liu1}{\textsc{W. Liu, M. R\"{o}ckner, X. Sun and Y. Xie}, Averaging principle for slow-fast stochastic differential equations with time dependent locally Lipschitz coefficients, {\it J. Differential Equations}, {\bf 268} (2020) 2910--2948.}

\bibitem{Pavliotis}{\textsc{G. A. Pavliotis and A. M. Stuart}, Multiscale methods: averaging and homogenization, Texts in Applied Mathematics, vol. 53, Springer, New York (2008).}

\bibitem{Pei1}{\textsc{B. Pei, Y. Xu and G. Yin}, Stochastic averaging for a class of two-time-scale systems of stochastic partial differential equations, {\it Nonlinear Anal.}, {\bf 160} (2017) 159--176.}

\bibitem{Guo}{\textsc{Q. Guo, P. Guo and F. Wan}, Strong convergence in the $p$th-mean of an averaging principle for two-time-scales SPDEs with jumps, {\it Adv. Difference Equ.} (2017) Paper No. 275, 23.}

\bibitem{Majda}{\textsc{A. Majda, I. Timofeyev and E. Vanden Eijnden}, A mathematical framework for stochastic climate models, {\it Comm. Pure Appl. Math.}, {\bf 54} (2001) 891--974.}

\bibitem{Arnold}{\textsc{L. Arnold, P. Imkeller and Y. Wu}, Reduction of deterministic coupled atmosphere---ocean models to stochastic ocean models: a numerical case study of the Lorenz-Maas system, {\it Dyn. Syst.}, {\bf 18} (2003) 295--350.}

\bibitem{Yeong}{\textsc{H. Yeong, R. Beeson, N. Namachchivaya and N. Perkowski}, Particle filters with nudging in multiscale chaotic systems: with application to the Lorenz '96 atmospheric model, {\it J. Nonlinear Sci.}, {\bf 30} (2020) 1519--1552.}

\bibitem{Gao1}{\textsc{P. Gao}, Averaging principle for multiscale stochastic fractional Schr\"{o}dinger--Korteweg-de Vries system, {\it J. Stat. Phys.}, {\bf 181} (2020) 1781--1816.}

\bibitem{Gao2}{\textsc{P. Gao}, Averaging principle for multiscale nonautonomous random 2D Navier-Stokes system, {\it J. Funct. Anal.}, {\bf 285} (2023) Paper No. 110036, 48.}

\bibitem{Gao3}{\textsc{P. Gao}, Averaging principles for multiscale stochastic Cahn-Hilliard system, {\it J. Math. Phys.}, {\bf 65} (2024) Paper No. 052702, 24.}

\bibitem{Yeyu}{\textsc{Y. Zhang, L. Smith and S. Stechmann}, Fast-wave averaging with phase changes: asymptotics and application to moist atmospheric dynamics, {\it J. Nonlinear Sci.}, {\bf 31} (2021) Paper No. 38, 46.}

\bibitem{Lihuai1}{\textsc{L. Du and T. Zhang}, The global existence and averaging theorem for the strong solution of the stochastic Boussinesq equations with the low Froude number, {\it J. Math. Fluid Mech.}, {\bf 24} (2022) Paper No. 32, 20.}

\bibitem{Brzezniak}{\textsc{Z. Brze\'{z}niak, W. Liu and J. Zhu}, Strong solutions for SPDE with locally monotone coefficients driven by L\'{e}vy noise, {\it Nonlinear Anal. Real World Appl.}, {\bf 17} (2014) 283--310.}

\bibitem{Justin}{\textsc{J. Cyr, P. Nguyen, S. Tang and R. Temam}, Review of local and global existence results for stochastic PDEs with L\'{e}vy noise, {\it Discrete Contin. Dyn. Syst.}, {\bf 40} (2020) 5639--5710.}

\bibitem{Martin}{\textsc{M. Hairer and J. C. Mattingly}, Ergodicity of the 2D Navier-Stokes equations with degenerate stochastic forcing, {\it Ann. of Math. (2)}, {\bf 164} (2006) 993--1032.}

\bibitem{Da}{\textsc{G. Da Prato and J. Zabczyk}, Stochastic equations in infinite dimensions, Second edition, Encyclopedia of Mathematics and its Applications, vol. 152, Cambridge University Press, Cambridge (2014).}

\bibitem{Cerrai1}{\textsc{S. Cerrai and M. Freidlin}, Averaging principle for a class of stochastic reaction-diffusion equations, {\it Probab. Theory Related Fields}, {\bf 144} (2009) 137--177.}

\bibitem{Gyongy}{\textsc{I. Gy\"{o}ngy and N. Krylov}, Existence of strong solutions for It\^{o}'s stochastic equations via approximations, {\it Probab. Theory Related Fields}, {\bf 105} (1996) 143--158.}

\bibitem{Dror}{\textsc{D. Givon}, Strong convergence rate for two-time-scale jump-diffusion stochastic differential systems, {\it Multiscale Model. Simul.}, {\bf 6} (2007) 577--594.}

\bibitem{Bihari}{\textsc{I. Bihari}, A generalization of a lemma of Bellman and its application to uniqueness problems of differential equations, {\it Acta Math. Acad. Sci. Hungar.}, {\bf 7} (1956) 81--94.}

\bibitem{Gou}{\textsc{Z. Gou, M. Wang and N. Huang}, Strong solutions for jump-type stochastic differential equations with non-Lipschitz coefficients, {\it Stochastics}, {\bf 92} (2020) 533--551.}
	
\end{thebibliography}

\addcontentsline{toc}{chapter}{References} %向目录中添加条目，以章的名义
%% -------------------------------------------------------------------

\end{document}